\documentclass{amsart}
\usepackage{amssymb,latexsym,graphics}
\usepackage[mathscr]{eucal}

\newtheorem{thm}{Theorem}[section]

\newtheorem{lem}[thm]{Lemma}
\newtheorem{prop}[thm]{Proposition}

\theoremstyle{remark}
\newtheorem{rem}[thm]{Remark}

\newtheorem{defn}[thm]{Definition}
\theoremstyle{definition}

\newcommand{\lp}[2]{\Vert \, #1 \, \Vert_{#2}}

\newcommand{\dint}{{\int\!\!\int}}

\newcommand{\ret}{\vspace{.3cm}}

\numberwithin{equation}{section}
\begin{document}
\title[Variable Cubic Klein-Gordon Scattering]{Scattering for the Klein-Gordon Equation with quadratic and variable
coefficient Cubic nonlinearities}
\author{Hans Lindblad}
\address{Department of Mathematics, Johns Hopkins University, 404 Krieger Hall, 3400 N. Charles
Street, Baltimore, MD 21218}
\email{lindblad@math.jhu.edu}
\author{Avy Soffer}
\address{Department of Mathematics, Rutgers University, 110 Frelinghuysen Rd., Piscataway, NJ 08854, USA}
\email{soffer@math.rutgers.edu}
\thanks{}
\subjclass{}
\keywords{nonlinear KG equation,long-range scattering, normal form analysis}
\date{}
\dedicatory{}
\commby{}


\begin{abstract}
We study the 1D  Klein-Gordon equation with variable coefficient cubic
nonlinearity. This problem exhibits a striking resonant
interaction between the spatial frequencies of the nonlinear
coefficients and the temporal oscillations of the solutions. In the
case where the worst of this resonant behavior is absent, we prove
$L^\infty$ scattering as well as a certain kind of strong smoothness
for the solution at time-like infinity with the help of several new
normal-forms transformations. Some explicit examples are also given
which suggest qualitatively different behavior in the case where the
strongest cubic resonances are present.
\end{abstract}

\maketitle

\section{Introduction}
In this paper we study the asymptotic behavior of solutions to nonlinear Klein Gordon
equations in one space dimension, with variable coefficients depending on the point in space.
These kind of equations show up in the stability analysis of stationary solutions
to some equations of mathematical physics.

The most important one dimensional Klein-Gordon type equations
are the well known $\phi^4$ sigma model:
$$
\partial_t^2 \phi - \partial_x^2 \phi -\phi+\phi^3=0
$$
which has the stationary 'kink' solution $\phi_0=\tanh{(x)}$,
and the sine-Gordon equation:
$$
\partial_t^2 \phi - \partial_x^2 \phi+\sin{\phi}=0.
$$
which has the stationary solution $\phi_0=\arctan{(e^x)}$. In both cases the substitution $\phi=\phi_0+u$ gives a nonlinear Klein Gordon equation for $u$ with coefficients depending on the point in space.

Topological solitons, which are static solutions of nonlinear
equations of mathematical physics, play important role in many
models of field theory, statistical models and nonlinear dynamics
\cite{Man-S}.

Topological solitons differ from the standard soliton solutions by
being topologically inequivalent to the vacuum solution, or the
zero/constant solution. As such, they are usually characterized by a
topological "charge", a winding number associated to the solution.
In one dimension such solutions are the kink solutions,
characterized by having a different constant limit as $x$ approaches
$+$ or $-$ infinity.

 In two dimensions topological solitons are
vortices, for example, in three dimensions monopoles, and in 4
dimensions instantons.

An important aspect of the analysis of such solutions, is the
dynamics of a perturbed kink. Since, by topological restriction,
they can not approach a vacuum state, they are stable. But, they can
radiate some field and change their parameters, such as velocity, in
the process.

Under small perturbation, one is led to consider the solution as a
sum of kink plus a remainder:

$$
\phi= K(x,\gamma(t))+u(x,t),
$$
where, to determine the equations of the parameters of the kink,
$\gamma(t)$, as a function of time, one needs to impose further
orthogonality conditions on $u$.

 The resulting equation for $u$ is
easily seen to be a nonlinear Klein-Gordon equation which also
depends on $\dot \gamma(t).$

 The nonlinear terms ,in general will
include terms which are linear, quadratic and cubic, with powers of
$K$ as coefficients.

In one dimension, the quadratic and cubic terms are leading to
serious long range problems. It is well known that in general, they
modify the large time behavior of the solution, and it is not given
by the simple free Klein-Gordon waves [Str, Sha,Del]. The problem also shows up in the Schr\"odinger  Equation, see, e.g [H-Nau].
To
understand the problem, notice, that in one dimension, the solution
of the free KG equation decays, in the $L^{\infty}$ norm, like
$t^{-1/2}.$ Hence the quadratic and cubic terms act like potentials
that decay like $t^{-1/2}$ and $t^{-1},$ respectively. These are not
integrable at infinity, and therefore lead to modified free
dynamics.

The way to deal with these terms is based on normal form
transformation, which is an effective way of using the oscillation
in the solution to integrate by parts, and gain extra decay.

This was first done by Shatah [Sha] for  nonlinear partial
differential equations. See also [F,Kat]. The case of Klein-Gordon
equation with quadratic and cubic terms was treated by Delort[Del].
A different approach was later developed by Lindblad-Soffer
[Lin-S1,Lin-S2], which gave a simplified and detailed asymptotic
expansion of the solution for large times, for both the nonlinear KG
and NLS, with cubic terms.

 All of the above works assumed a constant
coefficient nonlinear (and linear..) terms. The work we present here
is dealing with nonconstant coefficient cubic term, and constant
quadratic term. It turns out that there is a fundamental new
resonance phenomena in this case, between the oscillation in$x$
space of the coefficient function, denoted below by $\beta(x),$ and
the nonlinearity! These resonances may have different
manifestations, and they change the behavior of the solutions. The
treatment for such resonances requires pushing the method of normal
forms in a new direction.

 In this work we then consider the large time Cauchy problem
for the one dimensional Klein-Gordon equation of the following type:

In this work we consider the large time Cauchy problem for
the one dimensional Klein-Gordon equation:
\begin{subequations}\label{basic_EQ}
\begin{align}
    \partial_t^2 u - \partial_x^2 u + u  \ &= \
    \alpha_0 u^2 + (\beta_0+\beta(x)) u^3 \ , \label{basic_EQ1} \\
    u(0,x) \ &= \ u_0(x) \ , \label{basic_EQ2}\\
    \partial_t u(0,x) \ &= \ {u}_1(x) \ , \label{basic_EQ3}
\end{align}
\end{subequations}
for smooth compactly supported initial data $(u_0,\dot{u}_0)$.
Here we assume that the function $\beta(x)$ is a
rapidly decaying smooth function in $x$.

\begin{thm}\label{main_thm}
Let $u(t,x)$ be the global solution to the system \eqref{basic_EQ}
with sufficiently small compactly supported initial data $(u_0,\dot{u}_0)$, in the Sobolev norm $(H^2, H^1)$.
Then for $t\geqslant 0$ the solution $u(t,x)$ obeys the following  $L^\infty$ estimate:
\begin{equation}
    |u(t,x)| \ \lesssim \ \frac{C(u_0,{u}_1)}{(1+ |\rho|)^\frac{1}{2}}
    ,\qquad \rho=|t^2-x^2|^{1/2}
    \ . \label{main_decay_est}
\end{equation}
\end{thm}\ret

We remark that even in the constant coefficient case, when $\beta(x)=\beta_0$ is constant, the asymptotic behavior is nontrivial.
Delort proved that there is a logarithmic phase correction to the asymptotics of the free linear Klein Gordon:
\begin{equation} \label{eq:nonlinearasymptotic}
 v(t,x)\sim \rho^{-1/2} e^{i\phi_0(\rho\!,\,x/\rho)}\, a(x/\rho)+
 \rho^{-1/2} e^{-i\phi_0(\rho\!,\, x/\rho)}\, \overline{a(x/\rho)},
 \end{equation}
 were $a(x/\rho)=\sqrt{1+x^2/\rho^2}\,
\widehat{u}_+(-x/\rho)$ is obtained from the Fourier transform of initial data $\widehat{u}_+(\xi)=(\widehat{u}_0(\xi)-i(|\xi|^2+1)^{-1/2}
\widehat{u}_1(\xi))/2$ and the phase is given by
 \begin{equation}
 \phi_0(\rho,x/\rho)=\rho+\Big(\frac{3}{8}\beta_0+\frac{5}{12}\alpha_0^2\Big) |a(x/\rho)|^2\,\ln{\rho}.
 \end{equation}

\begin{rm}
It should be noted that we only get decay in powers of $\rho$, but not $t$.
Interestingly, it shows that this weaker type of decay is ,in fact sufficient for us to prove global existence.
On the other hand, to recover from our results the free decay estimate, $t^{-1/2}$ for the $L^{\infty}$ norm,
more analysis is required. One has to estimate the invariant Klainerman-Sobolev norms, adapted to one dimensional KG equation. In our case we can only use derivatives of order two. But we do have control of the scaling and boost to order two, in $L^2$.

This issue and further applications of the analysis in this paper, that is the asymptotic scattering and phase behavior will be considered elsewhere.
\end{rm}

To proceed, we first change coordinates to hyperbolic :
\begin{align}
    x \ &= \ \rho\sinh(y) \ , &t \ &= \ \rho\cosh(y) \ . \label{coords}
\end{align}
Then, we extract the leading order behavior at large time $\rho$:
$$v=\rho^\frac{1}{2}u.$$
Then the equation is
\begin{equation}
    \big(\partial_\rho^2
    - \frac{1}{\rho^2}\partial_y^2 + 1 +\frac{1}{4\rho^2}\big)v \ = \ \frac{\alpha_0}{\rho^\frac{1}{2}}v^2
    +   \frac{\beta_0}{\rho}v^3 + \frac{\beta\big(\rho y \big)}{\rho}v^3 +\mathcal{R} \ ,\label{hyperkg}
\end{equation}
with a small remainder
\begin{equation}
\mathcal{R}=\mathcal{R}_\beta v^3, \quad\text{where}\quad
 \mathcal{R}_\beta(\rho,y)=\beta\big(\rho \sinh(y)\big) - \beta(\rho y).
 \label{eq:Remainder}
\end{equation}
Most of the effort in our proof of the decay estimate
\eqref{main_decay_est} centers on obtaining simple pointwise in time
Sobolev bounds for the function $v(\rho,y)$ in the region where
$|y|\lesssim 1$. The main difficulty in closing such estimates stems
from the fact that the quadratic term does not have even close to
integrable decay, and also the fact that the coefficient function $\beta$ on
the RHS of \eqref{hyperkg} is extremely rough in the $y$ variable
for large values of $\rho$. Indeed, one can see that unless this
second factor is constant, the cubic non-linearity has \emph{no
pointwise decay} in $\rho$ once a $\partial_y$ derivative is applied to it,
and it is hopeless to try to apply the energy estimate to expressions of this form. However, with the help of
our rather involved normal-forms transformations applied to the
solution $v$, we will in fact be able to subtract off the leading behavior and
show that one may recover
the \emph{same} bounds one would get by simply differentiating the
equation and then applying the energy estimate.

 In the usual approach, one introduces a near identity
nonlinear change of variables of the function $v$:
$$
v=w_0+w_1
$$
and,
\begin{equation}
    w_1 \ = \ \frac{1}{\rho^\frac{1}{2}} B_1(v,v) + \frac{1}{\rho^\frac{1}{2}}
    B_2(\dot{v},\dot{v}) \ , \label{w1_line}
\end{equation}
where the operators are bilinear $\Psi$DO as defined on line
\eqref{pdo_def}. Here we are using the shorthand $\partial_\rho
u=\dot{u}$, and similarly for $\dot{v}$. Here $w_0$ solves the free Klein-Gordon equation, i.e. \eqref{hyperkg} but with vanishing right hand side.
Using this normal form transformation gets rid of the quadric term in the right
hand side and produces a nonlocal cubic term in its places. Apart from the nonlocal cubic term, which pretty much can be dealt with as the local cubic term $\beta_0\neq 0$, we have reduced to the case $\alpha_0=0$ in the above. Since the term
with $\rho^{-2} v$ in the left of \eqref{hyperkg} clearly decays much more, we will
for simplicity now neglect it and consider only the linear operator
\begin{equation}
\Box_\mathcal{H}+1,\qquad \text{where}\quad
\Box_\mathcal{H}= \partial_\rho^2-\frac{1}{\rho^2}\,\partial_y^2.
\end{equation}

\subsection{A simplified model describing the problem with variable coefficients}
The case of constant coefficient cubic $\beta_0\neq 0$ but $\beta=0$ and
$\alpha_0=0$ is by now much easier to deal with as in \cite{L-S2}.
We will roughly describe the argument there and an important modification that allows us to deal with quintic variable coefficients:
\begin{equation}
    (\partial_\rho^2
    - \frac{1}{\rho^2}\partial_y^2 + 1 )v \ =F,\qquad \text{where}\quad F=
     \frac{\beta_0}{\rho}v^3 + \frac{{\beta}\big(\rho y \big)}{\rho^{3/2}}v^4 \ . \label{eq:quintic}
\end{equation}
The argument below will involve bounds for the $H^1$ norm as well as the $L^\infty$ norm and will be the same norms we will bound for the general case. The argument will also show why we can't directly deal with the variable coefficient cubic.

First we have the energy estimate (obtained by multiplying \eqref{eq:quintic} by $\dot{v}$ and integrating) \begin{equation}
\lp{(v,\dot{v},\rho^{-1}\partial_y v)(\rho)}{H^1}
     \ \lesssim  \ \lp{(v,\dot{v},\rho^{-1}\partial_y v)(1)}{H^1}+
     \int_1^\rho\!\!
     \lp{F(s)}{H^1} \ d s \ , \label{eq:energy0}
\end{equation}
where $F$ is the right hand side of \eqref{eq:quintic}:
\begin{multline}
\|F(\rho)\|_{H^1} \lesssim \frac{1}{\rho} \|v(\rho)\|_{L^\infty}^2 \|v(\rho)\|_{H^1}
+\frac{1}{\rho^{3/2}}\| \rho\beta^\prime(\rho y)\|_{L^2_y} \|v(\rho)\|_{L^\infty}^4\\
\lesssim \frac{1}{\rho} \big(1+\|v(\rho)\|_{L^\infty}\big)^2\,\,
\|v(\rho)\|_{L^\infty}^2 \|v(\rho)\|_{H^1}
\label{inhomH1}
\end{multline}
since $\| \rho\beta^\prime(\rho y)\|_{L^2_y}=\rho^{1/2} \|\beta^\prime\|_{L^2}$.
We want to prove that for sufficiently small initial data  $\lp{(v,\dot{v},\rho^{-1}\partial_y v)(1)}{H^1}\leq \epsilon$
we have a global bound
\begin{equation}
\|v(\rho)\|_{L^\infty}\leq C\epsilon. \label{eq:epsdecay}
\end{equation}
The strategy is to assume that we have this bound and prove that it implies a better bound. It then follows that
\begin{equation}
\lp{(v,\dot{v},\rho^{-1}\partial_y v)(\rho)}{H^1}
     \ \lesssim  \ \lp{(v,\dot{v},\rho^{-1}\partial_y v)(1)}{H^1}+
     C\int_1^\rho\!\!
     \frac{\epsilon^2}{s}\lp{(v,\dot{v},\rho^{-1}\partial_y v)(s)}{H^1} \ d s \ ,
\end{equation}
which gives a bound
 \begin{equation}
 \lp{(v,\dot{v},\rho^{-1}\partial_\rho v)(\rho)}{H^1}\leq C \epsilon \rho^{C\epsilon}.
 \label{eq:h1bound}
 \end{equation}

 We now also need to recover the
$L^\infty$ bound we just used. By Sobolev's lemma \eqref{eq:h1bound} gives a weaker decay estimate \eqref{eq:epsdecay} but with a growing factor $\rho^{C\epsilon}$ which is not sufficient.
If however, the variable coefficient term $\beta $ is not present one can differentiate the equation further and obtain bounds of the form \eqref{eq:h1bound} also for the $H^k$ norms and hence by Sobolev's lemma it gives weaker decay estimates also for higher derivatives, e.g.
$\|\partial_y^2 v(\rho)\|_{L^\infty_y}\leq \|v(\rho)\|_{H^3}$.
 This was used in \cite{L-S2} together with additional estimates obtained by moving the term $\rho^{-1}\partial_y^2 v$ to the right hand side and integrating the left hand side of the equation as an ODE (i.e. multiplying by $\dot{v}$ first):
  \begin{equation}
    (\partial_\rho^2
     + 1 )v -\frac{\beta_0}{\rho}v^3\ =\frac{1}{\rho^2}\partial_y^2 v
      + \frac{{\beta}\big(\rho y \big)}{\rho^{3/2}}v^4  \ , \label{eq:ode}
\end{equation}
which gives the estimate
\begin{equation}
\lp{(v,\dot{v})(\rho)}{L^\infty}
     \ \lesssim  \ \lp{(v,\dot{v})(1)}{L^\infty}+
     \int_1^\rho\!\!
     \lp{G(s)}{L^\infty} \ d s \ , \label{eq:ode0}
\end{equation}
where $G$ is the right hand side of \eqref{eq:ode}:
\begin{equation}
\|G(\rho)\|_{L^\infty}\lesssim \frac{1}{\rho^2} \|\partial_y^2 v(\rho)\|_{L^\infty} +\frac{1}{\rho^{3/2}}\|v(\rho)\|_{L^\infty}^4
\label{eq:inhomode}
\end{equation}
is integrable if $\|\partial_y^2 v(\rho)\|_{L^\infty}\lesssim \epsilon \rho^{a}$ for some $a<1$. This would recover \eqref{eq:epsdecay}.
(We remark that similar ideas were also used in \cite{L1}, \cite{L-R} for the wave equation and is related to the weak null condition.)
However, because of the presence of the variable coefficient term $\beta(\rho y)$ we can not differentiate more than once with respect to $y$, since powers of $\rho$ comes out, so there is no hope of $H^k$ bounds for $k>1$.

The new idea here is to instead introduce a frequency projection onto
low and high frequencies $I=P_{\leq \rho^\sigma}+P_{\geq \rho^\sigma}$.
Let $v_1=P_{\leq \rho^\sigma}v$ and $v_2=P_{\geq \rho^\sigma}v$.
Since $\|P_{\geq \lambda} v\|_{L^\infty}^2\leq \|P_{\geq \lambda} v\|_{L^2}\|P_{\geq \lambda} v\|_{H^1}\leq \lambda^{-1/2} \|v\|_{H^1}^2$
it follows that
\begin{equation}
\| P_{\geq \rho^\sigma} v\|_{L^\infty}\lesssim
\rho^{-\sigma/2} \|v(\rho)\|_{H^1}\lesssim \epsilon \rho^{C\epsilon-\sigma/2}\lesssim \epsilon.\label{highdecay}
\end{equation}
It only remains to bound the low frequencies and for this we use a version of the ODE argument above.
The bound for the low frequency part then follows from projecting the equation \eqref{eq:ode} to
low frequencies:
 \begin{equation}
    (\partial_\rho^2
     + 1 )v_1 -\frac{\beta_0}{\rho}v_1^3\ =\frac{1}{\rho^2}\partial_y^2 v_1
      + \frac{\beta_0}{\rho} (v^3-v_1^3)+ \frac{{\beta}\big(\rho y \big)}{\rho^{3/2}}v^4 -P_{\geq \rho^\sigma}  F \ , \label{eq:lowfreqode}
\end{equation}
where $F$ is as in \eqref{eq:quintic}.
By a similar argument used to prove \eqref{highdecay} we have
\begin{equation}
\|\partial_y^2 v_1(\rho)\|_{L^\infty} \leq C \rho^{3\sigma/2} \|v(\rho)\|_{H^1}\leq C\epsilon \rho^{3\sigma/2 +C\epsilon}
\end{equation}
and
 \begin{equation}
\| v^3-v_1^3\|_{L^\infty}\lesssim
\|v_2\|_{L^\infty}\big(\|v_1\|_{L^\infty}+\|v_2\|_{L^\infty}\big)^2
\lesssim \frac{1}{\rho^{\,\sigma/2}}\|v\|_{H^1} \|v\|_{L^\infty}^2,
 \end{equation}
 since the kernels of the projections is uniformly bounded in $L^1$,
 and
 \begin{equation}
 \|P_{\geq \rho^\sigma} F\|_{L^\infty}\leq \rho^{-\sigma/2} \|F\|_{H^1}
 \end{equation}
which is bounded by \eqref{inhomH1}.
The $L^\infty$ norm of the right hand side \eqref{eq:lowfreqode} is again integrable if $C\epsilon<\sigma<2/3-C\epsilon$.

\subsection{Variable coefficient cubic normal form transformations}
 For the
variable coefficient cubic term $\beta\neq 0$, we have seen that
the above argument doesn't work directly but we need to first remove it with a new
type of variable coefficient normal form construction, given in section 7.
To simply let us just deal with the case
\begin{equation}
    (\partial_\rho^2
    - \frac{1}{\rho^2}\partial_y^2 + 1 )v=
     \frac{\beta\big(\rho y \big)}{\rho}v^3  \label{eq:variablecubic}
\end{equation}
but with using the same kind of norms as in the previous example.
First we note:

The decay estimate in the previous example works just as before,
one just replaces $\beta_0$ with $\beta$ in the left of \eqref{eq:ode}
and estimate it as an ode by multiplying by $\dot{v}$ and integrating the terms
in the left modulo errors that can be estimated, see section \ref{decaysection}.
Hence it is only the energy estimate that involves taking one $y$ derivative of the right hand side of \eqref{eq:variablecubic} in $L^2$ that fails and it is only when that derivative
falls on $\beta(\rho y)$. Moreover it is only the frequencies of $v$ less than
$\leq \rho$ that cause problems since $\|\beta(\rho y)\|_{H^1_y}\sim \rho^{1/2}$
and $\|P_{\geq \rho} v\|_{L^\infty}\leq C\rho^{-1/2} \|v\|_{H^1}$.

We will therefore attempt to find and subtract off a normal form $w_2$ such that
\begin{equation}
    (\partial_\rho^2
    - \frac{1}{\rho^2}\partial_y^2 + 1 )w_2=
     \frac{\beta\big(\rho y \big)}{\rho}(P_{\leq \rho} v)^3
     +{\mathcal E}_{cubic}\label{eq:variablecubicnormalform}
\end{equation}
modulo an error
\begin{equation}
\|\mathcal{E}_{cubic}\|_{H^1}\lesssim
\frac{1}{\rho}\|(v,\dot{v})\|_{L^\infty}^2 \|(v,\dot{v})\|_{H^1}\label{eq:errorest0}
\end{equation}
that can be absorbed in the energy estimate as described in the previous section,
see \eqref{eq:energy0}-\eqref{eq:h1bound}.

 We will attempt a new type of normal form transformation:
\begin{equation}
w_2=\frac{1}{\rho} \sum_{i=0}^3 f_i \, F_i(v_{0},\dot{v}_0),\quad
\text{where}\quad F_i(v,\dot{v})=v^{3-i} \dot{v}^i,\quad v_0=P_{\leq \rho} v,
\label{eq:ansatz}
\end{equation}
where we assume that we can find some functionals $f_i$ of $\beta$:
\begin{equation}
f_i=f_i[\beta]
\end{equation}
behaving like $\beta(\rho y)$:
\begin{equation}
\| \partial_\rho^\ell D_y^k f_i\|_{L^\infty_y}\leq C\quad
\quad
\| \partial_\rho^\ell D_y^k f_i\|_{H^1}\leq C\rho^{1/2},\quad \text{for}\quad
 k,\ell\leq 2\label{eq:ydecay}
\end{equation}
and satisfying some equations to be determined such that \eqref{eq:variablecubicnormalform} hold.
It will be easy to see from this and the arguments below that
\begin{equation}
\|(w_2,\dot{w}_2,D_y w_2)\|_{H^1}\lesssim \rho^{-1/2}
\|(v,\dot{v},D_y v)\|_{L^\infty}^2\|(v,\dot{v},D_y v)\|_{H^1}.
\end{equation}

Note that the operator
\begin{equation}
D_y=\frac{1}{i\rho} \partial_y
\end{equation}
is a bounded operator on $P_{\leq \rho} v$ in $L^2$ and anyway it is
also part of the energy estimate. It follows that as long as \eqref{eq:ydecay} holds
any term with at least one $D_y$ derivative falling on
$F_i$ or $(v_0,\dot{v}_0)$ is easy to control, since if $k\geq 1$
\begin{equation}
\|D_y^k (v_0,\dot{v}_0)\|_{H^1}\lesssim \rho^{-1}\|(v,\dot{v})\|_{H^1},\qquad
\|D_y^k (v_0,\dot{v}_0)\|_{L^\infty}\lesssim \rho^{-1/2}\|(v,\dot{v})\|_{H^1}.
\label{eq:yderdecay}
\end{equation}
Hence only $\rho$ derivatives falling on $F_i$ matter. Also any $\rho$ derivative falling on $\rho^{-1}$ produces a term with additional decay and hence easy to control. We therefore have
\begin{equation}
(\Box_\mathcal{H} + 1 )w_2\sim
    \frac{1}{\rho}\sum_{i=0}^3 \big[(\Box_\mathcal{H} + 1 )f_i \big] F_i +f_i\,\partial_\rho^2 F_i
    +2\partial_\rho f_i \, \partial_\rho F_i,
\end{equation}
modulo terms that decay faster or are bounded in $H^1$ and satisfy \eqref{eq:errorest0}.

Using the equation
\begin{equation}
\ddot{v}_0+v_0=D_y^2 v_0+\frac{1}{\rho} P_{\leq \rho} \big(\beta v^3\big)
+\big[\partial_\rho^2,P_{\leq \rho}\big] v\sim 0\label{eq:kgerrorintr}
\end{equation}
to replace $\ddot{v}_0$ by $-v_0$ we obtain polynomials of degree $3$,
$F_i^0=F_i$ and
\begin{equation}
F_i^1(v,\dot{v})=(\partial_\rho F_i)(v,\dot{v},\ddot{v})\Big|_{\ddot{v}=-v}\!\!\!\!,\quad
F_i^2(v,\dot{v})=(\partial_\rho^2 F_i)(v,\dot{v},\ddot{v},\dddot{v})
\Big|_{\ddot{v}=-v,\dddot{v}=-\dot{v}}
\end{equation}
where
\begin{align}
&F_i^1=(3-i) F_{i+1}-i F_{i-1},\\
&F_i^2=(3-i)(2-i)F_{i+2}+i(i-1)F_{i-2}-\big((3-i)(i+1)+(4-i)i\big)F_i.
\end{align}
The error in these approximations are
\begin{align}
\partial_\rho F_i-F_i^1&=G_i^1\big(\ddot{v}_0+v_0)\\
\partial_\rho^2 F_i-F_i^2&=G_i^2\big(\ddot{v}_0+v_0)+
G_i^3\big(\ddot{v}_0+v_0)^2+
G_i^4\partial_\rho\big(\ddot{v}_0+v_0)
\end{align}
where $G^k_i=G_i^k(v_0,\dot{v}_0)$ are polynomials such that all terms are cubic
and $\ddot{v}_0+v_0$ is given by \eqref{eq:kgerrorintr}.
The first term in the right of \eqref{eq:kgerrorintr}
is easy to control using \eqref{eq:yderdecay}.
The other two terms have more than enough additional decay since
\begin{align}
\|[\partial_\rho,P_{\leq \rho} ]u\|_{L^\infty}
     &\lesssim (1/\rho)\|P_{ \sim \rho} u\|_{L^\infty}\leq
     (1/\rho)\rho^{-1/2} \|u\|_{H^1},\\
    \|[\partial_\rho,P_{\leq \rho} ]u\|_{H^1}
     &\lesssim (1/\rho)\|u\|_{H^1}.
\end{align}

With this in mind we want to solve the system
\begin{equation}
\sum_{i=0}^3 \big[(\Box_\mathcal{H} + 1 )f_i \big] F_i +f_i\, F_i^2
    +2\partial_\rho f_i \,\,  F_i^1
    \sim\beta(\rho y) F_0
\end{equation}
by equating the coefficients of the monomials $F_i$, at least up to terms decaying faster.

Our first attempt is to assume that
\begin{equation}
f_i=f_i(\rho y),\qquad\text{where} \quad |f_i^{(k)}(z)| \leq C(1+|z|)^{-k}
\quad k\leq 2
\end{equation}
in which case \eqref{eq:ydecay} holds and in addition $|\partial_\rho^k f_i|\leq C\rho^{-k}$
so that we also can neglect the terms with $\rho$ derivatives falling on $f_i$.
Then the above system simplifies to
\begin{equation}
\sum_{i=0}^3 \big[(f_i(\rho y)-f_i^{\prime\prime}(\rho y) \big] F_i
+f_i(\rho y)\,F_i^2=\beta(\rho y) F_0.
\end{equation}
This simplifies to $f_1=f_3=0$ and
\begin{align}
f_0^{\prime\prime}(z)+2f_0(z)-2f_2(z)&=-\beta(z)\\
f_2^{\prime\prime}(z)+6f_2(z)-6f_0(z)&=0
\end{align}
or with $g_0=3f_0+f_2$ and $g_2=f_0-f_2$
\begin{equation}
g_0^{\prime\prime}=-3\,\beta,\qquad
g_2^{\prime\prime}+8\,g_2=-\beta.
\end{equation}

We remark that the same asymptotic system shows up in an
asymptotic expansion of a solution of \eqref{eq:variablecubic},
when one attempts to find an approximate solution by an ansatz
of the form \eqref{eq:ansatz} but with $u$ replaced by in first
approximation of the asymptotic expansion for the solution of the linear homogeneous equation \eqref{eq:nonlinearasymptotic}.

Recall that $\beta\in{\mathcal S}$, i.e. the space of
rapidly decaying smooth functions and that the Fourier transform maps ${\mathcal S}$ to itself. Taking the Fourier transform of this system we see that we obtain
a system
\begin{equation}
-\zeta^2 \widehat{g}_0(\zeta)=-3\, \widehat{\beta}(\zeta),\qquad
(-\zeta^2+8)\widehat{g}_2(\zeta)=-\widehat{\beta}(\zeta)
\end{equation}
This system has a solution in ${\mathcal S}$ if and only if $\widehat{\beta}(\zeta)$
has a double zero at $\zeta=0$ and simple zeros and $\zeta=\pm \sqrt{8}$,
in which case the above method works.

\begin{defn}\label{functional}
In case $\widehat{\beta}(\zeta)$
has a double zero at $\zeta=0$ and simple zeros and $\zeta=\pm \sqrt{8}$,
we define the functional
$f_i[\beta]$ to be the fast decaying solution of the above system.
 \end{defn}

 In what follows we may therefore assume that $\widehat{\beta}$
 is supported in a small neighborhood of either $0$ or $\sqrt{8}$ or $-\sqrt{8}$.
If $\widehat{\beta}$ is supported near $0$ then the $f_i$ are actually growing,
although the derivatives are bounded. In this case we modify the approach,
taking into account that the solution decays for large frequencies. Let $f_i[\beta]$, $i=0,2$
be the functionals defined by solving the system above depending on $\beta$,
see Definition \ref{functional}.
Let us make frequency decomposition
\begin{equation}
\beta=\sum_{j=0}^{\infty}\beta_{j},\quad \text{where}\quad
\beta_j=P_{2^{-j}} \beta,\qquad j\geq 0
\end{equation}
is the projection onto a dyadic region of frequencies $\sim 2^{-j}$.
We now define
\begin{equation}
w_{2,j}=\frac{1}{\rho} \sum_{i=0,2} f_i[\beta_j] \, F_i(v_{j},\dot{v}_j),\quad
\text{where},\quad v_j=P_{\leq \rho/2^j} v,\qquad
\widehat{\beta}_j(\xi)=\chi_1(2^j \xi) \widehat{\beta}(\xi),
\end{equation}
where $\chi_1$ is supported in a neighborhood of $1$,
which attempts to solve
\begin{equation}
(\Box_\mathcal{H} + 1 )w_{2,j}\sim \rho^{-1}\beta_j(\rho y) F_0(v_j)
=\rho^{-1}P_{\sim \rho /2^j} \beta (\rho y) \, F_0(P_{\leq \rho/2^j} v).
\label{jerror}
\end{equation}

We will sum only over values of $j$ for which $2^j\leq \rho^{1/2}$,
and form
\begin{equation}
w_2=\sum_{j=1}^{\infty} \chi_0(2^j/\rho^{1/2})\, w_{2,j}
\end{equation}
where $\chi_0\in C_0^\infty$ is $1$ in a neighborhood of the origin.
The remainder
\begin{equation}
\rho^{-1} P_{\leq \rho^{1/2}} \beta(\rho y) \, F_0(v)
\end{equation}
is easy to control since its easy to see that $P_{\leq \rho^{1/2}} \beta(\rho y)
=\int e^{i y\rho \zeta} \chi(\zeta \rho^{1/2}) \hat{\beta}(\zeta)\, d\zeta$ satisfies
\begin{equation}
|\partial_y^k P_{\leq \rho^{1/2}} \beta(\rho y) |\lesssim \rho^{(k+1)/2-1},\quad
\|\partial_y^k P_{\leq \rho^{1/2}} \beta(\rho y) \|_{L^2_y}
\lesssim \rho^{(k+1)/2-3/2},
\end{equation}
which are bounded for $k\leq 1$.

What makes this argument work is that
\begin{equation}
\|\beta_j(\rho y)\|_{L^\infty}\lesssim 1/2^j,\qquad
\|\beta_j(\rho y)\|_{\dot{H}^n}\lesssim (\rho / 2^{j})^{n-1/2}/\,2^j,\quad n=0,1.
\end{equation}
Note that
$\|v_j\|_{L^\infty}\lesssim \|v\|_{L^\infty}$,
since the kernels of the projections are uniformly bounded in $L^1$.
Since
\begin{equation}
\|v-v_j\|_{L^\infty}
\lesssim (\rho 2^{-j})^{-1/2} \|v\|_{\dot{H}^1},
\end{equation}
it follows that
\begin{equation}
\| \beta_j(\rho y) F_0(v_j)-\beta_j(\rho y) F_0(v)\|_{\dot{H}^1}\lesssim
2^{-j} \|v\|_{\dot{H}^1}\|v\|_{L^\infty}^2
\end{equation}
which can be summed over $j\geq 0$ to produce an error of the form
\eqref{eq:errorest0}.

For each fixed frequency the normal form is bounded and so is its error
so it remains to show that constants in the error bounds can be summed up
for $2^j\leq \rho^{1/2}$ to produce an error bound of the same form.
The details will be left to section 7.

If $\widehat{\beta}$ is supported near $\pm \sqrt{8}$ then the best thing we can say is that $f_0(\rho y)$ and $f_2(\rho y)$ solving the above system are only bounded and more importantly their derivatives no longer decay. Therefore we no longer can assume that their derivative with respect to $\rho$ are decaying
when $|y|$ is bounded from below. In our situation, because $\beta$ is fast decaying, this may be overcome by multiplying by a cutoff $\chi(\rho^{a} y)$ for some $0<a<1$, where $\chi$ is $1$ in a neighborhood of the origin. We will however instead take a different approach and obtain a new more general variable coefficient normal form transformation that is a better approximation as long as $|y|$ is bounded from above. This is obtained by taking into account the $\rho$ derivatives of the system to obtain:
\begin{align}
        \Box_\mathcal{H} f_0 - 2f_0 -2\partial_\rho f_1 + 2f_2 \ &= \
        \beta(\rho y) \ , \notag\\
        \Box_\mathcal{H} f_1 - 6f_1 + 6\partial_\rho f_0
        -4\partial_\rho f_2 + 6 f_3 \ &= \ 0 \ , \notag\\
        \Box_\mathcal{H} f_2 - 6f_2 + 6f_0 +
        4\partial_\rho f_1 - 6\partial_\rho f_3 \ &= \ 0 \ , \notag\\
        \Box_\mathcal{H} f_3 - 2f_3 + 2 f_1 + 2\partial_\rho f_2 \ &= \
        0 \ . \notag
\end{align}
As before introducing
with $g_0=3f_0+f_2$, $g_2=f_0-f_2$, $g_1=f_1+3f_3$, $g_3=f_1-f_3$:
we get
\begin{align}
        \Box_\mathcal{H} g_0 -2\partial_\rho g_1 \ &= \ 3\beta  \ , \notag\\
        \Box_\mathcal{H} g_1 + 2\partial_\rho g_0 \ &= \ 0 \ ,
        \notag\\
        \Box_\mathcal{H} g_2 - 8g_2 - 6\partial_\rho g_3 \ &= \ \beta \ , \notag\\
        \Box_\mathcal{H} g_3 - 8g_3 + 6\partial_\rho g_2 \ &= \ 0 \ . \notag
\end{align}

Complexifying the above system we get
$$
K_1=(g_0+ig_1)e^{i\rho}/3,\qquad K_3=(g_2+ig_3)e^{3i\rho}
$$
gives the system
\begin{align}
        (\Box_\mathcal{H} + 1) K_1 \ &= \  e^{i\rho}\beta(\rho y) \ ,
        &(\hbox{$0$ resonance equation})
        \ , \label{0_res_eq 21}\\
        (\Box_\mathcal{H} + 1) K_3 \ &= \  e^{3i\rho}\beta(\rho y) \
        ,
        &(\hbox{$\pm\sqrt{8}$ resonance equation})
        \ . \label{sqrt8_res_eq 21}
\end{align}
We note that we only have to solve these equations asymptotically,
which can be done with the stationary phase method.
We hence want to an asymptotic solution $K_n[\beta]$ that solves
\begin{equation}
(\Box_\mathcal{H} + 1) K_n=e^{in\rho}\beta(x)+{\mathcal E}_{K_n}
\end{equation}
where the error ${\mathcal E}_n[\beta]$ decays sufficiently fast.
Here the functional $K_n[\beta]$ and error ${\mathcal E}_n[\beta]$
satisfy the same kind of estimates
as we had for the functionals $f_i[\beta]$ before, i.e. if $\beta$ is smooth
and fast decaying we have
\begin{align}
    |\partial_\rho^k D_y^l\chi K_i| \ &\leqslant \ C_{k,l} \ ,
     &|\partial_\rho^k D_y^l\chi \mathcal{E}_{K_i}| \
     &\leqslant \ \rho^{-1}C_{k,l} \ , \label{med_K_ests11}\\
     \lp{(\chi K_i,D_y\chi K_i,\partial_\rho \chi K_i }{B_\rho^\infty} \ &\lesssim \ 1 \ . \label{med_K_ests21}
\end{align}
Here $\chi=\chi(y)$ is a smooth bump function in the $y$ variable.

\subsection{Proof of the main theorem in bootstrap form}

\begin{thm}(Main Energy bound in Bootstrapping Form)
Suppose that the function $v$ solves the equation \eqref{hyperkg}
with compactly supported initial data $(v_0,\dot{v}_0)$ at $\rho=1$.
Then there exists universal constants $C_0$, $\epsilon_0>0$,
$0<\delta<1/8$ such that if $\epsilon<\epsilon_0$ and
\begin{equation}
     \|(v,\dot{v},\partial_y v)(1)\|_{H^1}\leq \epsilon \ ,
     \label{initial_norm}
\end{equation}
and one also assumes the time dependent bounds:
\begin{align}
    \lp{(v,\dot{v},\rho^{-1}\partial_y v)(\rho)}{H^1}\ &\leqslant \ 2C_0\epsilon\rho^\delta \ ,
    \qquad \text{for}\quad 1\leq \rho\leq T \label{boot2}
\end{align}
then one also has the following improved time dependent bounds:
\begin{align}
    \lp{(v,\dot{v},\rho^{-1}\partial_y v)(\rho)}{H^1}\ &\leqslant \ C_0\epsilon\rho^\delta \ ,
    \qquad\text{for}\quad  1\leq \rho\leq T \label{better_boot2}
\end{align}
\end{thm}

We now define a norm that is just slightly stronger than $L^\infty$ for high frequencies:
\begin{equation}
   \,\,\lp{v}{B_\rho^{\infty}} =
    \lp{P_{\leqslant \rho }v}{L^\infty} +
    \sum_{\lambda\geq \rho}
    \ln{|\lambda/\rho|}\,
    \lp{P_\lambda v}{L^\infty}.\!\!\! \label{B_norm_def}
\end{equation}
where the sum is over dyadic $\lambda =2^j$. We prove in section $L^\infty$ bound follows from the
$H^1$ bound in section \ref{decaysection}:
\begin{prop} Suppose that for some $0<\delta<1/8$,
\begin{equation}
K=\sup_{1\leq \rho\leq T} \rho^{-\delta} \lp{(v,\dot{v},\rho^{-1}\partial_y v)(\rho)}{H^1}
<\infty .
\end{equation}
Then
\begin{equation}
\sup_{1\leq \rho\leq T} \lp{(v,\dot{v},\rho^{-1/2-\delta}\partial_y v)(\rho)}{B_{\rho}^\infty}
\leq C K(1+K^2).
\end{equation}
\end{prop}
Theorem \ref{main_thm} follows from this.

It is clear that for times $\rho=1+ O(1)$ one has the bound \eqref{better_boot2}
from standard local existence theory. Furthermore, it is also clear that
given the bound \eqref{better_boot2} up to time $T$, we
may extend the solution to a later time $T+O(1)$ such that the weaker bound
\eqref{boot2} hold. Therefore $T$ is not the maximal time for which \eqref{boot2}
hold and we can hence conclude that the estimate hold for $T=\infty$ so we have a
global bound.
Therefore, the  remainder of the paper will be devoted to recovering the estimate
\eqref{better_boot2} assuming that one has \eqref{initial_norm}
as well as the estimates \eqref{boot2}.

Because the $B^\infty_\rho$ follows from the $H^1$ it only remains to prove
the $H^1$ estimate assuming the $B^\infty_\rho$. In proving the $H^1$ estimate
for $v$ we will first subtract off a quadratic normal forms $w_1$, and a cubic normal form
$w_2$, depending
on the nonlinearities, i.e. on the coefficients and on the solution
$v$ and its $\rho$ derivative $\dot{v}$:
\begin{equation}
w_1=B\big( (v,\dot{v}),(v,\dot{v})\big)
\end{equation}
where $B$ is a bilinear (non-local) operator given by \eqref{w1} and
\begin{equation}
w_2=T\big( (w,\dot{w}),(w,\dot{w}),(w,\dot{w})\big),\quad
\text{where} \quad w=v-w_1
\end{equation}
and $T$ is a trilinear (non-local) operator given in section \ref{cubicnormalforms}.
After subtracting these normal forms we are left with a remainder
\begin{equation}
w_0=v- w_1-w_2,
\end{equation}
satisfying
\begin{equation}
\Big(\Box_\mathcal{H}+1\Big) w_0=
{\mathcal{E}},
\end{equation}
where the error ${\mathcal E}$ is controllable by the estimates in section
\ref{cubicnormalforms}:
\begin{prop} For some $N$ we have
\begin{equation}
\|\mathcal{E}\|_{H^1}\leq
\frac{C}{\rho}
\big(1+\|(v,\dot{v})\|_{B_\rho^\infty}\big)^N
\|(v,\dot{v})\|_{B_\rho^\infty}^2\|(v,\dot{v})\|_{H^1}.
\end{equation}
and with $D_y=-i\rho^{-1}\partial_y$
\begin{equation}
\sum_{i=1}^2 \lp{(w_i,\dot{w}_i,D_y w_i)(\rho)}{H^1}\leq \frac{C}{\rho^{1/4}}\,
\big(1+\|(v,\dot{v})\|_{B_\rho^\infty}\big)^N \|(v,\dot{v})\|_{B^\infty_\rho}
\|(v,\dot{v},D_y v)\|_{H^1}.
\end{equation}
\end{prop}

The energy estimate, see section 2, for $w_0$ gives
\begin{equation}
\sup_{1\leq \rho\leq T} \lp{(w_0,\dot{w_0},D_y w_0)(\rho)}{H^1}\leq \lp{(w_0,\dot{w_0},D_y w_0)(1)}{H^1}+\int_1^T
\|{\mathcal E}\|_{H^1} \, d\rho.
\end{equation}

Now, assuming \eqref{boot2} it follows that
\begin{equation}
\sup_{1\leq \rho\leq T} \lp{(v,\dot{v},D_y v)(\rho)}{B_{\rho}^\infty} \leq C_1 \, \epsilon
\end{equation}
and
\begin{equation}
\sum_{i=1}^2 \lp{(w_i,\dot{w}_i,D_y w_i)(\rho)}{H^1}\leq \frac{C_2\epsilon }{\rho^{1/4}}\,
\|(v,\dot{v},D_y v)\|_{H^1}
\end{equation}
and therefore since $w_0=v-w_1-w_2$:
\begin{equation}
\|\mathcal{E}\|_{H^1}\leq
\frac{C_3\, \epsilon^2}{\rho}\|(v,\dot{v})\|_{H^1}
\leq
\frac{C_4\, \epsilon^2}{\rho}\|(w_0,\dot{w}_0)\|_{H^1}.
\end{equation}
Hence
\begin{multline}
\sup_{1\leq \rho\leq T} \lp{(w_0,\dot{w_0},D_y w_0)(\rho)}{H^1}\leq \lp{(w_0,\dot{w_0},D_y  w_0)(1)}{H^1}\\+\int_1^T
\frac{C_4\, \epsilon^2}{\rho}\|(w_0,\dot{w}_0,D_y w_0)\|_{H^1} \, d\rho.
\end{multline}
It finally follow from this that
\begin{equation}
\sup_{1\leq \rho\leq T} \lp{(w_0,\dot{w_0},\rho^{-1}\partial_y w_0)(\rho)}{H^1}\leq \rho^{C_4\epsilon^2 }\lp{(w_0,\dot{w_0},\partial_y  w_0)(1)}{H^1}.
\end{equation}
Since also
\begin{equation}
\lp{(w_0,\dot{w_0},\partial_y  w_0)(1)}{H^1}\leq
(1+C\epsilon)\lp{(v,\dot{v},\partial_y  v)(1)}{H^1}\leq \frac{5}{4}\epsilon,
\end{equation}
if $\epsilon$ is sufficiently small by \eqref{initial_norm}, it follows that
\begin{equation}
\lp{(v,\dot{v},D_y  v)(1)}{H^1}\leq
\frac{5}{4} \lp{(w_0,\dot{w}_0,D_y  w_0)(1)}{H^1}\leq \Big(\frac{5}{4}\Big)^2 \epsilon\rho^{C_4\epsilon^2 }.
\end{equation}
If $\epsilon$ is so small that $C_4\epsilon^2\leq \delta$ this proves
\eqref{better_boot2} and concludes the proof of the theorem.


\section{Hyperbolic Coordinates and Energy Estimates}
In this section, we set up a convenient set of coordinates
for proving decay estimates of the form \eqref{main_decay_est}. Because we are
assuming that the initial data \eqref{basic_EQ2}--\eqref{basic_EQ3} for our problem
is compactly supported, by time translating the Cauchy problem forward by an $O(1)$
amount and using finite speed of propagation, it suffices to do all of our analysis
in the interior of the froward cone $|x|\leqslant t$. That is, without loss
of generality we may assume that the initial data \eqref{basic_EQ2}--\eqref{basic_EQ3}
is defined at $t=1$ and its support is contained in the set $|x|\leqslant \frac{1}{2}$.\\

\subsection{Hyperbolic coordinates}
As is well known, this forward region is foliated by hyperboloids
$const=\rho=\sqrt{t^2-|x|^2}$, and one may introduce the following
set of hyperbolic coordinates:
\begin{align}
    x \ &= \ \rho\sinh(y) \ , &t \ &= \ \rho\cosh(y) \ .
\end{align}
By the time shifting setup of the last paragraph and finite speed of propagation,
we have   that the support
of our solution $u(t,x)$ is contained in the region $1\leqslant \rho < \infty$.
Our first goal is to transfer the Cauchy problem \eqref{basic_EQ}
in this new system of coordinates.\\

The formulas for the coordinate derivatives of \eqref{coords} are as follows:
\begin{align}
    \partial_y \ &= \ t\partial_x + x\partial_t \ ,
    &\partial_\rho \ &= \ \frac{1}{\sqrt{t^2-|x|^2}}
    (t\partial_t + x\partial_x) \ . \notag
\end{align}
In particular, notice that the derivative $\partial_y$ is nothing
other than the 1D Lorentz boost, which commutes with the linear part
of  equation \eqref{basic_EQ1}. It will be estimates involving this
weighted vector-field that are ultimately
responsible for the bulk of our proof of the decay estimate \eqref{main_decay_est}.\\

In the coordinates \eqref{coords} the Minkowski metric takes the following simple
form:
\begin{equation}
    -dt^2 + dx^2 \ = \ - d\rho^2 + \rho^2  dy^2 \ . \notag
\end{equation}
Therefore, the linear Klein-Gordon operator takes the form:
\begin{equation}
    \partial_t^2 - \partial_x^2 + 1  \ = \
    \partial_\rho^2 + \frac{1}{\rho}\partial_\rho - \frac{1}{\rho^2}\partial_y^2
    + 1 \ . \label{eqs}
\end{equation}
The RHS of this last formula is further simplified via conjugation by
the weight $\rho^\frac{1}{2}$. That is one has the identity:
\begin{equation}
    \rho^\frac{1}{2}\left[\partial_\rho^2 + \frac{1}{\rho}\partial_\rho - \frac{1}{\rho^2}\partial_y^2 + 1\right]
    ( u ) \ = \  \left[\partial_\rho^2
    - \frac{1}{\rho^2}\partial_y^2 + \frac{1/4}{\rho^2} + 1 \right]
    (\rho^\frac{1}{2} u ) \ . \label{conj}
\end{equation}
Therefore, introducing the new quantity $v=\rho^\frac{1}{2}u$ the equation
\eqref{basic_EQ1} becomes:
\begin{equation}
    \Big(\partial_\rho^2
    - \frac{1}{\rho^2}\partial_y^2 + 1+\frac{1/4}{\rho^2}  \Big)v \ = \ \frac{\alpha_0}{\rho^\frac{1}{2}}v^2
    +   \frac{\beta\big(\rho \sinh(y)\big)}{\rho}v^3 +   \frac{\beta_0}{\rho}v^3\ . \notag
\end{equation}
It will  be convenient for us to further simplify the expression for the nonlinear
potential $\beta$. Using the fact that this is a smooth function, we may write:
\begin{equation}
    \beta\big(\rho \sinh(y)\big) - \beta(\rho y) \ = \ \mathcal{R}_\beta(\rho,y) \ .
    \label{R_beta_def}
\end{equation}
Therefore, we may  write the equation for $v$ in the form:
\begin{equation}
    \Big(\partial_\rho^2
    - \frac{1}{\rho^2}\partial_y^2 + \big(1+\frac{1}{4\rho^2}\big)\Big)v \ = \ \frac{\alpha_0}{\rho^\frac{1}{2}}v^2
    +   \frac{\beta(\rho y )}{\rho}v^3 +   \frac{\beta_0}{\rho}v^3+ \frac{1}{\rho}\mathcal{R}_\beta(\rho,y)v^3 \ . \label{main_eq}
\end{equation}
\begin{rem} In the sequel, the reader is encouraged to envision the original
equation \eqref{basic_EQ1} as transforming in under the change of
coordinates \eqref{coords} into equation \eqref{main_eq} without the
remainder term ${\rho^{-1}}\mathcal{R}_\beta(\rho,y)v^3$. Estimates involving  this remainder
will always have a lot of room in the norms we are using. In fact, by the lemma below it decays
two powers $\rho^{-2}$ faster than $\rho^{-1}\beta(\rho y) v^3$. Similarly the term
$v/(4\rho^2)$ will decay two powers $\rho^{-2}$ faster than the term $v$. The reader should therefore just keep in mind the simplified equation
\begin{equation}
    \Big(\partial_\rho^2
    - \frac{1}{\rho^2}\partial_y^2 + 1\Big)v \ = \ \frac{\alpha_0}{\rho^\frac{1}{2}}v^2
    +   \frac{\beta(\rho y )}{\rho}v^3  +   \frac{\beta_0}{\rho}v^3 \ ,
\end{equation}
and at some places we will just prove the estimates for this simplified equation since its
easy to estimate the additional terms.
\end{rem}
\begin{lem}\label{lem:Rbeta} For $0\leq m\leq 1$
\begin{equation}
|\partial_y^k \
\partial_\rho^m\big( \beta(\rho\sinh{(y)})-\beta(\rho y)\big)|\leq
\frac{ C_k  \rho^{k}}{\rho^{2+m}}\sum_{\ell\leq k+1+m}\sup_{|z|\geq \rho |y|}
|\beta^{(\ell)}(z)|(1+|z|)^{3+m}
\end{equation}
\end{lem}
\begin{proof} First note that the estimates for $m=1$ follows from the estimate for $m=0$ applied to $z\beta^\prime(z)$ so we may assume that $m=0$.
If $|y|\geq 1$ then each of the terms separately
is bounded by the right hand side (In fact with $2$ replaced by any
integer.) For the second term this just follows from that $\rho
|y|\geq \rho$ then. For the first term one also have to use that
when $|y|\geq 1$, $e^{|y|}/4\leq |\sinh{(y)}|\leq |\cosh{(y)}|\leq
e^{|y|}$.

For $0\leq y\leq 1$ we write $\sinh{(y)}=y+y^3 b(y)$ for some smooth
function $b(y)$. With $x=\rho y$ and $z=x+\rho^{-2}x^3 b(x/\rho)t $
we have
\begin{equation}
\beta(\rho\sinh{(y)})-\beta(\rho y)=\int_{\rho y}^{\rho \sinh{(y)}}
\beta^\prime(z)\, dz=\rho^{-2} x^3 b(x/\rho)\int_0^{1}
\beta^\prime\big((x+\rho^{-2} x^3 b(x/\rho)t\big)\, dt
\end{equation}
which is bounded by $C \rho^{-2} x^3 \sup_{z\geq x}
|\beta^\prime(z)|\leq C \rho^{-2} \sup_{z\geq x}
|\beta^\prime(z)|z^3 $.  This proves the lemma for $k=0$. The lemma
for $k\geq 1$ follows since
\begin{equation}
\big|\partial_x^k \big( \rho^{-2} x^3 b(x/\rho)\big)\big|\leq
C_k,\qquad k\geq 1, \quad |x|\leq \rho.
\end{equation}
\end{proof}

\subsection{Energy Estimates}
Our next order of business is to record how bounds in terms of the
energy norm may be recovered through Duhamel's
principle for the linear operator on the LHS of equation
\eqref{main_eq}.

\begin{lem}(Duhamel Estimate)
Suppose that $w$ solves the equation inside of the time slab
$[1,T]\times\mathbb{R}$:
\begin{align}
    \Big(\partial_\rho^2
    - \frac{1}{\rho^2}\partial_y^2 + \big(1+\frac{1}{4\rho^2}\big)\Big)w \ &= \ F \ . \label{generic_EQsec2}
\end{align}
Then for $m+n\leq 1$ and $k\geq 0$ one has the bounds:
\begin{multline}
     \sup_{1\leqslant \rho\leqslant T}\ \lp{(\rho^{-1}\partial_y)^m\partial_\rho^n(w,\dot{w},\rho^{-1}\partial_\rho w)(\rho)}{H^k}
     \ \lesssim  \ \lp{(\rho^{-1}\partial_y)^m\partial_\rho^n(w,\dot{w},\rho^{-1}\partial_y w)(1)}{H^k}\\ +
     \int_1^T\
     \lp{ (\rho^{-1}\partial_y)^m\partial_\rho^{n} F(\rho)}{H^k} \ d\rho \ . \label{D2sec2}
\end{multline}
\end{lem}

\begin{proof}
The energy estimate \eqref{D2sec2} is essentially standard. Multiplying the
equation \eqref{generic_EQsec2} by $\partial_\rho w$ and integrating by
parts on all  possible time slabs inside of $\mathbb{R}\times[1,S]$
and using Young's inequality on the product $\lp{
F(\rho)}{L^2}\cdot\lp{\partial_\rho w}{L^2}$ to absorb the second
factor to the LHS we have:
\begin{multline}
        \sup_{1\leqslant \rho\leqslant S}\ \lp{(w,\dot{w},\rho^{-1}\partial_y
        w)(\rho)}{L^2}^2
    + \int_1^S \|\rho^{-1}\partial_y w(\rho)\|_{L^2}^2\frac{d\rho}{\rho}\\
     \ \lesssim  \ \lp{(w,\dot{w},\partial_y w)(1)}{L^2}^2 +
        \left(\int_1^S\
     \lp{  F(\rho)}{L^2} \ d\rho\right)^2 \ . \notag
\end{multline}
We will not use space-time integrals like the one on the LHS in this
work, so we discard it in estimate \eqref{D2sec2}. The $\partial_y v$
differentiated terms in the estimate \eqref{D2sec2} are likewise bounded
by integrating the $\partial_y$ derivative of equation
\eqref{generic_EQsec2} multiplied by the quantity $\partial_\rho
\partial_y w$. To get estimate for an additional derivatives we
first apply $\rho^{-1}\partial_y$ to the equation and then multiply
with $\partial_\rho (\rho^{-1}\partial_y w)$ to get an additional
commutator term in the left hand side
$$
\rho^{-2}\partial_\rho\partial_y w\,
\partial_\rho(\rho^{-1}\partial_y w)=\rho^{-1}\big(\partial_\rho(\rho^{-1}\partial_y
w)\big)^2-\rho^{-5/2}\partial_y w\,
\rho^{-1/2}\partial_\rho(\rho^{-1}\partial_y w)
$$
where the first term gives a positive spacetime contribution to the
energy and the second can be controlled by the first and the space
time part and the lower parts of the energy using Cauchy Schwarz.
Similarly applying $\partial_\rho$ to the equation and multiplying
with $\partial_\rho^2 w$ gives an additional commutator term in the
left
$$
2\rho^{-3}\partial_y^2 w\, \partial_\rho^2 w=2\rho^{-5}\partial_y^2
w\, \partial_y^2 w -2\rho^{-3}\partial_y^2 w\,\, w
$$
where as before the first term gives a positive space time
contribution to the energy and the last term after integrating parts
gives a positive contribution as well.
\end{proof}\ret

Finally, we wish to formalize the relationship between the Duhamel
estimate \eqref{D2sec2} and the and the energy norm
we are working with. In our work, we will need to let this norm grow
at a slow rate due to long range corrections. The space
which keeps track of this is the following:\\

\begin{defn}
We define the \emph{source term} space $\mathcal{S}_\delta[1,T]$ of
index $\delta$ up to time $T$ to be completion of test functions
under the norm:
\begin{equation}
    \lp{F}{\mathcal{S}_\delta[1,T]} \ = \
     \sup_{1\leqslant  \rho\leqslant T}\ \rho^{1-\delta}\lp{F(\rho)}{H^1} \ . \label{S_def}
\end{equation}
\begin{equation}
 \lp{\rho^\delta v}{L^\infty([1,T];H^1)} \ = \
     \sup_{1\leqslant  \rho\leqslant T}\ \rho^{\delta}\lp{v(\rho)}{H^1} \ . \label{S_def1}
\end{equation}
\end{defn}
Based on this definition and the estimate \eqref{D2sec2}, the following is immediate:\\

\begin{lem}(Duhamel Estimates from Source Term Spaces)
Let $w$ solve the equation \eqref{generic_EQsec2}. Then one has the
following uniform bounds for any time slab $[1,T]$:
\begin{equation}
    \sup_{1\leqslant \rho\leqslant T}\
      \rho^{-\delta}\lp{(w,\dot{w},\rho^{-1}\partial_y w)(\rho)}{H^1}\
     \lesssim \ \delta^{-1}\lp{F}{\mathcal{S}_\delta[1,T]} \ . \label{S_est}
\end{equation}
\end{lem}\ret

We remark that there is lowest order energy estimate also for the
nonlinear equation which shows that the $L^2$ norm is bounded.
However, this is no longer the case for the differentiated equation
and higher energies will grow. The nonlinear terms will be too large
to treat as a right hand side in the energy estimate, but we will
first remove them by normal form transformations and after that
apply the energy estimate to the resulting equation.

We also have some nonlinear energy estimates:

\begin{lem}(Nonlinear energy estimates)
Suppose that $v$ solves the equation inside of the time slab
$[1,T]\times\mathbb{R}$:
\begin{equation}
    \big(\partial_\rho^2
    - \frac{1}{\rho^2}\partial_y^2 + \big(1+\frac{1}{4\rho^2}\big)\big) v +\frac{\alpha_0}{\rho^{1/2}} v^2  +\frac{\beta(\rho y)}{\rho} v^3=0\ .
\end{equation}
Suppose also that for $1\leq \rho\leq T$
\begin{equation}
\|(v,\dot{v})(\rho,\cdot))\|_{L^\infty}\leq K\leq \Big(1+16\big(|\alpha_0|+\sup_z {\big[|\beta(z)|+|z|\, |\beta^{\,\prime}(z)|\big]^{1/2}}\big)\Big)^{-1}.
\end{equation}
Then for $1\leq \rho\leq T$ we have
\begin{equation}
     \sup_{1\leqslant \rho\leqslant T}\ \lp{(v,\dot{v},\rho^{-1}\partial_\rho v)(\rho)}{L^2}
     \ \lesssim  \lp{(v,\dot{v},\rho^{-1}\partial_y v)(1)}{L^2}
\end{equation}
\end{lem}
\begin{proof} Let
\begin{equation}
E_0(\rho)=\int \frac{1}{2} v_\rho^2+\frac{1}{2\rho^2} v_y^2 +\frac{1}{2}\big(1+\frac{1}{4\rho^2}\big)v^2 +\frac{\alpha_0}{3\rho^{1/2}} v^3+
\frac{\beta}{4\rho} v^4\,d y
\end{equation}
We note that by assumption the last two terms are bounded by the third so $E_0$ is equivalent
to the norm $\lp{(v,\dot{v},\rho^{-1}\partial_\rho v)(\rho)}{L^2}$, and in particular positive.
 After integrating one term by parts in $y$ and using the equation we get
 \begin{equation}
 E_0^{\,\prime}(\rho)=\int -\frac{1}{\rho^3}(v_y^2 +v^2/4)
 -\frac{\alpha_0}{6\rho^{3/2}} v^3-\frac{\beta_1}{4\rho^2} v^4\, dy ,
 \end{equation}
 where
 \begin{equation}
 \beta_1(z)=\beta(z)-z\beta^\prime(z)
 \end{equation}
 Hence
 \begin{equation}
 E_0^{\,\prime}(\rho)\leq \Big(
 \frac{K}{6\rho^{3/2}} +\frac{K^2}{\rho^2}\Big) E_0(\rho)
 \end{equation}
 from which the bound $E_0(\rho)\lesssim E_0(1)$ follows.
\end{proof}


\section{Littlewood-Paley Theory, Function Spaces}

We begin this section with some standard notation. \\


\subsection{Littlewood-Paley Setup}
Our next task is to set up the standard Littlewood-Paley theory in
the $y$ coordinate. First, we introduce the spatial Fourier
transform:
\begin{equation}
    \widehat{v}(\xi) \ = \ \int_\mathbb{R}\ e^{- i \xi y }v(y)\ dy \ . \label{FT}
\end{equation}
In this normalization, the Plancherel theorem reads
$\lp{v}{L^2(dy)}=(2\pi)^\frac{1}{2}\lp{\widehat{v}}{L^2(d\xi)}$.
We let $p_\lambda$ be a dyadically indexed partition of unity in the frequency variable
such that $p_1\equiv 1$ in a neighborhood of $\xi=0$, and such that for
$\lambda\geq 1$:
\begin{equation}
    p_\lambda(\xi) \ = \ p_0(\lambda^{-1} \xi) \ , \notag
\end{equation}
for some basic annular cutoff $p_0$. We will not fix once and for all this
collection of cutoffs $p_\lambda$, but rather we shall let them vary from line
to line in the sequel. This will relieve us from having to consecutively
label sequences cutoffs whose supports may increase in the various proofs which follow.
As usual, we denote by $P_\lambda$ the corresponding convolution operator
in the physical variable $y$.\\

It will often be convenient for us to combine these dyadic cutoffs in various ways.
We use a standard notation for this:
\begin{align}
    P_{\lesssim \lambda} \ &= \ \sum_{\sigma:\ \sigma \lesssim \lambda} P_\sigma
    &P_{\gtrsim \lambda} \ &= \ \sum_{\sigma:\ \lambda \lesssim \sigma} P_\sigma \ . \notag
\end{align}
and so on depending on the context. In such expressions as the
above, is always assumed that
 we are summing over consecutive dyadic indices. We \emph{always} assume that
 we are summing over dyadic values larger than one, unless otherwise specified.  The reader should keep in
 mind that problem at hand is not scale invariant. \\

Notice that all of the operators, $P_\lambda$, $P_{\lesssim \lambda}$,
and $P_{\gtrsim \lambda}$ are given by $L^1$ convolution kernels with uniform
bounds. Therefore their action is uniformly bounded on all $L^p$ spaces including
$L^1$ and $L^\infty$.\\

Associated with cutoffs of the from $P_\lambda$, one
has the \emph{Bernstein's inequality}:
\begin{equation}
    \lp{P_\lambda v }{L^\infty} \ \lesssim \
    \lambda^\frac{1}{2}\lp{P_\lambda v}{L^2} \ . \label{B_eq}
\end{equation}
This may be rewritten as the local Sobolev type bound:
\begin{equation}
    \lp{P_\lambda v}{L^\infty} \ \lesssim \
    \lambda^{-\frac{1}{2}} \lp{P_\lambda \partial_y v}{L^2} \ . \label{loc_sob}
\end{equation}
These two estimates will be used many times in the sequel.\\


\subsection{Function Spaces}
In this paper, most of our work will be to bound an weighted energy norm of the solution
$v$. One may introduce the usual $H^1$ inhomogeneous Sobolev space of functions in the $y$ variable:
\begin{equation}
    \lp{v}{H^1} = \lp{v}{L^2} + \lp{\partial_y v}{L^2} \approx
    \left(\sum_\lambda \lambda^2 \lp{P_\lambda v}{L^2}^2 \right)^\frac{1}{2}
      \approx
    \left(\sum_\lambda \lp{P_\lambda v}{H^1}^2 \right)^\frac{1}{2}.\label{energy_norm}
\end{equation}
It will also be necessary for us to have a version of $L^\infty$ that is strong enough
to obey estimates for a certain brand of singular integral operators that are common
in our work. This norm must also be weak enough that it will allow us to recapture it
on certain frequency ranges using \emph{non-dyadic} decompositions. Therefore, we construct
the following time dependent hybrid $L^\infty$ space:
\begin{equation}
   \,\,\lp{v}{B_\rho^{\infty}} =
    \lp{P_{\leqslant \rho }v}{L^\infty} +
    \sum_{\lambda\geq \rho}
    \ln{|\lambda/\rho|}\,
    \lp{P_\lambda v}{L^\infty}.\!\!\! \label{B_norm_def1}
\end{equation}\ret
Note that by \emph{Bernstein's inequality} and Cauchy Schwarz $\|v\|_{B_\rho^\infty}\lesssim \|v\|_{H^1}$.

In the sequel, we  shall also use the notation
$\lp{(v,\dot{v},\rho^{-1}\partial_y v)}{H^1}$, etc.
to denote that norm applied to the triple $(v,\partial_\rho v,\rho^{-1}\partial_y v)$ as a direct sum.\\

We now record a few simple estimates involving the relationship between
these $L^\infty$ type spaces and the energy type norms defined above, as well as an
algebra estimate:

\begin{lem}[Time Dependent $L^\infty$ Type Bounds]
For test functions $v$ one has the following uniform bounds (e.g.
uniform in $0\leqslant \sigma$):
\begin{align}
\|v\|_{L^\infty} &\lesssim \| v\|_{H^1}^{1/2}\,  \|v\|_{L^2}^{1/2}\label{LinfH1sob00}\\
\lp{P_{\geq \rho^\sigma} v}
    {L^\infty} \ &\lesssim \ \rho^{-\sigma/2} \lp{v}{H^1}
    \ , \label{LinfH1sob0}\\
    \lp{P_{ \lesssim\rho^{\sigma}} \partial_y^k v}
    {L^\infty} \ &\lesssim \ \rho^{(k-1/2)\sigma} \lp{v}{H^1}
    \ , \qquad k\geq 1 \label{LinfH1sob03}\\
    \lp{P_{\geq\rho^\sigma} v}
    {B^\infty_{\rho}} \ &\lesssim \ \rho^{-\sigma/2} \lp{v}{H^1}
    \ , \label{Linf_H1_sob}\\
    \lp{\rho^{-1}\partial_y P_{\leqslant \lambda}v }{L^\infty}
    \ &\lesssim \ \frac{\lambda^\frac{1}{2}}{\rho}
    \lp{v}{H^1} \ , \label{low_Dy_linfty}\end{align}
 Moreover,
 \begin{align}
 \|P_{\sim\lambda} u\|_{L^\infty}&\lesssim \|u\|_{L^\infty}\label{eq:projinfty1}\\
 \|P_{\leq\lambda} u\|_{L^\infty}&\lesssim \|u\|_{L^\infty}\label{eq:projinfty2}\\
 \|\partial_y P_{\leq\lambda} u\|_{L^\infty}&\lesssim \lambda
 \|u\|_{L^\infty}\label{eq:projinftyder1}\\
    \lp{u\cdot v}{B^\infty_{\rho}} \ &\lesssim \
     \lp{u}{B^\infty_{\rho}}\cdot\lp{ v}{B^\infty_{\rho}}
     \label{alg_est_linfty}\\
     \|[\partial_\rho,P_{\leq c\rho} ]u\|_{L^\infty}
     &\lesssim \frac{1}{c \rho}\|P_{\sim c\rho} u\|_{L^\infty}\label{eq:projinfty2der}\\
     \|[\partial_\rho,P_{\leq c\rho} ]u\|_{H^1}
     &\lesssim \frac{1}{c \rho}\|u\|_{H^1}\label{eq:projinfty2derH1}
 \end{align}
\end{lem}

\begin{rem}
These first  bound above shows that for the most part the norm  $H^1$ with
slow growth will give $L^\infty$ type control of the solution $v$. However, for very low frequencies
in the range $|\xi|\lesssim \rho^{2\delta}$, we will need to use  an additional argument to gain $L^\infty$
bounds.  This will \emph{not} be based on energy type spaces (see Section 4 below),
but will follow from a direct manipulation of the equation \eqref{main_eq} restricted to low frequencies.
\end{rem}\ret

\begin{proof}[Proof of \eqref{Linf_H1_sob}--\eqref{alg_est_linfty}]
The proof of estimates of this type is standard, and essentially boils down to
quoting the estimates \eqref{B_eq}--\eqref{loc_sob}.
We begin with \eqref{Linf_H1_sob}. Summing over \eqref{loc_sob} we have:
\begin{multline}
    \lp{P_{ \geq \rho^{\sigma}} v}
    {B^\infty_{\rho}} \ \lesssim \ \sum_{\rho^{\sigma}\lesssim \lambda \lesssim \rho}
    \lp{P_\lambda v}{L^\infty} +\sum_{\rho\lesssim \lambda }
    \ln{|\lambda/\rho|}\lp{P_\lambda v}{L^\infty}\ , \\
    \lesssim \ \sum_{\rho^{\sigma}\lesssim \lambda \lesssim \rho}
    \lambda^{-\frac{1}{2}}\, \lambda \lp{\widetilde{P}_\lambda v}{L^2}
    +\sum_{\rho\lesssim \lambda }\ln{|\lambda/\rho|}
    \lambda^{-\frac{1}{2}}\, \lambda \lp{\widetilde{P}_\lambda v}{L^2}\ , \\
    \lesssim \ \left[\left( \sum_{\rho^{\sigma}\lesssim \lambda \lesssim \rho}
    \lambda^{-1}\right)^\frac{1}{2}+\left( \sum_{\rho\lesssim \lambda}(\ln{|\lambda/\rho|})^2
    \lambda^{-1}\right)^\frac{1}{2}\right]\cdot
    \left( \sum_{\rho^{\sigma}\lesssim \lambda \lesssim \rho}
    \lambda^2 \lp{\widetilde{P}_\lambda  v}{L^2}^2\right)^\frac{1}{2}  \\
    \lesssim \ (\rho^{-\sigma/2}+\rho^{-1/2})\lp{v}{H^1} \ .
\end{multline}

The proof of \eqref{low_Dy_linfty} follows from similar dyadic
summing arguments. We leave the details to the reader.\\

The proof of \eqref{eq:projinfty1} and \eqref{eq:projinfty2}
follows from that the
convolution kernels of $ P_{\sim \rho}$ and $ P_{\leqslant \rho}$ is uniformly $L^1$
and \eqref{eq:projinfty2der} follows from the same argument.
Finally, let us prove the algebra estimate \eqref{alg_est_linfty}.
This follows from a standard ``trichotomy''. The first step is to
decompose the frequencies of the product dyadically. This only needs
to be done for frequency blocks bigger than size $\rho$. That is, we
first make the rough decomposition:
\begin{equation}
        u\cdot v \ = \ P_{\leqslant \rho}(u\cdot v) + P_{\geqslant \rho}
        (u\cdot v) \ . \label{uv_decomp}
\end{equation}
For the first term in this last expression, we have
\begin{equation}
         \lp{  P_{\leqslant \rho}(u\cdot v) }{L^\infty} \ \lesssim \
        \lp{u}{L^\infty}\cdot \lp{v}{L^\infty} \ . \notag
\end{equation}
For the RHS of this last line, we easily have a bound in terms of
$\lp{u}{B^\infty_{\rho}}\cdot\lp{v}{B^\infty_{\rho}}$ by
expanding into frequencies and using the triangle inequality.\\

It remains to show \eqref{alg_est_linfty} for the second term on the
RHS of \eqref{uv_decomp}. This will be done by further decomposing
this product into all frequencies $\rho\leqslant \lambda$:
\begin{align}
        P_{\geqslant \rho}(u\cdot v) \ &= \ \sum_{\rho\leqslant
        \lambda}\ P_\lambda \left[ P_{\lambda}u\cdot
        P_{\ll\lambda}v\right] +
        \sum_{\rho\leqslant
        \lambda}\ P_\lambda \left[ P_{\ll\lambda}u\cdot
        P_{\lambda}v\right]\notag\\
        &\ \ \ \ \ \ \ \ \ \ \ \ \ \
        + \ \sum_{\rho\leqslant\lambda}\ P_\lambda \sum_{\mu_1\sim\mu_2\gtrsim\lambda}
        P_{\mu_1}u\cdot P_{\mu_2}v \ , \notag\\\
        &= \ B_1 + B_2 + B_3 \ . \notag
\end{align}
The proof of \eqref{alg_est_linfty} for the terms $B_1$ and $B_2$ is
essentially symmetric (one only needs to keep the weight function
with the first factor). We focus on bounding $B_1$, and leave the
details of the other to the reader. The multipliers $P_\lambda$ and
$P_{\ll\lambda}$ are bounded on $L^\infty$. Therefore we have:
\begin{align}
        \lp{B_1}{B^\infty_{\rho}} \ &\lesssim \ \sum_{\rho\leqslant
        \lambda}\ln{ \left(\frac{\lambda}{\rho}\right)}
        \lp{P_{\lambda}u\cdot
        P_{\ll\lambda}v}{L^\infty} \ , \notag\\
        &\leqslant \ \sup_{\lambda}\lp{P_{\ll\lambda}v}{L^\infty}\cdot
        \sum_{\rho\leqslant \lambda}\ln{\left(\frac{\lambda}{\rho}\right)}
        \lp{P_{\lambda}u}{L^\infty} \ , \notag\\
        &\lesssim \
        \lp{u}{B^\infty_{\rho}}\cdot\lp{v}{B^\infty_{\rho}}
        \ . \notag
\end{align}

Finally, we need to prove the bound \eqref{alg_est_linfty} for the
term $B_3$. By again using the boundedness of various Fourier cutoffs
on $L^\infty$  we have:
\begin{align}
         \lp{B_3}{B^\infty_{\rho}} \ &\lesssim \ \sum_{\lambda\geqslant\,
        \rho}\ln{\left(\frac{\lambda}{\rho}\right)}
        \sum_{\mu_1\sim\mu_2\gtrsim\lambda}
        \lp{P_{\mu_1}u\cdot
        P_{\mu_2}v}{L^\infty} \ , \notag\\
        &\lesssim \
        \sum_{\mu_1\sim\mu_2\gtrsim\rho}\,\sum_{\rho\leqslant \lambda\leqslant\,
        \mu_1}\ln{\left(\frac{\lambda}{\rho}\right)}\,
        \lp{P_{\mu_1}u\cdot
        P_{\mu_2}v}{L^\infty} \ , \notag\\
        &\lesssim \
        \sum_{\mu_1\sim\mu_2\gtrsim\rho}\,\ln{\left(\frac{\mu_1}{\rho}\right)}
        \ln{\left(\frac{\mu_2}{\rho}\right)}\,
        \lp{P_{\mu_1}u\cdot
        P_{\mu_2}v}{L^\infty} \ , \notag\\
        &\lesssim \ \sup_{\mu_1}\lp{P_{\mu_1}v}{L^\infty}\ln{\left(\frac{\mu_1}{\rho}\right)}\cdot
        \sum_{\mu_2\gtrsim\rho} \ln{\left(\frac{\mu_2}{\rho}\right)}
        \lp{P_{\mu_2}u}{L^\infty} \ , \notag\\
        &\lesssim \
        \lp{u}{B^\infty_{\rho}}\cdot\lp{v}{B^\infty_{\rho}}
        \ . \notag
\end{align}
This completes the proof of \eqref{alg_est_linfty}.
\end{proof}


\section{Decay estimates}\label{decaysection}

\subsection{Decay estimates}  In addition to energy estimate, as usual we also need
decay estimates. The decay estimates for high frequencies can be
obtained from the energy bounds of higher derivatives, essentially
using \eqref{LinfH1sob0}. The decay estimates for low frequencies
can be obtained directly from integrating the equation, provided the
decay of the larger frequencies are established. In this section we will show that

\begin{thm} Suppose that for some $0<\delta<1/8$,
\begin{equation}
K=\sup_{1\leq \rho\leq T} \rho^{-\delta} \lp{(v,\dot{v},\rho^{-1}\partial_y v)(\rho)}{H^1}
<\infty .
\end{equation}
Then
\begin{equation}
\sup_{1\leq \rho\leq T} \lp{(v,\dot{v},\rho^{-1/2-\delta}\partial_y v)(\rho)}{B_{\rho}^\infty}
\lesssim K(1+K^2).
\end{equation}
\end{thm}
\begin{proof} The estimate for $(v,\dot{v})$ follows from Lemma \ref{lem:lowfreqdecay} below. The estimate
for $\rho^{-1}\partial_y v$ follows from that
$\|\partial_y v\|_{L^\infty}\lesssim \|\partial_y v\|_{L^2}^{1/2} \|\partial_y v\|_{H^1}^{1/2}
\lesssim K\rho^{\delta+1/2}$.
\end{proof}

\begin{lem}(Low frequency decay estimate)
Suppose that $w$ solves the equation inside of the time slab
$[1,T]\times\mathbb{R}$:
\begin{align}
    \partial_\rho^2 w
    - \frac{1}{\rho^2}\partial_y^2 w + \big(1+\frac{1}{4\rho^2}\big)w
    +\frac{\alpha_0}{\rho^{1/2}} w^2-\frac{\beta_0}{\rho}w^3+\frac{\beta_1}{\rho}
     w \dot{w}^2\ &= \ F \ . \label{generic_EQsec4}
\end{align}
Suppose also that $P_\lambda w=0$, for $\lambda\geq \rho^\sigma$ and
with $3\sigma/2+\delta<1$, $1/2>\delta>0$, $\sigma>0$,
\begin{equation}
K= \sup_{1\leq \rho\leq T} \rho^{-\delta}
\lp{(w,\dot{w})(\rho)}{H^1}<\infty\label{deltaEnergybound}
\end{equation}
Then
\begin{equation}
     \sup_{1\leqslant \rho\leqslant T}\ \lp{(w,\dot{w})(\rho)}{L^\infty}
     \ \lesssim  K(1+K^{2\delta/(1-2\delta)})+
     \int_1^T\
     \lp{  F(\rho)}{L^\infty} \ d\rho \ . \label{D2sec4}
\end{equation}
\end{lem}
\begin{proof} To begin with we note that for $1\leq \rho\leq T,$
\begin{equation}
\lp{(w,\dot{w})(\rho)}{L^\infty}\lesssim K\rho^\delta, \qquad\qquad
\lp{\partial_y^2 w(\rho)}{L^\infty}\lesssim
K\rho^{\delta+3\sigma/2}.\label{weakdecaysec4}
\end{equation}
The first inequality follows from \eqref{deltaEnergybound} and Sobolev estimate; the second also uses the assumed
high frequency cutoff on $w$:
$\|\partial_y^2 w\|_{L^\infty}\lesssim \|\partial_y^2 w\|_{L^2}^{1/2} \|\partial_y^2 w\|_{H^1}^{1/2}
\lesssim \rho^{3 \sigma /2}\|w\|_{H^1}^{1/2}$.
We have
\begin{equation}
\partial_\rho\Big( G(w,\rho) \frac{\dot{w}^2}{2}+H(w,\rho)\Big)=G\dot{w}\Big(
\ddot{w}+\frac{G_w}{2G} \dot{w}^2 +\frac{H_w}{G} \Big) +G_\rho
\frac{\dot{w}^2}{2} +H_\rho\label{energyder}
\end{equation}
where we now take
\begin{equation}
G=e^{\,\beta_1 w^2/2\rho},\qquad
H_w=G\Big(w+\frac{\alpha_0}{\rho^{1/2}}
w^2-\frac{\beta_0}{\rho} w^3\Big),\quad H(0,\rho)=0.\label{energydef}
\end{equation}
Then it follows that (using also equations \eqref{generic_EQsec4}, \eqref{energyder}, \eqref{energydef})
\begin{equation}
\partial_\rho\Big( G(w,\rho) \frac{\dot{w}^2}{2}+H(w,\rho)\Big)=G_\rho\frac{\dot{w}^2}{2} +H_\rho
+G\dot{w}
\Big(F+\frac{1}{\rho^2}\partial_y^2 w -\frac{1}{4\rho^2} w\Big)
\label{energydereq}
\end{equation}
Writing $
H_w=f(w/\rho^{1/2})w$ for some smooth function $f$
it follows that
\begin{equation}
H=f(w/\rho^{1/2})\frac{ w^2}{2}- \rho \int_0^{w/\rho^{1/2}} f^\prime(s)\, \frac{s^2}{2}\, ds
\end{equation}
from which its easy to see that
\begin{equation}
w^2 /4\leq H\leq w^2,\qquad 1/4 \leq G\leq 4,\qquad \text{if}\quad w^2\leq \mu \rho
\end{equation}
for some constant $\mu>0$ depending on $\alpha_0,\beta_0,\beta_1$,
and
\begin{equation}
 H_\rho\lesssim \frac{1}{\rho^{3/2}} w^2 |w|
 \qquad G_\rho\lesssim \frac{1}{\rho^{3/2}} |w|, \qquad \text{if}\quad w^2\leq \mu\rho.
\end{equation}
Hence if we also use \eqref{weakdecaysec4}
\begin{equation}
\partial_\rho\Big( G(w,\rho)
\frac{\dot{w}^2}{2}+H(w,\rho)\Big)\lesssim \frac{1}{\rho^{3/2}}
(\dot{w}^2+w^2)|w|+\big(|F|+K\rho^{\delta+3\sigma/2-2}+K \rho^{\delta-2} \big)|\dot{w}|,
\end{equation}
if $w^2\leq \mu\rho$. This is satisfied if $CK^2\rho^{2\delta} \leq
\mu \rho$, i.e. if $\rho\geq \rho_0=\max\big\{(CK^2/\mu)^{1/(1-2\delta)},1\big\}$.
If
$\rho\leq \rho_0$ we use the bound $K\rho_0^\delta,$ \eqref{deltaEnergybound}, and for
$\rho\geq \rho_0$ we integrate the above equation to get the bound.
In fact, if $E=\sqrt{G\dot{w}^2/2+H}$ and we use \eqref{weakdecaysec4} we get
\begin{equation}
\partial_\rho E\leq C\frac{K\rho^{\delta}}{\rho^{3/2}} E+C |F|
+CK\rho^{\delta+3\sigma/2-2}
\end{equation}
Multiplying by the integrating factor
\begin{equation}
\partial_\rho (E e^{-g})\leq C(|F|+K\rho^{\delta+3\sigma/2-2}) e^{-g},
\end{equation}
where
\begin{equation}
g(\rho)=C\int_{\rho_0}^\rho K s^\delta s^{-3/2}\, ds\lesssim K \rho_0^{\delta-1/2}\lesssim \mu^{1/2},\qquad\text{if } \rho\geq \rho_0.
\end{equation}
Hence
\begin{equation}
E(\rho)\lesssim E(\rho_0)+\int_{\rho_0}^\rho |F(s)|\, ds +K
\end{equation}
where
\begin{equation}
E(\rho_0)\lesssim K\rho_0^\delta \lesssim K\big(1+K^{2\delta/(1-2\delta)}\big).
\end{equation}
\end{proof}

\begin{lem}(Low frequency cubic decay estimate)\label{lem:lowfreqdecay}
Suppose that $v$ solves the equation inside of the time slab
$[1,T]\times\mathbb{R}$:
\begin{align}
    \partial_\rho^2 v
    - \frac{1}{\rho^2}\partial_y^2 v + \big(1+\frac{1}{4\rho^2}\big)v
    -\frac{\beta_0}{\rho}v^3+\frac{\beta_1}{\rho}
      v\dot{v}^2\ &= \frac{\beta(\rho y)}{\rho} v^3+F \ . \label{generic_EQ}
\end{align}
Suppose that
with $0<\delta<1/8$,
\begin{equation}
K= \sup_{1\leq \rho\leq T} \rho^{-\delta}
\lp{(v,\dot{v})(\rho)}{H^1}<\infty
\end{equation}
Then for $6\delta<\sigma<1-\delta$, $\sigma\leq 2/3$,
\begin{equation}
     \sup_{1\leqslant \rho\leqslant T}\ \lp{(v,\dot{v})(\rho)}{L^\infty}
     \ \lesssim  K(1+K^{2\delta/(1-2\delta)}+K^2)
     +\int_1^T\
     \lp{  P_{\leq \rho^\sigma }F(\rho)}{L^\infty} \ d\rho \ . \label{D2}
\end{equation}
\end{lem}
\begin{proof}
First, we control the high frequency part:
$$
\|P_{\ge \rho^\sigma}(v,\dot v)\|_{L^{\infty}}\lesssim \rho^{-\sigma/2}\|(v,\dot v)\|_{H^1}\lesssim K,
$$
since, by our conditions on $\rho,\sigma,$ $\sigma/2\ge\delta.$
 Applying the projection on low frequency to the equation, gives
with $v_1=P_{\leq \rho^\sigma} v$:
\begin{align}
    \partial_\rho^2 v_1
    - \frac{1}{\rho^2}\partial_y^2 v_1 + \big(1+\frac{1}{4\rho^2}\big)v_1
   -\frac{\beta_0}{\rho}v_1^3+\frac{\beta_1}{\rho}
     v_1 \dot{v}_1^2\ &= \ P_{\leq \rho^\sigma} F + R \ . \label{generic_EQ1}
\end{align}
where we must estimate
\begin{multline}
R=\big[\partial_\rho^2,P_{\leq \rho^\sigma}\big] v+
\frac{1}{\rho} P_{\leq\rho^\sigma} \big[\beta v^3\big]\\
+\frac{\beta_0}{\rho}\big(P_{\leq\rho^\sigma} v^3-v_1^3\big)+\frac{\beta_1}{\rho}\big(P_{\leq\rho^\sigma} v\dot{v}^2-v_1\dot{v}_1^2\big)
\end{multline}
Since
\begin{equation}
 P_{\leq \rho^\sigma} v(y)=\int e^{iy\xi}
 \chi(\xi/\rho^{\sigma}) \hat{v}(\xi)\, dx
\end{equation}
where $\chi\in C_0^\infty$ is $1$ on $[-1,1]$ and $0$ outside
$[-2,2]$, we have
\begin{equation}
[ \partial_\rho, P_{\leq \rho^\sigma}]
v(y,\rho)=-\frac{\sigma}{\rho}\int e^{iy\xi}
 \chi^{\,\prime}\big(\frac{\xi}{\rho^{\sigma}}\big)\frac{\xi}{\rho^\sigma}  \hat{v}(\xi,\rho)\, dx
 =-\frac{\sigma}{\rho}{P}^{\,\prime}_{\sim \rho^{\sigma}}  v(y,\rho)
\end{equation}
where now $\chi^\prime$ vanishes inside $[-1,1]$ and outside
$[-2,2]$. It follows that
 \begin{equation}
\|\big[\partial_\rho^2,P_{\leq \rho^\sigma}\big]
v\|_{L^\infty} \lesssim
\frac{1}{\rho}\big(\|P^{\,\prime}_{\sim
\rho^\sigma}\dot{v}\|_{L^\infty}\! + \|P^{\,\prime}_{\sim
\rho^\sigma}v\|_{L^\infty}\big)\leq \frac{1}{\rho^{1+\sigma}}
\|(v,\dot{v})(\rho)\|_{H^1}
 \end{equation}
 Next
 \begin{equation}
\|(1-P_{\leq \rho^{\,\sigma}})v^3\|_{L^\infty} \lesssim
\frac{1}{\rho^{\,\sigma/2}} \|v^3\|_{H^1}\lesssim
\frac{1}{\rho^{\,\sigma/2}} \|v\|_{L^\infty}^2 \|v\|_{H^1}\lesssim
\frac{1}{\rho^{\,\sigma/2}}\|v\|_{H^1}^3
 \end{equation}
 Similarly
 \begin{equation}
\| v^3-v_1^3\|_{L^\infty}\lesssim
\|v-v_1\|_{L^\infty}\big(\|v\|_{L^\infty}+\|v_1\|_{L^\infty}\big)^2
\lesssim \frac{1}{\rho^{\,\sigma/2}} \|v-v_1\|_{H^1}
\big(\|v\|_{H^1}+\|v_1\|_{H^1}\big)^2
 \end{equation}
 The terms with derivatives $\dot{v}$ are estimated in a similar fashion.
Note that
\begin{multline}
\|P_{\leq\rho^\sigma} \big[ fg \big]\|_{L^\infty} \lesssim
\iint_{|\xi|\leq \rho^\sigma}|\hat{f}(\xi-\eta)|| \hat{g}(\eta)|\,
d\eta d\xi\leq \\ \iint_{|\xi-\eta|\leq
2\rho^\sigma}|\hat{f}(\xi-\eta)|| \hat{g}(\eta)|\, d\eta
d\xi+\iint_{|\eta|\geq \rho^\sigma}|\hat{f}(\xi-\eta)||
\hat{g}(\eta)|\, d\eta d\xi\\
\leq \int_{|\xi|\leq 2\rho^\sigma} |\hat{f}(\xi)|\, d\xi
\int|\hat{g}(\eta)|\, d\eta +\int |\hat{f}(\xi)|\, d\xi
\int_{|\eta|\geq \rho^\sigma}|\hat{g}(\eta)|\, d\eta\\
\leq \Big[\int_{|\xi|\leq 2\rho^\sigma} |\hat{f}(\xi)|\, d\xi
\Big(\int \!\!\!\frac{d\eta}{\langle\eta \rangle^2} \Big)^{1/2}
\!\!\!+\int|\hat{f}(\xi)|\, d\xi \Big(\int_{|\eta|\geq \rho^\sigma}
\!\!\frac{d\eta}{\langle \eta\rangle^2} \Big)^{1/2}\Big] \Big(\int
\langle\eta\rangle^2 |\hat{g}(\eta)|^2\, d\eta\Big)^{1/2}
\end{multline}
where $\langle\eta\rangle=(1+\eta^2)^{1/2}$. Since the Fourier
transform of $\beta(\rho y)$ is
$\widehat{\beta}(\xi/\rho)/\rho$ it follows that
\begin{equation}
\int_{|\xi|\leq 2\rho^\sigma}
|\widehat{\beta}(\xi/\rho)/\rho\,|\,d \xi
=\int_{|\xi|\leq 2\rho^{\sigma-1}}
|\widehat{\beta}(\xi)|\,d \xi \leq C_\beta
\,\rho^{\,\sigma-1}
\end{equation}
Hence
\begin{equation}
\|P_{\leq\rho^\sigma} \big[ (\beta-\beta_0)v^3
\big]\|_{L^\infty}\lesssim C(\rho^{\sigma-1}+\rho^{-\sigma/2})
\|v^3\|_{H^1}\lesssim C\rho^{-\sigma/2} \|v\|_{H^1}^3
\end{equation}
since $\sigma\leq 2/3$. It follows that
\begin{equation}
\int_1^T \|R(\rho)\|_{L^\infty}\, d\rho\lesssim K^3
\end{equation}
\end{proof}

\begin{lem}(Low frequency quadric normal form decay estimate)\label{lem:lowfreqdecay1}
Suppose that $v$ solves the equation inside of the time slab
$[1,T]\times\mathbb{R}$:
\begin{align}
    \partial_\rho^2 v
    - \frac{1}{\rho^2}\partial_y^2 v + \big(1+\frac{1}{4\rho^2}\big)v
    -\frac{\alpha_0}{\rho^{1/2}}v^2-\frac{\beta_0}{\rho}v^3+\frac{\beta_1}{\rho}
      v\dot{v}^2\ &= \frac{\beta(\rho y)}{\rho} v^3\ . \label{generic_EQ2}
\end{align}
Suppose that
with $0<\delta<1/8$,
\begin{equation}
K= \sup_{1\leq \rho\leq T} \rho^{-\delta}
\lp{(v,\dot{v})(\rho)}{H^1}<\infty
\end{equation}
Then
\begin{equation}
     \sup_{1\leqslant \rho\leqslant T}\ \lp{(v,\dot{v})(\rho)}{L^\infty}
     \ \lesssim  K(1+K^2)
     \label{D3}
\end{equation}
\end{lem}
\begin{proof} What is different from the previous lemma is that we now also
have the term with $\alpha_0$ which we deal with, with a normal form transformation.
Let
\begin{equation}
w=\frac{\alpha_0}{3\rho^{1/2}} \big( v_{4}^2+2\dot{v}_{4}^2\big),\qquad v_4=P_{\leq 4\lambda} v,\qquad \lambda =\rho^\sigma
\end{equation}
Then
\begin{multline}
\big(\partial_\rho^2+1\big)w=\frac{\alpha_0}{3\rho^{1/2}} \big( 2 v_4 \ddot{v}_4 +2\dot{v}_4^2+4\dot{v}_4\dddot{v}_4+4\ddot{v}_4^2+v_{4}^2+2\dot{v}_{4}^2\big)
+O\big(\frac{1}{\rho^{3/2}}\big)\\
=\frac{\alpha_0}{3\rho^{1/2}} \big( 4\dot{v}_4(\dddot{v}_4+\dot{v}_4)+4(\ddot{v}_4+v_4)^2+3v_{4}^2
-6(\ddot{v}_4+v_4)v_4\big)
+O\big(\frac{1}{\rho^{3/2}}\big)
\end{multline}
so with $\Box_1=\partial_\rho^2-\rho^{-2}\partial_y^2+1$
\begin{multline}
\Box_1 w-\frac{\alpha_0}{\rho^{1/2}} v_4^2-\frac{\alpha_0^2}{\rho}\big(\frac{8}{3} v_4\dot{v}_4^2-2v_4^3\big)\\=F_1=
\frac{\alpha_0}{3\rho^{1/2}} \Big(4\dot{v}_4\big(\Box_1\dot{v}_4-2\frac{\alpha_0}{\rho^{1/2}}v_4\dot{v}_4\big)
-6v_4\big(\Box_1 v_4-\frac{\alpha_0}{\rho^{1/2}}v_4^2\big)+
\\
+4\big(\Box_1 v_4\big)^2+\big(- 8\Box_1 v_4+4\big(v_4-\frac{1}{\rho^2}\partial_y^2 v_4)\big)\frac{1}{\rho^2}\partial_y^2 v_4+2\big(\frac{1}{\rho}\partial_y v_4\big)^2 +4\big(\frac{1}{\rho}\partial_y \dot{v}_4\big)^2\Big)+ O\big(\frac{1}{\rho^{3/2}}\big)
\end{multline}
Here
\begin{equation}
\Box_1 v_4=\ddot{v}_4+v_4-\frac{1}{\rho^2} \partial_y^2 v_4=\big[\partial_\rho^2,P_{\leq 4\rho^\sigma}\big]v+\frac{\alpha_0}{\rho^{1/2}} P_{\leq 4\lambda} v^2
+ O\big(\frac{1}{\rho}\big)
\end{equation}
and
\begin{equation}
P_{\leq 4\lambda} (uv)-u_4 v_4=\big((u-u_4)v +u_4(v-v_4)\big)
+P_{\geq 4\lambda}( u v)
\end{equation}
so
\begin{equation}
|P_{\leq 4\lambda} (uv)-u_4 v_4|\lesssim \frac{1}{\lambda^{1/2}} \|u\|_{H^1}\|v\|_{H^1}.
\end{equation}
It follows that
\begin{equation}
|F_1|\lesssim \frac{K^2}{\rho^{1+1/4}}
\end{equation}
We have
\begin{multline}
   \Big(\partial_\rho^2
    - \frac{1}{\rho^2}\partial_y^2 v + 1+\frac{1}{4\rho^2}\Big)(v-w)
   +\frac{\alpha_0^2}{\rho}\big(\frac{8}{3} v_4\dot{v}_4^2-2v_4^3\big) -\frac{\beta_0}{\rho}v^3+\frac{\beta_1}{\rho}
      v\dot{v}^2\\=
       \frac{\alpha_0}{\rho^{1/2}}( v^2-v_4^2)+\frac{\beta(\rho y)}{\rho} v^3-\frac{1}{4\rho^2} w-F_1
\end{multline}
In the terms in the top row we can first replace $v_4$ by $v$ with an error
$\rho^{-1-\sigma/2+2\delta}$ and then replace $v$ by $v-w$ with an error
$\rho^{-1-1/2+2\delta}$ so with $|F_2|\lesssim \rho^{-1-1/4}$ we have
\begin{multline}
   \Big(\partial_\rho^2
    - \frac{1}{\rho^2}\partial_y^2 v + 1+\frac{1}{4\rho^2}\Big)(v-w)
 +\frac{8\alpha_0^2/3-\beta_0}{\rho}(v-w)^3+\frac{\beta_1+2\alpha_0^2}{\rho}
      (v-w)(\dot{v}-\dot{w})^2\\=
       \frac{\alpha_0}{\rho^{1/2}}( v^2-v_4^2)+\frac{\beta(\rho y)}{\rho} (v-w)^3-\frac{1}{4\rho^2} w-F_1+F_2
\end{multline}
We are now in position to apply the previous lemma. The only term that remains to be controlled is
\begin{equation}
|P_{\leq \lambda} (v^2-v_4^2)|\lesssim
|P_{\leq \lambda} (v-v_4)^2|
+|P_{\leq \lambda}\big( (v-v_4) P_{\geq 3\lambda}v_4\big)|
\end{equation}
In either case we have a product of factors $uw$ each of which has frequencies
$\geq 2\lambda $ say so
\begin{multline}
|P_{\leq \lambda} (uw)|\leq \|uw\|_{L^2}^{1/2} \|uw\|_{H^1}^{1/2}
\\ \lesssim ( \|u\|_{L^\infty}+\|w\|_{L^\infty})
(\|u\|_{L^2}+\|w\|_{L^2})^{1/2} (\|u\|_{H^1}+\|w\|_{H^1})^{1/2} \\
\lesssim
(\|u\|_{L^2}+\|w\|_{L^2})(\|u\|_{H^1}+\|w\|_{H^1})
\lesssim
\frac{1}{\lambda} (\|u\|_{H^1}+\|w\|_{H^1})^2
\end{multline}
The proof follows from this. Note that the critical thing here was that
for each of the factors $(v-v_4)v_4$ only frequencies higher than $3\lambda $
entered because of the projection $P_{\leq \lambda}$.
\end{proof}

\begin{rem} Note that the region with frequencies of $\beta(\rho y)$ less than
$\lambda\leq \rho^{1/2}$ is easy to deal with. In fact, with $\chi(s)$ a cutoff function supported when $|s|\leq 2$ we have
$$
P_{\leq \rho^{\sigma}} \beta(\rho y)
=\int e^{i y\rho \zeta} \chi(\zeta \rho^{1-\sigma}) \hat{\beta}(\zeta)\, d\zeta
$$
It follows that
$$
|\partial_y^k P_{\leq \rho^{\sigma}} \beta(\rho y) |\lesssim \rho^{(k+1)\sigma-1},
$$
and
$$
\|\partial_y^k P_{\leq \rho^{\sigma}} \beta(\rho y) \|_{L^2_y}\lesssim \rho^{(k+1)\sigma-3/2},
$$
It follows that these are bounded for $k\leq 1$ as long as $\sigma\leq 1/2$ or $\sigma\leq 3/4$, respectively.
See also the proof of Lemmas \ref{lem:lowfreqdecay},\ref{lem:lowfreqdecay1}.
\end{rem}


\section{Semi-Classical Operators}

All of the multilinear operators we consider in this paper are defined for fixed time
(but depend on time), and in general are non-local in the spatial variable.
Because the derivative $\rho^{-1}\partial_y$ has a semi-classical
nature as $\rho^{-1}\to 0$, we will define all of our operators semi-classically
with respect to the derivative:
\begin{equation}
    D_y \ = \ \frac{1}{i\rho}\partial_y \ . \notag
\end{equation}
This will turn out to be very convenient in calculations, but one needs to keep in mind
that there is a non-trivial commutator with $\partial_\rho$:
\begin{equation}
    [\partial_\rho,D_y] \ = \ -\frac{1}{\rho} D_y \ . \notag
\end{equation}
This will lead to certain error terms in the sequel, but they are asymptotically mild
given that they contain an extra power of decay in $\rho$. In terms of this notation, the
equation \eqref{main_eq} becomes:
\begin{equation}
    \Big(\partial_\rho^2 +
    D_y^2 + \big(1+\frac{1}{4\rho^2}\big)\Big)v \ = \ \frac{\alpha_0}{\rho^\frac{1}{2}} v^2 +
    \frac{\beta(\rho y )}{\rho}v^3 +  \frac{\beta_0}{\rho}v^3+ \frac{1}{\rho}\mathcal{R}_\beta(\rho,y)v^3
     \ . \label{semi_main_eq}
\end{equation}
Also associated with this notation, one has the semi-classical version of the Fourier transform
\eqref{FT}:
\begin{equation}
    \mathcal{F}_\rho (v)(\xi)  =
    \widetilde{v}(\xi)  = \rho \int_\mathbb{R} e^{-i\rho \xi y} v(y) \,dy
    =\rho \,\hat{v}\,(\rho\xi)\ , \quad\hat{v}(\xi)=\int e^{ix\xi} v(x)\, dx. \label{SC_F}
\end{equation}
This definition gives $\widetilde{D_y v} = \xi\widetilde{v} $. The inverse of the transformation
is of course:
\begin{equation}
     v(y) \ = \ \mathcal{F}^{-1}_\rho (\widetilde{v})(y) \ = \
    \frac{1}{2\pi}\ \int_\mathbb{R}\  e^{i\rho y  \xi }\ \widetilde{v}(\xi) \ d\xi \ . \label{SC_Finv}
\end{equation}
It is important to keep in mind, although we suppress it with the
$\widetilde{v}$ notation, that our semi-classical Fourier transform
is time dependent, and one has the commutator relation:
\begin{equation}
    [\partial_\rho , \mathcal{F}_\rho] \ = \ -i \mathcal{F}_\rho\, yD_y \
    \ . \notag
\end{equation}
In the present work, it will turn out that expressions of this form will
not play a role because all of our operators will be translation invariant
with respect to $y$.\\

We now define the type of bilinear  $\Psi$DO we are working with
here. Given a ``symbol'' $\widehat{b}(\xi,\eta)$, we define its
associated operator:
\begin{multline}
    B( {}^{1\!}D_y,{}^{2\!}D_y)(u,v) \ = \ \frac{1}{4\pi^2}\,
    \int\int {b}(\xi,\eta)\,
    e^{i\rho y(\xi+\eta)}
    \widetilde{u}(\xi)\cdot\widetilde{v}(\eta)\ d\xi d\eta\\
   =\frac{1}{4\pi^2}\,
    \int\int {b}(\xi/\rho,\eta/\rho)\,
    e^{iy(\xi+\eta)}
    \widehat{u}(\xi)\cdot\widehat{v}(\eta)\ d\xi d\eta \\
    \equiv B(u,v) \ . \label{pdo_def}
\end{multline}
We say that $B$ is a bilinear $\Psi$DO if its symbol obeys the uniform bounds:
\begin{equation}
\sum_{k+\ell\leq N}\sup_{\xi,\eta}  (1+|\xi|)^k (1+|\eta|)^\ell|
\partial_\xi^k \partial_\eta^\ell\, {b}\,(\xi,\eta)|
    \leq \ C_N \ . \label{pdo_bounds}
\end{equation}
 A typical such expression would be something like
${b}(\xi,\eta)=\frac{q(\xi,\eta)}{p(\xi,\eta)}$ for
polynomials $p$ and $q$, with $p$ nonvanishing and of degree greater
that or equal to the degree of $q$. To further our computations,
we record here the effect of differentiating the expressions $K(u,v)$:\\

\begin{lem}\label{lem:com} [Bilinear $\Psi$DO Calculus]
Let $K$ be an operator as defined on  \eqref{pdo_def}, then one has the following
identities:
\begin{equation}
        D_y B( {}^{1\!}D_y,{}^{2\!}D_y)(u,v) \ = \
     B( {}^{1\!}D_y,{}^{2\!}D_y)(D_y u,v)
    + B( {}^{1\!}D_y,{}^{2\!}D_y)(u,D_y v) \ . \label{leib1}
\end{equation}
and:
\begin{multline}
        \partial_\rho B( {}^{1\!}D_y,{}^{2\!}D_y)(u,v) \ = \
     B( {}^{1\!}D_y,{}^{2\!}D_y)(\partial_\rho u,v)
    + B({}^{1\!}D_y,{}^{2\!}D_y)(u,\partial_\rho v)\\
    - \rho^{-1}\partial_1 B( {}^{1\!}D_y,{}^{2\!}D_y)(D_y u,v)
    - \rho^{-1}\partial_2 B( {}^{1\!}D_y,{}^{2\!}D_y)(u,D_y v)
    \ . \label{leib2}
\end{multline}
Here  $\partial_1 B$ and $\partial_2 B$
have symbols $\partial_\xi {b}$ and $\partial_\eta {b}$ respectively.
\end{lem}\ret

\begin{proof}[Proof of the identities \eqref{leib1}--\eqref{leib2}]
We only prove the second identity \eqref{leib2} as the first is immediate.
By rescaling into the time independent Fourier variables
given by \eqref{FT},  \eqref{pdo_def} may be rewritten as:
\begin{equation}
    B( {}^{1\!}D_y,{}^{2\!}D_y)(u,v) \ = \ \frac{1}{4\pi^2}\,
    \int\int {b}(\rho^{-1}\xi,\rho^{-1}\Theta)
    e^{i y(\xi+\Theta)}
    \widehat{u}(\xi)\cdot\widehat{v}(\Theta)\ d\xi d\Theta \ . \notag
\end{equation}
Taking a $\partial_\rho$ derivative of this last line, and then rescaling back
to the semi-classical variables $\xi,\eta$ we see that:
\begin{multline}
    \partial_\rho B( {}^{1\!}D_y,{}^{2\!}D_y)(u,v) \ = \
    - \frac{1}{4\pi^2}\,
    \int\int \frac{\xi}{\rho} \partial_\xi {b}(\xi,\eta)
    e^{i\rho y(\xi+\eta)}
     \widetilde{u}(\xi)\cdot\widetilde{v}(\eta)\ d\xi d\eta\\
     - \frac{1}{4\pi^2}\,
    \int\int \frac{\eta}{\rho} \partial_\eta {b}(\xi,\eta)
    e^{i\rho y(\xi+\eta)}
    \widetilde{u}(\xi)\cdot \widetilde{v}(\eta)\ d\xi d\eta
     \ . \notag
\end{multline}
The proof is finished with an application of the identity $\widetilde{D_y v} = \xi\widetilde{v} $.
\end{proof}\ret

Finally, we end this section by listing and proving a set of estimates that will
allow us to control expressions involving operators of the form $K(u,v)$. These are:

\begin{prop}[Estimates for $\Psi$DO]\label{lem:psidoop}
Let $b(\xi,\eta)$ be a symbol which obeys the uniform
bounds \eqref{pdo_bounds}, for $N\geq 2$. Then if $B(u,v)$ is the
operator defined on  \eqref{pdo_def}, one has the uniform
bounds:
\begin{align}
    \lp{B(u,v)}{L^p} \ &\lesssim \ \|B\|\,
    \lp{u}{L^p}\cdot \lp{v}{B^\infty_{\rho}} \ , \quad
    1\leq p\leq \infty\label{pdo_est1}\\
    \lp{B(u,v)}{H^1} \ &\lesssim \ \|B\|\,
    \Big(\lp{u}{H^1}\cdot \lp{v}{B^\infty_{\rho}} +
     \lp{u}{B^\infty_{\rho}}\cdot\lp{v}{H^1}\Big) \ , \label{pdo_est3}\\
    \lp{B(u,v)}{B^\infty_{\rho}} \
    &\lesssim \ \|B\|\, \lp{u}{B^\infty_{\rho}}\cdot \lp{v}{B^\infty_{\rho}}  \ . \label{pdo_est5}
\end{align}
Here
 \begin{equation}
\|B\|\lesssim \|B\|_{1,1}\big(1+\ln{\big(\|B\|_{2,2}/\|B\|_{1,1}\big)}\big)^2,
\label{eq:Kopnorm}
 \end{equation}
where
 \begin{equation}
\|B\|_{M,N}=\max_{k,\ell\geq 0} \sum_{m\leq M,n\leq N}
\iint_{D_{k,\ell}} (1+|\xi|)^{m-1}
(1+|\eta|)^{n-1}|\partial_\xi^m\partial_\eta^n b(\xi,\eta)|\,
d\xi\, d\eta,
\end{equation}
and $ D_{0,0}=\{|\xi|,|\eta|\leq 4\}$ and $D_{k,\ell}=\{
2^{-1}\leq |\xi|/2^k,|\eta|/2^\ell\leq 2\}$.
\end{prop}
  \ret
The proof of the proposition will follow from dyadic decomposition and integration by parts on the Fourier transform side; it uses the
following well known estimate (proven by duality):
\begin{lem} For $1\leq p\leq \infty$ we have
\begin{equation}
        \lp{\dint K(y;,x_1,x_2) u(x_1)v(x_2)dx_1dx_2}{L^p(y)}
        \leq C(K) \,\,  \lp{u}{L^p}\cdot\lp{v}{L^\infty},
\end{equation}
if
\begin{equation}
 C(K)=\sup_y \lp{ K(y;x_1,x_2)}{L^1_{x_1,x_2}}+
\sup_{x_1} \lp{ K(y;x_1,x_2) }{L^1_{y,x_2}}<\infty \label{eq:KopL1norm}
\end{equation}
\end{lem}

\begin{proof}[Proof of the estimate \eqref{pdo_est1}]
We begin with the proof of \eqref{pdo_est1}.
An important property of the translation independent bilinear
operators $B(u,v)$ we are working with is that they obey the same
frequency combination rules as ordinary products. Thus, one may
always decompose them via ``trichotomy''. This follows at once from
taking the semi-classical Fourier transform of the RHS of
\eqref{pdo_def} which gives (note that the usual FT \eqref{FT} is
just a rescaling of this):
\begin{equation}
           \rho\  \int_\mathbb{R}\ e^{-i\rho \zeta y}\ B(u,v)(y)\ dy
           \ = \ \frac{1}{2\pi}\ \dint_{\zeta=\xi+\eta} b(\xi,\eta)
        \widetilde{u}(\xi)\cdot\widetilde{v}(\eta)\ d\xi d\eta \ .
        \notag
\end{equation}
Because the operator $B(u,v)$ is semiclassical, it also  has a preferred
scale. Therefore, we first make a bulk decomposition:
\begin{equation}
        B(u,v) \ = \ P_{\lesssim \rho}B(u,v) +
        P_{\gtrsim\rho} B(u,v) \ . \label{rough_pdo_decomp}
\end{equation}
We will prove \eqref{pdo_est1} for each of these terms separately,
beginning with the first. This term may again be decomposed as
follows:
\begin{equation}
         P_{\lesssim \rho}B(u,v) \ = \
         P_{\lesssim \rho}B(P_{\lesssim \rho}u,P_{\lesssim
         \rho}v) +  \sum_{\mu_1\sim\mu_2\gtrsim\rho}\
        P_{\lesssim \rho} B(P_{\mu_1}u,P_{\mu_2}v) \ . \notag
\end{equation}
Note that here the last sum is over all $|\mu_1+\mu_2|\leq \rho$,
but since also $|\mu_1|>\rho$ and $|\mu_2|>\rho$ and the sum is over
dyadic $2^k$, there is only finitely many $\mu_1$ satisfying this
for each $\mu_2$.  Our aim is now to prove the two frequency
localized estimates:
\begin{align}
        \lp{B(P_{\lesssim \rho}u,P_{\lesssim
         \rho}v)}{L^p} \ &\lesssim \|B\|\,\,
         \lp{u}{L^p}\cdot\lp{v}{L^\infty} \ , \label{freq_loc_est1}\\
        \lp{B(P_{\mu_1}u,P_{\mu_2}v)
        }{L^p} \ &\lesssim \|B\|\,\,
         \lp{P_{\mu_1}u}{L^p}\cdot\lp{P_{\mu_2}v}{L^\infty} \ .
         \label{freq_loc_est2}
\end{align}
By summing over all $\mu_1\sim\mu_2\gtrsim \rho$ in the second
estimate we see that \eqref{freq_loc_est2} implies:
\begin{equation}
       \sum_{\mu_1\sim\mu_2\gtrsim \rho} \lp{B(P_{\mu_1}u,P_{\mu_2}v)
        }{L^p} \ \lesssim
        \sup_{\mu\geq \rho} \lp{P_{\mu}u}{L^p}\cdot\sum_{\mu\geq \rho}\lp{P_{\mu}v}{L^\infty}
        \lesssim
         \lp{u}{L^p}\cdot\lp{v}{B^\infty_\rho} \ . \notag
\end{equation}
This, together with \eqref{freq_loc_est1} will establish
\eqref{pdo_est1} for the first term on the RHS of
\eqref{rough_pdo_decomp}.\\

We are now trying to prove
\eqref{freq_loc_est1}--\eqref{freq_loc_est2}. Both of these will
result from a similar argument. The goal is to show that in each
case the associated integral kernel $K(y;x_1,x_2)$ enjoys a sort of
uniform bilinear $L^1$ property.

Therefore, we only need to recover the conditions \eqref{eq:KopL1norm}
for the kernels of the bilinear operators $B(P_{\lesssim \rho}u,P_{\lesssim\rho}v)$ and
$B(P_{\mu_1}u,P_{\mu_2}v)$. The symbols are respectively given by:
\begin{align}
    b_{\lesssim\rho,\lesssim\rho}(\xi,\eta) \ &= \ b(\xi,\eta)\chi(\xi)\chi(\eta)
    \ ,
    &b_{\mu_1,\mu_2}(\xi,\eta) \ &= \ b(\xi,\eta)
    p_1(\rho\mu_1^{-1}\xi)p_1(\rho\mu_2^{-1}\eta) \  . \notag
\end{align}
Here $\chi\equiv 1$ in a neighborhood of $0$, and $p_1$ is an
annular cutoff of the form used to define $P_\lambda$ for $1 <
\lambda$. To obtain the desired bounds, in both cases we compute the physical space kernel from
the symbol:
\begin{equation}
    K(y;x_1,x_2) \ = \ \frac{1}{4\pi} \rho^{\,2}\ \dint b(\xi,\eta)
    e^{i\rho\xi(y-x_1)}e^{i\rho\eta(y-x_2)}\ d\xi d\eta \ .
    \label{phy_space_form}
\end{equation}
Multiple integration by parts then shows that for any $M,N\geq 0$
\begin{equation}
|K(y;x_1,x_2)| \lesssim
\frac{\rho^{\,2}}{(\rho|y-x_1|)^M(\rho|y-x_2|)^N} \ \dint
|\partial_\xi^M
\partial_\eta^M b(\xi,\eta)| d\xi d\eta \ .
\end{equation}
Here
\begin{equation*}
\dint |\partial_\xi^M
\partial_\eta^M b_{\lesssim\rho,\lesssim\rho}(\xi,\eta)| d\xi
d\eta\lesssim \|B\|_{M,N}
\end{equation*}
and
\begin{equation*}
 \dint |\partial_\xi^M
\partial_\eta^M b_{\mu_1,\mu_2}(\xi,\eta)| d\xi
d\eta\lesssim \|B\|_{M,N}\frac{\rho^{M+N}}{\mu_1^M \mu_2^N}
\end{equation*}
 Hence one has  the bounds:
\begin{align}
    |K_{\lesssim\rho,\,\lesssim\rho}(y;x_1,x_2) | \ &\lesssim
    \frac{\|B\|_{M,N}\,\, \rho^{\,2} }
    {(1+\rho\,|y-x_1|)^M(1+\rho\,|y-x_2|)^N} \ , \label{K_rr_est}\\
    |K_{\mu_1,\mu_2}(y;x_1,x_2) | \ &\lesssim
    \frac{\|B\|_{M,N}\,\,\mu_1 \mu_2 }
    {(1+\mu_1|y-x_1|)^M(1+\mu_2|y-x_2|)^N}
    \label{K_mm_est}
\end{align}
These easily imply \eqref{eq:KopL1norm} with a constant as in
\eqref{eq:Kopnorm} in both cases.
In fact, for the first one when $\rho|y-x_1|\leq
\|B\|_{2,2}/\|B\|_{1,1}$ and $\rho|y-x_2|\leq
\|B\|_{2,2}/\|B\|_{1,1}$ we use the inequality for $M=N=1$ and in
the complement we use the inequality for $M=N=2$.

To finish, we need to establish \eqref{pdo_est1} for the second term
on the RHS of \eqref{rough_pdo_decomp}. This will result from
decompositions and estimates that are almost identical to what we
have just done. We begin by decomposing:
\begin{multline}
        P_{\gtrsim \rho} B(u,v) \ = \ \sum_{\lambda\gtrsim
        \rho}\ P_\lambda
        B(P_\lambda u, P_{\ll \lambda }v) + \sum_{\lambda\gtrsim
        \rho}\ P_\lambda B(P_{ \ll \lambda} u, P_\lambda v)\\
        + \sum_{\substack{\lambda\gtrsim \rho\\
        \mu_1\sim\mu_2\gtrsim\lambda}}\
        P_\lambda B(P_{\mu_1}u,P_{\mu_2}v) \ . \label{*}
\end{multline}
For the last term we have by
\begin{multline}
       \sum_{\mu_1\sim\mu_2\gtrsim \lambda\gtrsim \rho} \lp{P_\lambda B(P_{\mu_1}u,P_{\mu_2}v)
        }{L^p}
      \lesssim  \sum_{\mu_1\sim\mu_2\gtrsim \rho}\,\,\,\sum_{\mu_2\gtrsim \lambda\gtrsim \rho} \lp{B(P_{\mu_1}u,P_{\mu_2}v)
        }{L^p}
        \\ \lesssim \|B\|
        \sum_{\mu_1\sim\mu_2\geq \rho} \ln{|\mu_2/\rho|}\,\, \lp{P_{\mu_1}u}{L^p}\lp{P_{\mu_2}v}{L^\infty}
        \lesssim \|B\|\,\,
         \lp{u}{L^p}\cdot\lp{v}{B^\infty_\rho} \ . \notag
\end{multline}

 By using the Plancherel theorem and a
bit of dyadic summation (e.g. the dyadic convolution version of
Young's inequality),  it suffices to show the two estimates:
\begin{align}
        \lp{ B(P_\lambda u, P_{\ll \lambda }v)}{L^p} \ &\lesssim
        \|B\|\,\,
        \lp{P_\lambda u}{L^p}\cdot\lp{v}{B^\infty_{\rho}} \
        , \label{high_est1}\\
        \lp{B(P_{\ll \lambda} u, P_\lambda v)}{L^p} \ &\lesssim
        \|B\|\,\,
        \lp{ u}{L^p}\cdot\ln{|\lambda/\rho|}\,\,\lp{P_\lambda v}{L^\infty} \
        , \label{high_est2}
\end{align}

The proof of the first estimate \eqref{high_est1} is also very
similar to work already done. Using the $\ell^1$ Besov structure in
the definition of the space $B^\infty_\rho$ for frequencies
$\rho\leqslant \mu_2$, we may reduce the proof of this estimate to
showing the two bounds:
\begin{align}
            \lp{ B(P_\lambda u, P_{\lesssim \rho }v)}{L^p} \ &\lesssim
            \|B\|\,\,
        \lp{P_\lambda u}{L^p}\cdot\lp{v}{L^\infty} \
        , \notag\\
        \lp{ B(P_\lambda u, P_{\mu_2 }v)}{L^p} \ &\lesssim \|B\|\,\,
        \lp{P_\lambda u}{L^p}\cdot\lp{P_{\mu_2}v}{L^\infty} \
        , &\rho\ \leqslant\ \mu_2 \ . \label{**}
\end{align}
The proof of both of these estimates boils down to showing that the
kernels of these two bilinear operators obey the bounds
\eqref{K_rr_est}-\eqref{K_mm_est}. This in turn comes from
integrating by parts in the formula \eqref{phy_space_form} to
produce estimates of the from \eqref{K_rr_est}--\eqref{K_mm_est}. We
leave the details to the
reader.\\

Finally, we need to establish \eqref{high_est2}. This bound may in turn be derived
from the two estimates:
\begin{align}
    \lp{B(P_{\ll \rho} u, P_\lambda v)}{L^p} \ &\lesssim
    \|B\|\,\,
        \lp{ u}{L^p}\cdot\lp{P_\lambda v}{L^\infty} \ , \notag\\
    \lp{B(P_{\mu_1} u, P_\lambda v)}{L^p} \ &\lesssim \|B\|\,\,
        \lp{P_{\mu_1} u}{L^p}\cdot\lp{P_\lambda v}{L^\infty} \ ,
    &\rho\ \leqslant\ \mu_1 \ . \notag
\end{align}
Both of these bounds result from kernel estimates similar to
\eqref{K_rr_est}--\eqref{K_mm_est} and are left to the reader.
Notice that the second estimate may be summed over all
$\rho\leqslant \mu_1\leqslant \lambda$ at a cost of
$\ln{|\lambda/\rho|}\sup_{\rho\lesssim
\mu_1\lesssim\lambda}\lp{P_{\mu_1}u}{L^2}$. This gives
\eqref{high_est2}. The details are again left to the reader. This
completes our proof of the estimate \eqref{pdo_est1}.
Next, note that\eqref{pdo_est3} follows from \eqref{pdo_est1} and \eqref{leib2}.
Finally we need to prove \eqref{pdo_est5}:
The proof is similar to the proof of \eqref{pdo_est1}; we need to control an extra log factor $\ln{|\mu_2/\rho|}$
on the right hand side of \eqref{*} and a similar $\ln{|\lambda/\rho|}$ factor on \eqref{**}.
The result now follows from the definition of the $B_\rho^{\infty}$ norm.
\end{proof}\ret

\begin{lem}\label{lem:psidoopsmall}
Let $b(\xi,\eta)$ be a symbol which obeys the uniform
bounds \eqref{pdo_bounds}, for $N\geq 2$ and suppose that
$ b(\xi,\eta)$ vanishes to second order at the origin $(0,0)$.
Then if $B(u,v)$ is the operator defined on  \eqref{pdo_def},
one has the uniform bounds:
\begin{equation}
    \lp{P_{\leq \rho^{\sigma}\,} B(u,v)}{B_\rho^\infty} \
    \lesssim \rho^{-\sigma} (\lp{u}{B^\infty_{\rho}}+\|u\|_{H^1})\cdot
    (\lp{v}{B^\infty_{\rho}}+\|v\|_{H^1})  \,
\end{equation}
for $\sigma<2/3$.
\end{lem}
\begin{proof} With the projection we now have the kernel
\begin{equation}
\hat{R}(\xi,\eta)=\chi( (\xi+\eta)\rho^{1-\sigma})  b(\xi,\eta)
 \end{equation}
Since $|\xi+\eta|<2\rho^{\sigma-1}$ in the support it follows that
either both $|\xi|\leq 4\rho^{\sigma-1}$ and $|\eta|\leq
4\rho^{\sigma-1}$ or both $|\xi|\geq 2\rho^{\sigma-1}$ and
$|\eta|\geq 2\rho^{\sigma-1}$. Hence we have to consider two
operators with kernels
\begin{equation}
\hat{R}_1(\xi,\eta)=\chi( (\xi+\eta)\rho^{1-\sigma})
 b(\xi,\eta)\chi(\xi\rho^{1-\sigma}/8)\chi(\eta\rho^{1-\sigma}/8)
 \end{equation}
and
\begin{equation}
\hat{R}_2(\xi,\eta)=\chi( (\xi+\eta)\rho^{1-\sigma})
 b(\xi,\eta)(1-\chi(\xi\rho^{1-\sigma}))(1-\chi(\eta\rho^{1-\sigma}))
 \end{equation}
 Then its easy to see that
 $$
 \|R_1\|_{2,2}\lesssim 1, \qquad \|R_1\|_{1,1}\lesssim
 \rho^{-2(1-\sigma)}
 $$
and therefore $\|R_1\|\lesssim\rho^{-\sigma}$, if $\sigma<2/3$. For
the last operator we simply have
 \begin{equation}
\|P_{\leq \rho^\sigma} B (P_{\geq \rho^\sigma}u,P_{\geq \rho^\sigma
}v)\|_{B_\rho^\infty}\lesssim \rho^{-\sigma} \|u\|_{H^1}\|v\|_{H^1},
 \end{equation}
 by the previous proposition and \eqref{Linf_H1_sob}.
\end{proof}


\section{Quadratic Normal Forms}\label{quad sect}

In this section, we will construct bilinear operators which can be
used to control the quadratic term on the RHS of
\eqref{semi_main_eq}. This construction follows the classical
normal-forms calculation in Shatah's seminal paper \cite{Sh}, albeit
adapted to the semi-classical setup of the previous section.

Suppose that $u$ and $v$ solve the following Klein-Gordon equations:
\begin{align}
    (\partial_\rho^2 +
   D_y^2 + 1 )u \ &= \ F \ , \notag\\
    (\partial_\rho^2 +
    D_y^2 + 1 )v \ &= \ G \ . \notag
\end{align}
We wish to understand the nature of the solution $w$ to the equation:
\begin{equation}
    (\partial_\rho^2
     + D_y^2 + 1 )w \ = \ \frac{\alpha_0}{\rho^\frac{1}{2}} u\cdot v
     \ . \label{pure_quad}
\end{equation}
In particular, we are trying to obtain estimates on the LHS of the
source term bound \eqref{S_est} \emph{without} having to directly
put the RHS of \eqref{pure_quad} into the space
$\mathcal{S}_\delta[1,T]$ with uniform bounds (assuming of course
that $F$ and $G$ obey reasonable estimates). To do this, we attempt to
construct the solution as a power series $w = w_0 + w_1 + \ldots$
where the first two terms are defined by:
\begin{align}
    (\partial_\rho^2
   + D_y^2 + 1 )w_0 \ &= \ 0 \ ,
    &w_1 \ &= \ Q(u,v;\partial_\rho u,\partial_\rho v ) \ , \notag
\end{align}
where $Q$ is some appropriately chosen bilinear expression. We will take these to be of the form:
\begin{equation}
    w_1 \ = \ \frac{1}{\rho^\frac{1}{2}} B_1(u,v) + \frac{1}{\rho^\frac{1}{2}}
    B_2(\dot{u},\dot{v}) \ , \label{w1_line1}
\end{equation}
where the operators are bilinear $\Psi$DO as defined on
\eqref{pdo_def}. Here we are using the shorthand $\partial_\rho
u=\dot{u}$, and similarly for $\dot{v}$. What we are expecting to
see is the following cancelation:
\begin{equation}
    (\partial_\rho^2 +
    D_y^2 + 1 )w_1 - \frac{\alpha_0}{\rho^\frac{1}{2}} u\cdot v
    \ = \ \{\hbox{terms which decay better}\} \ , \notag
\end{equation}
where the RHS is an expression which can be put into the source term
space $\mathcal{S}_\delta[1,T]$ (assuming a reasonable form for $F$ and $G$).
To uncover all of this, we need to do some calculations.\\


\subsection{Kernel Calculations for the General Case}

\begin{lem} Suppose that $Lu=F$ and $Lv=G$, where $L=\partial_\rho^2 +D_y^2+1$ and set
\begin{equation}
     w_1 \ = \ \frac{1}{\rho^\frac{1}{2}} B_1(u,v) + \frac{1}{\rho^\frac{1}{2}}
    B_2(\dot{u},\dot{v}).  \label{w1_def}
\end{equation}
Then
\begin{multline}
        L w_1 \ = \frac{1}{\rho^\frac{1}{2}}\Big[ - B_1(u,v) +
    2B_1(\dot{u},\dot{v}) - B_2(\dot{u},\dot{v})\\
    + 2 B_2\big((D_y^2+1)u,(D_y^2+1)v\big)
    + 2 B_1(D_y u, D_y v) +
    2 B_2(D_y \dot{u}, D_y \dot{v}) \Big]\\
    +\frac{1}{\rho^\frac{1}{2}}\Big[ B_1(F,v) + B_1(u,G)
    +2 B_2(F,G) + B_2(\dot{F},\dot{v}) + B_2(\dot{u},\dot{G})\\
    - 2 B_2\big((D_y^2 + 1)u,G\big) - 2 B_2\big(F,(D_y^2 + 1)v\big)
    \Big]+\mathcal{C}_1+\mathcal{C}_2+\mathcal{C}_3,\label{Lw1}
\end{multline}
where the commutator terms are
\begin{equation}
\mathcal{C}_1=\partial_\rho^2\Big(\frac{1}{\rho^{1/2}} B_1(u,v)\Big)
- \frac{1}{\rho^{1/2}} \sum_{\ell=0}^2
\binom{2}{\ell}B_1(\partial_\rho^\ell u,\partial_\rho^{2-\ell} v),
\end{equation}
\begin{equation}
\mathcal{C}_2=\partial_\rho^2\Big(\frac{1}{\rho^{1/2}} B_2(\dot{u},\dot{v})\Big)
- \frac{1}{\rho^{1/2}} \sum_{\ell=0}^2
\binom{2}{\ell}B_2(\partial_\rho^\ell \dot{u},\partial_\rho^{2-\ell} \dot{v}),\qquad
\end{equation}
and
\begin{equation}
\mathcal{C}_3=D_y^2\Big(\frac{1}{\rho^{1/2}}
B_2(\dot{u},\dot{v})\Big) - \frac{1}{\rho^{1/2}} \big( 2B_2(D_y
\dot{u},D_y \dot{v}) +B_2(\partial_\rho D_y^2
u,\dot{v})+B_2(\dot{u},\partial_\rho D_y^2 v)\big)
\end{equation}
\end{lem}
The proof of the lemma follows from using the commutator formulas and
and then replacing $\ddot{u}=F-(D_y^2+1)u$ and $\ddot{v}=G-(D_y^2+1)v$.

The primary term which is given explicitly on the RHS of
\eqref{Lw1} must be set to cancel with the quadratic
interaction on the RHS of \eqref{pure_quad}. We write this as a
diagonal $2\times 2$ system in the pair of variables $(u,v)$ and
$(\dot{u},\dot{v})$:
\begin{align}
     - B_1(u,v)
     + 2 B_1(D_y u, D_y v)
    + 2 B_2\big((D_y^2+1)u,(D_y^2+1)v\big) \ &= \ \alpha_0\,  u\cdot v \ , \notag\\
    - B_2(\dot{u},\dot{v}) + 2 B_2(D_y \dot{u}, D_y \dot{v})
    + 2B_1(\dot{u},\dot{v})
     \ &= \ 0 \ . \notag
\end{align}
We would like these last two identities regardless of the form of
$(u,v)$ and $(\dot{u},\dot{v})$. Therefore, we formulate them as a
system for the symbols of the kernels $B_1$ and $B_2$ themselves:
\begin{align}
     (-1+2\xi\eta)b_1
    + 2(\xi^2+1)(\eta^2 +1)b_2
    \ &= \ -\alpha_0 \ , \notag\\
     (-1 + 2\xi\eta)b_2 + 2b_1
    &= 0  . \notag
\end{align}
The solution to this system is:
\begin{align}
        b_1(\xi,\eta)
    \ &= \  \frac{
      \alpha_0 (1-2\xi\eta)}
    {p(\xi,\eta)} \ ,
    &b_2(\xi,\eta)
    \ &= \ \frac{2\alpha_0}
    { p(\xi,\eta)} \ . \label{ops}
\end{align}
Here the determinantal polynomial $p(\xi,\eta)$ is given by the
expression:
\begin{equation}
        p(\xi,\eta)
    \ = \ - (4\xi^2 + 4\eta^2 - 4\xi\eta + 3)
    \ , \label{main_poly}
\end{equation}
which is ``elliptic''. It is therefore not difficult to see that
with this definition, the operators defined on  \eqref{ops} obey
the bounds \eqref{pdo_bounds}. Moreover, any operators that come up
as commutators will satisfy the same estimates.

\begin{lem} \label{lem:generror}
Let $B_1$ and $B_2$ be the operators with symbols \eqref{ops} and set
\begin{equation}
     w_1 \ = \ \frac{1}{\rho^\frac{1}{2}} B_1(u,v) + \frac{1}{\rho^\frac{1}{2}}
    B_2(\dot{u},\dot{v}) \ , \label{w1}
\end{equation}
where $(\partial_\rho^2 +D_y^2+1)u=F$ and $(\partial_\rho^2
+D_y^2+1)v=G$. Then
\begin{equation}
       {\mathcal{E}}_{quad}= (\partial_\rho^2 + D_y^2 + 1)w_1 \ - \frac{\alpha_0 u\,
        v}{\rho^\frac{1}{2}}=\frac{1}{\rho^{1/2}} \mathcal{Q}
   +\mathcal{C}_1+\mathcal{C}_2+\mathcal{C}_3 ,
\end{equation}
where
\begin{multline}
\mathcal{Q}=B_1(F,v) + B_1(u,G)- 2 B_2\big((D_y^2 + 1)u,G\big) - 2 B_2\big(F,(D_y^2 + 1)v\big)\\
     + B_2(\dot{F},\dot{v}) + B_2(\dot{u},\dot{G})+2 B_2(F,G).
\end{multline}
\end{lem}

To analyze and estimate the commutators we have:
\begin{lem}
\label{lem:quadcom} Let $p$ be a symmetric elliptic polynomial of degree two
\begin{equation}
p(\xi,\eta)=c_1(\xi^2+\eta^2)+c_2\xi\eta+c_3,\qquad
|p(\xi,\eta)|\geq c(1+|\xi|+|\eta|)^2>0.
\end{equation}
We say that $B_{ij}$ is an operator of type
$(i,j)$, $0\leq i+j\leq 2$, if the kernel satisfy
\begin{equation}
b_{ij}(\xi,\eta)=\frac{c \, \xi^i\,\eta^j}{p\,(\xi,\eta)^k}.\label{eq:opform}
\end{equation}
We have
\begin{multline}
\partial_\rho^k\Big(\frac{1}{\rho^{1/2}} B_{ij}(u,v)\Big) -
\frac{1}{\rho^{1/2}} \sum_{\ell=0}^k
\binom{k}{\ell}B_{ij}(\partial_\rho^\ell u,\partial_\rho^{k-\ell} v)\\
=\sum_{m=1,...,k}\frac{1}{\rho^{1/2+m}} \sum_{\ell=0}^{k-m} B_{ij,\,m\ell
k}(\partial_\rho^\ell u,\partial_\rho^{k-m-\ell}
v)\label{eq:commutators}
\end{multline}
where $B_{ij,\, m\ell k}$ are operator of the same type as $B_{ij}$.

 We say that $B_1$ is an operator of type $1$ if its
a sum of operators of type $(0,0)$ and $(1,1)$.
We say that $B_2$ is an operator of type $2$ if its an operator of type $(0,0)$.
We say that $B_4$ respectively $B_5$ is an operator of type $4$ if its an operator of type $(2,0)$ respectively $(0,2)$.
 We say that $B_3$ is an operator of type $3$ if its
a sum of operators of type $0,2,4,5$.

If $B_2$ is an
operator of type $2$ then
\begin{equation}
B_2(D_y^2 u,v)=B_4(u,v),\qquad B_2(u,D_y^2 v)=B_5(u,v),
\end{equation}
where $B_4,B_5$ are operators of type $4,5$, respectively, and
\begin{equation}
B_4(u,D_y^2 v)=B_1(D_y u,D_y v),\qquad B_5(D_y^2 u,v)=B_1(D_y u,D_y v),
\end{equation}
where $B_1$ is an operator of type $0$.
If $B_1$ is an operator of type $1$ then
\begin{equation}
B_1(D_y^2 u,v)=B_{31}(D_y u,D_y v)+B_{32}(u,v)
\end{equation}
where $B_{3i}$ are operators of type $3$.
Moreover if $B_2$ is and operator of type $2$ then
\begin{equation}
B_2(D_y^2 u,D_y^2 v)=B_1(D_y u , D_y v).
\end{equation}
where $B_1$ is an operator of type $1$.
\end{lem}
\begin{proof} According to \eqref{leib2} the commutator with $\partial_\rho$ is
$-(\xi\partial_\xi+\eta\partial_\eta) b(\xi,\eta)$. For any homogeneous
polynomial of degree two we have
$(\xi\partial_\xi+\eta\partial_\eta) q(\xi,\eta)=2q(\xi,\eta)$.
The commutator formula follows from this.
\end{proof}

\begin{lem}  With notation as in the previous lemmas we have:
\begin{equation}
\mathcal{C}_1=\frac{1}{\rho^{\,3/2}} \big(B_{1,1}(u,\dot{v})
+B_{1,1}(\dot{u},v)\big)+\frac{1}{\rho^{\,5/2}}
B_{1,2}(u,v).
\end{equation}
\begin{equation}
\mathcal{C}_2=\frac{1}{\rho^{\,3/2}} \big(B_{4, 1}(u,\dot{v})+B_{5,1}(\dot{u},
v)+B_{2,1}(\dot{u},G-v)+B_{2,1}(F-u,\dot{v})\big)+\frac{1}{\rho^{\,5/2}}
B_{2,2}(\dot{u},\dot{v}).
\end{equation}
where we replaced $\ddot{u}=-D_y^2 u-u+F$ and $\ddot{v}=-D_y^2
v-v+G$, and
\begin{equation}
\mathcal{C}_3=\frac{1}{\rho^{3/2}} \big(
B_{4,2}(u,\dot{v})+B_{5,2}(\dot{u},v)\big),
\end{equation}
where $B_{i,j}$ are operators of type $i$.
Moreover
\begin{multline}
\partial_\rho\mathcal{C}_1=\frac{1}{\rho^{\,3/2}}
\big(B_{1,2}(\dot{u},\dot{v}) +
B_{1,3}(u,v)+B_{3,1}(D_y u,D_y
v)\big)\\
+\frac{1}{\rho^{\,5/2}} \big(B_{1,3}(u,\dot{v})
+B_{1,3}(\dot{v},u)\big)+\frac{1}{\rho^{\,7/2}}
B_{1,4}(u,v).
\end{multline}
\begin{multline}
\partial_\rho \mathcal{C}_2=\frac{1}{\rho^{\,3/2}} \big(B_{3,2}(\dot{u},
\dot{v})+B_{3,3}(D_y u,D_y
v)+B_{4,3}(u,G)+B_{5,3}(F,v)\big)\\
+\frac{1}{\rho^{\,3/2}} \big(B_{2,3}(F,G)+B_{2,4}(\dot{u},\dot{G})
+B_{2,5}(\dot{F},\dot{v})\big)\\
\qquad\qquad\qquad\qquad+\frac{1}{\rho^{\,5/2}}
\big(B_{3,4}(\dot{u}, v)+B_{3,5}(u,
\dot{v})+B_{2,6}(\dot{u},G)+B_{2,7}(F,\dot{v})\big)\\
+\frac{1}{\rho^{\,7/2}} B_{2,8}(\dot{u},\dot{v}).
\end{multline}
and
\begin{equation}
\partial_\rho \mathcal{C}_3=\frac{1}{\rho^{3/2}} \big(
B_{0,1}( D_y u, D_y v)+B_{3,6}(\dot{u},\dot{v})\big)
+\frac{1}{\rho^{5/2}} \big(
B_{3,7}(u,\dot{v})+B_{3,8}(\dot{u},v)\big)
\end{equation}
\end{lem}
The following two lemmas follow from Proposition \ref{lem:psidoop} and Lemma \ref{lem:psidoopsmall}:
\begin{lem} \label{lem:comerror} For $i=1,2,3$
\begin{align}
    \lp{\mathcal{C}_i\!}{H^1}  \lesssim \,& \rho^{-3/2}
    \big(\lp{(u,\!\dot{u},F)}{H^1} \lp{(v,\dot{v},G)}{B^\infty_{\rho}} +
    \lp{(u,\!\dot{u},F)}{B_\rho^\infty} \lp{(v,\dot{v},G)}{H^1}\big)\!\!\!\!\!\!\!\!\\
    \lp{\mathcal{C}_i\!}{B^\infty_{\rho}}
    \lesssim \,& \rho^{-3/2}\lp{(u,\dot{u},F)}{B_\rho^\infty} \lp{(v,\dot{v},G)}{B^\infty_{\rho}}.
    \label{b}\\
     \lp{\partial_\rho\mathcal{C}_i}{H^1} \ &\lesssim \rho^{-3/2}
    \lp{(u,\dot{u},D_y u, F,\dot{F})}{H^1} \lp{(v,\dot{v},D_y
    v,G,\dot{G})}{B^\infty_{\rho}}\\
    &+\rho^{-3/2}
    \lp{(u,\dot{u},D_y u, F,\dot{F})}{B_\rho^\infty} \lp{(v,\dot{v},D_y
    v,G,\dot{G})}{H^1}\\
     \lp{\partial_\rho\mathcal{C}_i}{B_\rho^\infty} &\lesssim \rho^{-3/2}
    \lp{(u,\dot{u},D_y u, F,\dot{F})}{ B_\rho^\infty} \lp{(v,\dot{v},D_y
    v,G,\dot{{G}^{}})}{B^\infty_{\rho}}
\end{align}
\end{lem}

\begin{lem} \label{lem:nonlinerror} With notation as in the previous lemma we have
\begin{align}
   \lp{ \mathcal{Q}}{B_\rho^\infty} \
    \lesssim  &\, \lp{(u,\dot{u},v,\dot{v},F,G)}{B^\infty_{\rho}}
    \, \lp{(F,\dot{F},G,\dot{G})}{B^\infty_{\rho}} \,\\
   \lp{ \mathcal{Q}}{H^1} \
    \lesssim & \, \lp{(u,\dot{u},v,\dot{v},F,G)}{B^\infty_{\rho}}
    \, \lp{(F,\dot{F},G,\dot{G})}{H^1}\\
    &+\lp{(u,\dot{u},v,\dot{v},F,G)}{H^1}
    \, \lp{(F,\dot{F},G,\dot{G})}{B^\infty_{\rho}}
\end{align}
Moreover
\begin{equation}
\mathcal{Q}=-\frac{\alpha_0}{3}\Big(-Fv-uG+2\dot{F}\dot{v}+2\dot{u}\dot{G}+4FG\Big)
+\mathcal{R}
\end{equation}
where for $\sigma<2/3$
\begin{equation}
    \lp{P_{ \leq \rho^\sigma\,} \mathcal{R}}{B_\rho^\infty} \
    \lesssim \rho^{-\sigma} \, \lp{(u,\dot{u},v,\dot{v},F,G)}{B^\infty_{\rho}\, \cap H^1}
    \, \lp{(F,\dot{F},G,\dot{G})}{B^\infty_{\rho}\, \cap H^1} \,
\end{equation}
Here $\lp{u}{B^\infty_{\rho}\, \cap
H^1}=\lp{u}{B^\infty_{\rho}}+\|u\|_{H^1}$. Furthermore
\begin{align}
   \lp{ \partial_\rho \mathcal{Q}}{B_\rho^\infty} \
    \lesssim  &\, \lp{(u,\dot{u},v,\dot{v},F,G)}{B^\infty_{\rho}}
    \, \lp{(F,\dot{F},LF,G,\dot{G},LG)}{B^\infty_{\rho}} \,\\
   \lp{\partial_\rho \mathcal{Q}}{H^1} \
    \lesssim & \, \lp{(u,\dot{u},v,\dot{v},F,G)}{B^\infty_{\rho}}
    \, \lp{(F,\dot{F},LF,G,\dot{G},LG)}{H^1}\\
    &+\lp{(u,\dot{u},v,\dot{v},F,G)}{H^1}
    \, \lp{(F,\dot{F},LF,G,\dot{G},LG)}{B^\infty_{\rho}}
\end{align}
where $L=\partial_\rho^2+D_y^2+1$.
\end{lem}

\ret

\subsection{Estimates for the Quadratic Terms}

We now apply the machinery of the last few sections to the task of estimating the
quadratic terms on the RHS of
\begin{equation}
    \Big(\partial_\rho^2
    - \frac{1}{\rho^2}\partial_y^2 +1\Big) v \
    =F= \ \frac{\alpha_0}{\rho^\frac{1}{2}}v^2
    +   \frac{\beta(\rho y )}{\rho}v^3+ \frac{\beta_0}{\rho}v^3
    + \frac{1}{\rho}\mathcal{R}_\beta(\rho,y)v^3 \ .
    \label{semi_main_eq sec6}
\end{equation} To do this, we define the correction:
\begin{equation}
     w_1 \ = \ \frac{1}{\rho^\frac{1}{2}} B_1(v,v) + \frac{1}{\rho^\frac{1}{2}}
    B_2(\dot{v},\dot{v}) \ , \label{w11}
\end{equation}
where $v$ is our solution to \eqref{semi_main_eq sec6} and $B_1$ and $B_2$ are
the operators with symbols \eqref{ops}.  We define the associated error as:
\begin{equation}
    \mathcal{E}_{quad} \ = \ (\partial_\rho^2 +  D_y^2 + 1)w_1 - \frac{\alpha_0}{\rho^\frac{1}{2}}v^2
     \ . \label{quad_error}
\end{equation}
Let us first give a weak decay estimate for $F$ and $LF$:
\begin{lem}
\begin{align}\|(F,\dot{F})\|_{H^1}\lesssim
\frac{1}{\rho^{1/2}} \big(1+\|(v,\dot{v})\|_{L^\infty}\big)\|(v,\dot{v})\|_{L^\infty}
\|(v,\dot{v})\|_{H^1}\label{eq:Fest}\\
\|(F,\dot{F})\|_{B_\rho^\infty}\lesssim
\frac{1}{\rho^{1/2}} \big(1+\|(v,\dot{v})\|_{B_\rho^\infty}\big)
\|(v,\dot{v})\|_{B_\rho^\infty}^2\label{eq:Festinfty}
\end{align} and with $L=\partial_\rho^2+D_y^2+1$
\begin{align}
\|LF\|_{H^1}\lesssim
\frac{1}{\rho^{1/2}} \big(1+\|(v,\dot{v})\|_{L^\infty}\big)
\|(v,\dot{v})\|_{L^\infty}\big( \|(v,\dot{v})\|_{H^1} +\|F\|_{H^1}\big)
\label{eq:LFest}\\
\|LF\|_{B_\rho^\infty}\lesssim
\frac{1}{\rho^{1/2}} \big(1+\|(v,\dot{v})\|_{B_\rho^\infty}\big)
\|(v,\dot{v})\|_{B_\rho^\infty}
\big( \|(v,\dot{v})\|_{B_\rho^\infty} +\|F\|_{B_\rho^\infty}\big)\label{eq:LFestinfty}
\end{align}
\end{lem}
\begin{proof} Let us first prove \eqref{eq:Fest}.
First $\rho$ derivatives falling on the coefficients will only improve the decay
(for the ${\mathcal R}_\beta$ and $\beta$ see Lemma \ref{lem:Rbeta} and its proof.)
and $\rho$ derivatives falling on $v$ can be estimated as the terms without the $\rho$
derivative. It therefore only remains to consider the estimate for $F$ itself.
If the $y$ derivatives fall on $v$ this obvious.
It therefore only remains to consider the case when the $y$ derivative fall on
$\beta$ (or ${\mathcal R}_\beta$ which is better behaved by Lemma \ref{lem:Rbeta}).
It therefore only remains to estimate
\begin{equation}
\|\beta^\prime(\rho y) v^3 \|_{L^2}\leq
\|\beta^\prime(\rho y)\|_{L^2} \|v^3\|_{L^\infty} \lesssim
\frac{1}{\rho^{1/2}} \|v\|_{L^\infty}^3,
\end{equation}
where we estimate one factor $\|(v,\dot{v})\|_{L^\infty}\leq \|(v,\dot{v})\|_{H^1}$.
The only terms that could possibly problematic in
 proving the estimate \eqref{eq:LFest} are when $\partial_\rho$ and $D_y$ falls on
 $\beta(\rho y)$. However that just produces terms
 \begin{equation}
 \rho^{-1-\ell} (\rho y)^\ell \beta^{(k)}(\rho y) v^3 ,\qquad \ell\geq 0
 \end{equation}
Since $\|(\rho y)^\ell \beta^{(k)}(\rho y)\|_{H^1}\leq C\rho^{1/2}$ the estimate
follows also for this term.
\end{proof}
We now have the following:
\begin{prop}
\begin{align}
\|(w_1,\dot{w}_1)\|_{H^1}\lesssim \frac{1}{\rho^{1/2}}\,
\big(1+\|(v,\dot{v})\|_{L^\infty}^2\big)\|(v,\dot{v})\|_{B^\infty_\rho}
\|(v,\dot{v})\|_{H^1}\\
\|(w_1,\dot{w}_1,D_y w_1)\|_{H^1}\lesssim \frac{1}{\rho^{1/2}}\,
\big(1+\|(v,\dot{v})\|_{L^\infty}^2\big)\|(v,\dot{v},D_y v)\|_{B^\infty_\rho}
\|(v,\dot{v},D_y v)\|_{H^1}\\
\|(w_1,\dot{w}_1,D_y w_1)\|_{B_\rho^\infty}\lesssim \frac{1}{\rho^{1/2}}\,
\big(1+\|(v,\dot{v})\|_{B_\rho^\infty}^2\big)\|(v,\dot{v},D_y v)\|_{B^\infty_\rho}^2
\end{align} and
\begin{align}
\|\mathcal{E}_{quad}\|_{H^1}&\lesssim
 \big(1+\|(v,\dot{v})\|_{B^\infty_\rho}^4\big)\, \frac{1}{\rho}\,
 \|(v,\dot{v})\|_{B^\infty_\rho}^2\|(v,\dot{v})\|_{H^1}\\
\|\dot{\mathcal{E}}_{quad}\|_{H^1}&\lesssim
 \big(1+\|(v,\dot{v})\|_{B^\infty_\rho}^6\big)\, \frac{1}{\rho}\,
 \|(v,\dot{v})\|_{B^\infty_\rho}^2 \|(v,\dot{v})\|_{H^1}\\
 \|\mathcal{E}_{quad}\|_{B_\rho^\infty}&\lesssim
 \big(1+\|(v,\dot{v})\|_{B^\infty_\rho}^4\big)\, \frac{1}{\rho}\,
 \|(v,\dot{v})\|_{B^\infty_\rho}^3\\
\|\dot{\mathcal{E}}_{quad}\|_{B_\rho^\infty}&\lesssim
 \big(1+\|(v,\dot{v})\|_{B^\infty_\rho}^6\big)\, \frac{1}{\rho}\,
 \|(v,\dot{v})\|_{B^\infty_\rho}^3
\end{align}
\end{prop}
\begin{proof}
The error $\mathcal{E}_{quad}$ is of the form in Lemma \ref{lem:generror} and the estimate
follows
from the estimates for the
commutator errors in Lemma \ref{lem:comerror} and the nonlinear errors in Lemma
\ref{lem:nonlinerror}. Here we estimate one factor
$\|(v,\dot{v})\|_{B_\rho^\infty}\leq C \|(v,\dot{v})\|_{H^1}$.

The bilinear estimate in
Proposition \ref{lem:psidoop} directly gives the estimate for $\|w_1\|_{H^1}$.
 To get the estimate for $\| D_y w_1\|_{H^1}$ one first applies the
commutator in Lemma \ref{lem:com} and the term so obtained
\begin{equation}
 \frac{1}{\rho^{1/2}} B_2(\dot{v},D_y\dot{v})
\end{equation}
can be estimate as before since the operator $B(u,v)=B_2(v, D_y u)$
also satisfies the estimate in Proposition \ref{lem:psidoop}. To estimate
$\| \dot{w}_1\|_{H^1}$ we again use the commutator in Lemma \ref{lem:com}
and obtain the term
\begin{equation}
 \frac{1}{\rho^{1/2}} B_2(\dot{v},\ddot{v})
\end{equation}
To estimate this term we use the equation $\ddot{v}=D_y^2 v-v+F$. We see that we
hence have to control
\begin{equation}
 \frac{1}{\rho^{1/2}} B_2(\dot{v},D_y^2 v)
\end{equation}
However $B(u,v)=B_2(u,D_y^2 v)$ is also an operator of the form
we have estimates for in Proposition \ref{lem:psidoop}.
\end{proof}
We remark that the above estimates are sufficient for the bootstrap
to prove the main theorem in the end of the introduction.





\section{Variable coefficient Cubic Normal Forms}\label{cubicnormalforms}

According the last section we may form the quantity $w= v - w_1$, where $w_1$ is defined in \eqref{w11}, and the equation for $w$ is of the form:
\begin{equation}
    (\partial_\rho^2 + D_y^2 + 1)w  \ = \ \frac{1}{\rho}\, \beta(\rho y) v^3
    + \frac{\beta_0}{\rho}v^3 -\mathcal{E}_{quad} + \mathcal{R} \ , \label{1st_norm_prod}
\end{equation}
where $\mathcal{R}$ is given by \eqref{eq:Remainder} and
$\mathcal{E}_{quad}$ is given by \eqref{quad_error}.
To further clean up this expression, we
make the following two definitions:
\begin{equation}
        \mathcal{E}_{grand} \ = \
        \frac{\beta_0}{\rho}v^3+\frac{1}{\rho}\, \beta(\rho y) \big(v^3-(v-w_1)^3\big) -\mathcal{E}_{quad} + \mathcal{R}
        \ . \label{first_grand_error}
\end{equation}
With these definitions, we may rewrite \eqref{1st_norm_prod} as:
\begin{equation}
    (\partial_\rho^2 + D_y^2 + 1)w \ = \ \frac{1}{\rho}\, \beta(\rho y) w^3
    + \mathcal{E}_{grand} \ . \label{1st_norm_prod_red}
\end{equation}
Before continuing on with the analysis of the cubic term on the RHS
of this last line, we pause to review the error estimates which we
have shown thus far:\\
\begin{lem}
\begin{align}
\|\mathcal{E}_{grand}\|_{H^1}&\lesssim
 \big(1+\|(v,\dot{v})\|_{B^\infty_\rho}^4\big)\, \frac{1}{\rho}\,
 \|(v,\dot{v})\|_{B^\infty_\rho}^2\|(v,\dot{v})\|_{H^1}\\
\|\dot{\mathcal{E}}_{grand}\|_{H^1}&\lesssim
 \big(1+\|(v,\dot{v})\|_{B^\infty_\rho}^6\big)\, \frac{1}{\rho}\,
 \|(v,\dot{v})\|_{B^\infty_\rho}^2 \|(v,\dot{v})\|_{H^1}\\
 \|\mathcal{E}_{grand}\|_{B_\rho^\infty}&\lesssim
 \big(1+\|(v,\dot{v})\|_{B^\infty_\rho}^4\big)\, \frac{1}{\rho}\,
 \|(v,\dot{v})\|_{B^\infty_\rho}^3\\
\|\dot{\mathcal{E}}_{grand}\|_{B_\rho^\infty}&\lesssim
 \big(1+\|(v,\dot{v})\|_{B^\infty_\rho}^6\big)\, \frac{1}{\rho}\,
 \|(v,\dot{v})\|_{B^\infty_\rho}^3
\end{align}
 \end{lem}
 \begin{proof}
 These estimate for ${\mathcal{E}}_{quad}$ where previously proven. The only new
 possibly problematic term is
 \begin{equation}
 \frac{1}{\rho}\, \beta(\rho y) \big(v^3-(v-w_1)^3\big)=
 \frac{1}{\rho}\, \beta(\rho y) \big( 3v^2+w_1^2 +3w_1 v)w_1
 \end{equation}
 but the extra $\rho^{-1/2}$ decay in $w_1$ in the previous section compensates
 for the loss of $\rho$ when taking the $H^1$ norm of $\beta(\rho y)$.
 \end{proof}

\begin{rem} We remark that for the constant coefficient cubic
$(\partial_\rho^2+D_y^2+1)v=\beta_0 v^3/\rho$ we do not need normal forms
to prove global existence, see \cite{L-S2}. The reason it is needed here is
to get an estimate for the $H^1$ norm to remove the problematic term
that otherwise would be there $(\partial_y \beta(\rho y) )\rho^{-1} v^3
=\beta^{\,\prime}(\rho y) v^3$ which do not decay enough in $L^2$.
These estimates are sufficient for the bootstrap to prove the main theorem in the end
of the introduction.
\end{rem}

We furthermore note that its only low frequencies that can cause problem:
\begin{lem}
\begin{equation}
\mathcal{E}_{high} =\frac{1}{\rho} \beta(\rho y) \big( w^3 -(P_{\leq \rho} w)^3\big)
\end{equation}
satisfies
\begin{equation}
\|\mathcal{E}_{high}\|_{H^1}\lesssim\frac{1}{\rho}\,
 \big(1+\|(v,\dot{v})\|_{B^\infty_\rho}^4\big)\,
 \|(v,\dot{v})\|_{B^\infty_\rho}^2\|(v,\dot{v})\|_{H^1}
\end{equation}
\end{lem}
\begin{proof} The extra decay comes from that
$\|P_{\geq \rho} w\|_{L^\infty}\leq \rho^{-1/2} \|w\|_{H^1}$ and this decay exactly
compensates for the loss of $\rho^{1/2}$ when taking the $H^1$ norm of
$\beta(\rho y)$.
\end{proof}

We will therefore attempt to find and subtract off a normal form $w$ such that
\begin{equation}
    (\partial_\rho^2
    - \frac{1}{\rho^2}\partial_y^2 + 1 )w_2=
     \frac{\beta\big(\rho y \big)}{\rho}(P_{\leq \rho} w)^3 +\mathcal{E}_{cubic}
     \label{eq:variablecubicnormalformw}
\end{equation}
modulo an error $
\mathcal{E}_{cubic}$ that satisfy the estimate in the previous lemma.

In general we decompose $\beta$ into a sum of a three functions such that the Fourier transform
of the first one vanishes in a neighborhood
of the origin and in neighborhoods of $\pm \sqrt{8}$, and the support of the Fourier transform of the second one is contained in a small neighborhood of the origin and
the support of the transform of the third one is contained in a small neighborhoods of $\pm\sqrt{8}$. The normal form $w_2$ above is then obtained by adding up the normal forms obtained for each of the three functions in the decomposition. The case when $\widehat{\beta}(\xi)$ is vanishing in neighborhood of $0$ and in a neighborhoods of $\pm \sqrt{8}$ was dealt with in the introduction and the argument in that case is anyway a special case of the argument below so we
will deal with the remaining two cases.

\subsection{The case when $\widehat{\beta}(\xi)$ is supported in a neighborhood of $0$}

If $\widehat{\beta}$ is supported near $0$ then the $f_i$ are actually growing,
although the derivatives are bounded. In this case we have to modify the approach,
taking into account that the solution decays for large frequencies. Let $f_i[\beta]$, $i=0,2$
be the functionals defined by solving the system in the introduction depending
on $\beta$, see Definition \ref{functional}.
Let us make frequency decomposition
\begin{equation}
\beta=\sum_{j=0}^{\infty}\beta_{j},\quad \text{where}\quad
\beta_j=P_{2^{-j}} \beta,\qquad j\geq 0
\end{equation}
is the projection onto a dyadic region of frequencies $\sim 2^{-j}$.
We now define
\begin{equation}
w_{2,j}=\frac{1}{\rho} \sum_{i=0,2} f_i[\beta_j] \, F_i(w_{j},\dot{w}_j),\quad
\text{where},\quad w_j=P_{\leq \rho/2^j} w,\qquad
\widehat{\beta}_j(\xi)=\chi_1(2^j \xi) \widehat{\beta}(\xi),
\end{equation}
where $\chi_1$ is supported in a neighborhood of $1$,
and $F_i(v,\dot{v})=v^{3-i}\dot{v}^i$.
Then $w_{2,j}$ solves
\begin{equation}
(\Box_\mathcal{H} + 1 )w_{2,j}=\rho^{-1}\beta_j(\rho y) F_0(w_j)+{\mathcal E}_{cubic,j}
\label{jerror1}
\end{equation}
We know from the argument in the introduction that for each $j$
the error $\mathcal{E}_{cubic,j}$ satisfy the estimate \eqref{eq:errorest0}
but we need to be able to sum the errors up.
We will sum only over values of $j$ for which $2^j\leq \rho^{1/2}$,
and form
\begin{equation}
w_2=\sum_{j=1}^{\infty} \chi_0(2^j/\rho^{1/2})\, w_{2,j}
\end{equation}
where $\chi_0\in C_0^\infty$ is $1$ in a neighborhood of the origin.
The remainder satisfies our bound \eqref{eq:errorest0} since
\begin{lem}
\begin{equation}
\| \rho^{-1} P_{\leq \rho^{\sigma}} \beta(\rho y) \, F_0(w)\|_{H^1}
\lesssim \frac{1}{\rho}\|w\|_{L^\infty}^2 \|w\|_{H^1} ,\qquad \sigma\leq 3/4
\end{equation}
\end{lem}
\begin{proof}
It is easy to see that $P_{\leq \rho^{\sigma}} \beta(\rho y)
=\int e^{i y\rho \zeta} \chi(\zeta \rho^{1-\sigma}) \hat{\beta}(\zeta)\, d\zeta$
satisfies
\begin{equation}
|\partial_y^k P_{\leq \rho^{\sigma}} \beta(\rho y) |\lesssim \rho^{(k+1)\sigma-1},\quad
\|\partial_y^k P_{\leq \rho^{\sigma}} \beta(\rho y) \|_{L^2_y}
\lesssim \rho^{(k+1)\sigma-3/2}.
\end{equation}
Hence the quantity in the lemma can be bounded by
\begin{equation}
\rho^{-1} \|P_{\leq \rho^{\sigma}} \beta(\rho y)\|_{H^1} \|F_0(w)\|_{L^\infty}
+\rho^{-1} \|P_{\leq \rho^{\sigma}} \beta(\rho y)\|_{L^\infty} \|F_0(w)\|_{H^1}
\lesssim \rho^{-1} \|F_0(w)\|_{H^1}.
\end{equation}
\end{proof}

We will show:
\begin{prop}\label{first0res} Suppose $\widehat{\beta}$ is supported in a neighborhood of the origin.
Then with $w_2$ as above we have
\begin{equation}
(\Box_\mathcal{H} + 1 )w_{2}=\rho^{-1}\beta(\rho y) F_0(w)+{\mathcal E}_{cubic}
\label{error}
\end{equation}
where for some large $N$
\begin{equation}
\|\mathcal{E}_{cubic}\|_{H^1}\lesssim
\frac{1}{\rho}
\big(1+\|(v,\dot{v})\|_{B_\rho^\infty}\big)^N
\|(v,\dot{v})\|_{B_\rho^\infty}^2\|(v,\dot{v})\|_{H^1}.
\end{equation}
Moreover
\begin{equation}
\|(w_2,\dot{w}_2,D_y w_2)\|_{H^1}\lesssim \frac{1}{\rho^{1/4} }
\|(w,\dot{w})\|_{L^\infty}^2 \|(w,\dot{w})\|_{H^1}
\end{equation}
\end{prop}

What makes this argument work is that $\beta_j$ will become small:
\begin{lem}\label{lem:betaest}
\begin{equation}
\|\beta_j(\rho y)\|_{L^\infty}\lesssim 1/2^j,\qquad
\|\beta_j(\rho y)\|_{\dot{H}^n}\lesssim (\rho / 2^{j})^{n-1/2}/\,2^j,\quad n=0,1.
\end{equation}
\end{lem}
\begin{proof} First by scaling we may assume that $\rho=1$.
Since $\widehat{\beta}_j(\xi)=\chi_1(2^j \xi) \widehat{\beta}(\xi)$ it follows the
$L^\infty$ bound follows from
\begin{equation}
\|\beta_j\|_{L^\infty}\lesssim
\|\widehat{\beta}_j\|_{L^1}\lesssim \int|\chi_1(2^j \xi)|\, d\xi\lesssim 2^{-j}
\end{equation}
and the $L^2$ bound from
\begin{equation}
\|\partial_y^n \beta_j\|_{L^2}^2\lesssim
\|\xi^n\widehat{\beta}_j\|_{L^2}^2\lesssim \int|\xi^n\chi_1(2^j \xi)|^2 \, d\xi
\lesssim 2^{-j}2^{-n 2j}
\end{equation}
\end{proof}

Therefore the remainder is summable;
\begin{lem}
\begin{equation}
\| \beta_j(\rho y) F_0(w_j)-\beta_j(\rho y) F_0(w)\|_{\dot{H}^1}\lesssim
2^{-j} \|w\|_{\dot{H}^1}\|w\|_{L^\infty}^2.
\end{equation}
\end{lem}
\begin{proof} This is bounded by
\begin{equation}
\| \beta_j(\rho y)\|_{\dot{H}^1}\| F_0(w_j)-F_0(w)\|_{L^\infty}
+\| \beta_j(\rho y)\|_{L^\infty}\| F_0(w_j)-F_0(w)\|_{\dot{H}^1}
\end{equation}
which can be bounded by the previous lemma and
\begin{equation}
\|w-w_j\|_{L^\infty}
\lesssim (\rho 2^{-j})^{-1/2} \|w\|_{\dot{H}^1}.
\end{equation}
\end{proof}

Therefore only remains to look at the error terms in the approximation
\eqref{jerror1}. We first note that:
Its easy to see that
\begin{lem}\label{lem:fofbetaest} The support of $\widehat{f_i[\beta]}$ is contained in the support of
$\widehat{\beta}_j$. Moreover, we have
\begin{equation}
\|(\rho y)^\ell f_i[\beta_j]^{(k)}(\rho y)\|_{\dot{H}^n}
\lesssim (\rho 2^{-j})^{n-1/2} 2^{-j(k-\ell-1)}.
\end{equation}
and
\begin{equation}
\|(\rho y)^\ell f_i[\beta_j]^{(k)}(\rho y)\|_{L^\infty}\lesssim 2^{-j(k-\ell-1)}
\end{equation}
\end{lem}
\begin{proof} First, we note that the $\rho$ factor can be removed by a scaling
so we may assume that $\rho=1$. The functionals $f_i[\beta_j]$ are linear
combinations of $g_0$ and $g_2$ satisfying
\begin{equation}
-\xi^2 \widehat{g}_0(\xi)=-3\, \widehat{\beta}_j(\xi),\qquad
(-\xi^2+8)\widehat{g}_2(\xi)=-\widehat{\beta}_j(\xi)
\end{equation}
where $\widehat{\beta}_j(\xi)=\chi_1(2^j \xi) \widehat{\beta}(\xi)$,
where the support of $\widehat{\beta}$ is bounded away from $\pm \sqrt{8}$.
We have
\begin{equation}
\|z^\ell g_0^{(k)}(z)\|_{L^\infty}\lesssim
\|\partial_\xi^\ell \xi^k \widehat{g}_0(\xi)\|_{L_\xi^1}\lesssim
2^{j\ell} 2^{-j(k-2)} 2^{-j}
\end{equation}
and
\begin{equation}
\|z^\ell g_0^{(k)}(z)\|_{L^2}\lesssim
\|\partial_\xi^\ell \xi^k \widehat{g}_0(\xi)\|_{L_\xi^2}\lesssim
2^{j\ell} 2^{-j(k-2)} 2^{-j/2}
\end{equation}
and the estimates for $g_2$ are better.
\end{proof}

\begin{lem} We have
\begin{multline}
\frac{1}{\rho}\big(\Box_{\mathcal H} +1\big)
\sum_{i=0,2} f_i[\beta_j] \, F_i(w_j,\dot{w}_j)\\
=\frac{1}{\rho}\sum_{i=0,2} \big(\Box_{\mathcal H} +1\big)f_i[\beta_j] \,\, F_i
+2\partial_\rho f_i[\beta_j] \,\, F_i^1(w_j,\dot{w}_j)
+f_i[\beta_j]F_i^2(w_j,\dot{w}_j) +\mathcal{E}_{1,j}+\mathcal{E}_{2,j}.
\end{multline}
Here
\begin{equation}
\mathcal{E}_{1,j}=\frac{1}{\rho}
\Big(2\partial_\rho f_i[\beta_j] \,\, (\partial_\rho F_i-F_i^1) +f_i[\beta_j]
(\partial_\rho^2 F_i-F_i^2)\Big)
\end{equation}
where
\begin{align}
\partial_\rho F_i-F_i^1&=G_i^1\big(\ddot{w}_j+w_j)\\
\partial_\rho^2 F_i-F_i^2&=G_i^2\big(\ddot{w}_j+w_j)+
G_i^3\big(\ddot{w}_j+w_j)^2+
G_i^4\partial_\rho\big(\ddot{w}_j+w_j)
\end{align}
where $G^k_i=G_i^k(w_j,\dot{w}_j)$ are polynomials such that all terms are cubic,
and
\begin{equation}
\mathcal{E}_{2,j}=\frac{1}{\rho}\sum_{i=0,2}
2D_y f_i[\beta_j] \,\,D_y F_i(w_j,\dot{w}_j)
+f_i[\beta_j] \,\,D_y^2 F_i(w_j,\dot{w}_j).
\end{equation}
\end{lem}
Here we have the estimates
\begin{lem}
We assume as before that $ \rho \leq 2^{j}.$
\begin{equation}
\|\mathcal{E}_{2,j}\|_{H^1}\lesssim
\frac{2^{-j}}{\rho}\Big(1+\frac{2^{2j}}{\rho}+\frac{2^{4j}}{\rho^2}\Big) \|w_j\|_{L^\infty}^2\|w_j\|_{H^1}.
\end{equation}
\end{lem}
\begin{proof}
Due to Lemma \ref{lem:fofbetaest}, the main term under the sum is the second one.
 Since $D_y=\rho^{-1}\partial_y$ is bounded by $2^j/\rho$ (in
$L^2$) acting on $w_j$, the $H^1$ norm of this is bounded by
\begin{equation}
\sum_{n=0}^2 \| D_y^n f_i[\beta_j]\|_{L^\infty} \|D_y^{2-n} F_i\|_{H^1} \lesssim
\sum_{n=0}^2 2^{-j(n-1)}\frac{2^{j(2-n)}}{\rho^{2-n}}
\|F_i\|_{H^1}.
\end{equation}
\end{proof}
and
\begin{lem} If $2^j\leq \rho^{1/2}$ then for some $N$:
\begin{equation}
\|\mathcal{E}_{1,j}\|_{H^1}\lesssim \frac{2^{j/2}}{\rho^{1/2}}
\big(1+\|(v,\dot{v})\|_{B_\rho^\infty}\big)^N
\|(v,\dot{v})\|_{B_\rho^\infty}^2\frac{1}{\rho}\|(v,\dot{v})\|_{H^1}.
\end{equation}
\end{lem}
\begin{proof} We have
\begin{equation}
\ddot{w}_j+w_j=D_y^2 w_j+ P_{\leq \rho/2^j}
\Big(\frac{1}{\rho}\beta w^3+\mathcal{E}_{grand}\Big)
+\big[\partial_\rho^2,P_{\leq \rho/2^j}\big] w.\label{eq:kgerrorintr7}
\end{equation}
Here
\begin{equation}
\| \big[\partial_\rho^2,P_{\leq \rho/2^j}\big] w\|_{H^1} \lesssim
\frac{1}{\rho} \|P_{\sim \rho/2^j} \dot{w}\|_{H^1}
+\frac{1}{\rho^2} \|P_{\sim \rho/2^j} w\|_{H^1}.
\end{equation}
Hence
\begin{equation}
\|\ddot{w}_j+w_j\|_{H^1}\lesssim \frac{2^{2j}}{\rho^2}\|w\|_{H^1}
 +\frac{1}{\rho} \|w\|_{L^\infty}^2 \|w\|_{H^1}
 +\frac{1}{\rho}\|(w,\dot{w})\|_{H^1}+\|\mathcal{E}_{grand}\|_{H^1}.
\end{equation}
To estimate the $\rho$ derivative we must use this estimate again since
\begin{equation}
\|\partial_\rho \big[\partial_\rho^2,P_{\leq \rho/2^j}\big] w\|_{H^1} \lesssim
\frac{1}{\rho} \|P_{\sim \rho/2^j} \ddot{w}\|_{H^1}
+\frac{1}{\rho^2} \|P_{\sim \rho/2^j} \dot{w}\|_{H^1}
+\frac{1}{\rho^3} \|P_{\sim \rho/2^j} w\|_{H^1}.
\end{equation}
Hence if $2^j \leq \rho^{1/2}$ and $k=0,1$ we get for some $N$
\begin{equation}
\|\partial_\rho^k\big(\ddot{w}_j+w_j\big)\|_{H^1}\leq
\big(1+\|(v,\dot{v})\|_{B_\rho^\infty}\big)^N \frac{1}{\rho}\|(v,\dot{v})\|_{H^1}.
\end{equation}
Similarly
\begin{equation}
\|\partial_\rho^k\big(\ddot{w}_j+w_j\big)\|_{B_\rho^\infty}\leq
\big(1+\|(v,\dot{v})\|_{B_\rho^\infty}\big)^N \frac{1}{\rho}\|(v,\dot{v})\|_{B_\rho^\infty}.
\end{equation}
\end{proof}

The error terms involving $\rho$ derivatives
that we have to bound are the $H^1$ norm of
\begin{equation}
\mathcal{E}_{3,j}=\frac{1}{\rho}\big(\partial_\rho^2 f_i \,\,   F_i^0+2\partial_\rho f_i \,\, F_i^1\big)
\end{equation}
plus terms that decay faster in $\rho$.
\begin{lem} We have
\begin{equation}
\|\mathcal{E}_{3,j}\|_{H^1}\lesssim\frac{1}{\rho}
(2^j/\rho)^{1/2}\|(w_j,\dot{w}_j)\|_{L^\infty}^3
+\big((2^j/\rho)+2^{-j}\big) \|(w_j,\dot{w}_j)\|_{L^\infty}^2 \|(w,\dot{w})\|_{H^1}.
\end{equation}
\end{lem}
\begin{proof}
Taking the $y$ derivative of this we see that we must bound
the $L^2$ norm of
\begin{equation}
\rho^{-k}\Big(k(\rho y)^{k-1} f_i[\beta_j]^{(k)}(\rho y)
+(\rho y)^k f_i[\beta_j]^{(k+1)}(\rho y) \Big)F_i^{1-k}(w_j,\dot{w}_j), \quad k=1,2
\label{error1}
\end{equation}
and
\begin{equation}
\rho^{-1-k}(\rho y)^k f_i[\beta_j]^{(k)}(\rho y) \partial_y F_i^{1-k}(w_j,\dot{w}_j),
 \quad k=1,2.\label{error2}
\end{equation}

The error terms involving the $y$ derivatives are
\begin{equation}
 f_i[\beta_j]^{(2-k)}(\rho y) \,
\frac{1}{\rho^{1+k}}\partial_y^{k}F_i
(w_j,\dot{w}_j),\quad k=1,2
\end{equation}
and if take the $H^1$ norm we see that we must estimate the following terms in $L^2$:
\begin{equation}
 f_i[\beta_j]^{(2-k)}(\rho y) \,
\frac{1}{\rho^{1+k}}\partial_y^{k+1}F_i
(w_j,\dot{w}_j),\quad k=0,1,2.\label{error3}
\end{equation}

Taking the first factor of \eqref{error1} in $L^2$ we get the bound
$\rho^{-1}(2^j/\rho)^{1/2}\|(w_j,\dot{w}_j)\|_{L^\infty}^3$.
Taking the second factor of \eqref{error2} in $L^2$ gives the bound
$\rho^{-1}(2^j/\rho)\|(w_j,\dot{w}_j)\|_{L^\infty}^2 \|(w,\dot{w})\|_{H^1}$
Taking the second factor or \eqref{error3} in $L^2$ we get the bound
\begin{equation}
\rho^{-1}2^{-j} \| F_i\|_{\dot{H}^1}
\end{equation}
since the operator $\partial_y/\rho$ is bounded by $2^{-j}$ acting on functions
who's frequencies are bounded by $C\rho 2^{-j}$.
\end{proof}
The final type of error is when derivatives fall on the factor $\rho^{-1}$:
\begin{equation}
\mathcal{E}_{4,j}=\frac{2}{\rho^2} \sum_{i=0}^3
\Big( \partial_\rho f_i\, F_i +f_i\partial_r F_i\Big)
+\frac{2}{\rho^3} \sum_{i=0}^3 f_i\, F_i
\end{equation}
which clearly has been control by previous arguments.
Summing up, we have proven:
\begin{prop}\label{last0res} We have
\begin{multline}
\big(\Box_{\mathcal H} +1\big)\frac{1}{\rho}
\sum_{i=0,2} f_i[\beta_j] \, F_i(w_j,\dot{w}_j)\\
=\frac{1}{\rho}\sum_{i=0,2} \big(f_i[\beta_j]-f_i[\beta_j]^{(2)}\big)\,\, F_i(w_j,\dot{w}_j)
+f_i[\beta_j]F_i^2(w_j,\dot{w}_j)+\mathcal{E}_{j},
\end{multline}
where for some $N$
\begin{equation}
\sum_{2^j\leq \rho^{1/2}}\|\mathcal{E}_{j}\|_{H^1}\lesssim
\frac{1}{\rho}
\big(1+\|(v,\dot{v})\|_{B_\rho^\infty}\big)^N
\|(v,\dot{v})\|_{B_\rho^\infty}^2\|(v,\dot{v})\|_{H^1}.
\end{equation}
\end{prop}
Since the functionals $f_i[\beta_j]$ where chosen so that
\begin{equation}
\sum_{i=0,2} \big(f_i[\beta_j](\rho y)-f_i[\beta_j]^{(2)}(\rho y)\big)
\,\, F_i(w_j,\dot{w}_j)
+f_i[\beta_j](\rho y) F_i^2(w_j,\dot{w}_j)=\beta_j(\rho y) F_0(w_j,\dot{w}_j)
\end{equation}
Proposition \ref{first0res} follows from Proposition \ref{last0res} apart from the
estimates for $w_2$ itself. For this note that $w_2$ has frequencies at most $4\rho$ so $D_y$ is a bounded operator on $w_2$.

\subsection{The case when $\widehat{\beta}(\xi)$ is supported in a neighborhood of
$\pm\sqrt{8}$}
If $\widehat{\beta}$ is supported near $\pm \sqrt{8}$ then the best thing we can say is that $f_0(\rho y)$ and $f_2(\rho y)$ solving the above system are only bounded and more importantly their derivatives no longer decay. Therefore we no longer can assume that their derivative with respect to $\rho$ are decaying
when $|y|$ is bounded from below. In our situation, because $\beta$ is fast decaying, this may be overcome by multiplying by a cutoff $\chi(\rho^{a} y)$ for some $0<a<1$, where $\chi$ is $1$ in a neighborhood of the origin. We will however instead take a different approach and obtain a new more general variable coefficient normal form transformation that is a better approximation as long as $|y|$ is bounded from above. This is obtained by taking into account the $\rho$ derivatives of the system to obtain:
\begin{align}
        \Box_\mathcal{H} f_0 - 2f_0 -2\partial_\rho f_1 + 2f_2 \ &= \
        \beta(\rho y) \ , \notag\\
        \Box_\mathcal{H} f_1 - 6f_1 + 6\partial_\rho f_0
        -4\partial_\rho f_2 + 6 f_3 \ &= \ 0 \ , \notag\\
        \Box_\mathcal{H} f_2 - 6f_2 + 6f_0 +
        4\partial_\rho f_1 - 6\partial_\rho f_3 \ &= \ 0 \ , \notag\\
        \Box_\mathcal{H} f_3 - 2f_3 + 2 f_1 + 2\partial_\rho f_2 \ &= \
        0 \ . \notag
\end{align}
As before introducing
with $g_0=3f_0+f_2$, $g_2=f_0-f_2$, $g_1=f_1+3f_3$, $g_3=f_1-f_3$:
we get
\begin{align}
        \Box_\mathcal{H} g_0 -2\partial_\rho g_1 \ &= \ 3\beta  \ , \notag\\
        \Box_\mathcal{H} g_1 + 2\partial_\rho g_0 \ &= \ 0 \ ,
        \notag\\
        \Box_\mathcal{H} g_2 - 8g_2 - 6\partial_\rho g_3 \ &= \ \beta \ , \notag\\
        \Box_\mathcal{H} g_3 - 8g_3 + 6\partial_\rho g_2 \ &= \ 0 \ . \notag
\end{align}

Complexifying the above system we get
\begin{equation}
K_1=(g_0+ig_1)e^{i\rho}/3,\qquad K_3=(g_2+ig_3)e^{3i\rho}\label{eq:geq}
\end{equation}
gives the system
\begin{align}
        (\Box_\mathcal{H} + 1) K_1 \ &= \  e^{i\rho}\beta(\rho y) \ ,
        &(\hbox{$0$ resonance equation})
        \ , \label{0_res_eq 2}\\
        (\Box_\mathcal{H} + 1) K_3 \ &= \  e^{3i\rho}\beta(\rho y) \
        ,
        &(\hbox{$\pm\sqrt{8}$ resonance equation})
        \ . \label{sqrt8_res_eq 2}
\end{align}
We note that we only have to solve these equations asymptotically,
which can be done with the stationary phase method.
We hence want to an asymptotic solution $K_n[\beta]$ that solves
\begin{equation}
(\Box_\mathcal{H} + 1) K_n=e^{in\rho}\beta(x)+{\mathcal E}_{K_n}
\end{equation}
where the error ${\mathcal E}_n[\beta]$ decays sufficiently fast.
Here the functional $K_n[\beta]$ and error ${\mathcal E}_n[\beta]$
satisfy the same kind of estimates
as we had for the functionals $f_i[\beta]$ before, i.e. if $\beta$ is smooth
and fast decaying we have
\begin{align}
    |\partial_\rho^k D_y^l\chi K_i| \ &\leqslant \ C_{k,l} \ ,
     &|\partial_\rho^k D_y^l\chi \mathcal{E}_{K_i}| \
     &\leqslant \ \rho^{-1}C_{k,l} \ , \label{med_K_ests12}\\
     \lp{(\chi K_i,D_y\chi K_i,\partial_\rho \chi K_i }{B_\rho^\infty} \ &\lesssim \ 1 \ . \label{med_K_ests22}
\end{align}
Here $\chi=\chi(y)$ is a smooth bump function in the $y$ variable.
This will be proven in the next section.
Assuming that this is true we now define $g_i[\beta]$ by \eqref{eq:geq} and $f_i[\beta]$
from $g_i[\beta]$. In that way we get approximate solutions
\begin{align}
        \Box_\mathcal{H} g_0 -2\partial_\rho g_1 \ &= \ 3\beta+\mathcal{E}_{g,0}  \ , \notag\\
        \Box_\mathcal{H} g_1 + 2\partial_\rho g_0 \ &= \ \mathcal{E}_{g,1} \ ,
        \notag\\
        \Box_\mathcal{H} g_2 - 8g_2 - 6\partial_\rho g_3 \ &= \ \beta +\mathcal{E}_{g,2}\ , \notag\\
        \Box_\mathcal{H} g_3 - 8g_3 + 6\partial_\rho g_2 \ &= \ \mathcal{E}_{g,3} \ . \notag
\end{align}
and
\begin{align}
        \Box_\mathcal{H} f_0 - 2f_0 -2\partial_\rho f_1 + 2f_2 \ &= \
        \beta(\rho y)+\mathcal{E}_{f,0} \ , \notag\\
        \Box_\mathcal{H} f_1 - 6f_1 + 6\partial_\rho f_0
        -4\partial_\rho f_2 + 6 f_3 \ &= \mathcal{E}_{f,1} \ , \notag\\
        \Box_\mathcal{H} f_2 - 6f_2 + 6f_0 +
        4\partial_\rho f_1 - 6\partial_\rho f_3 \ &= \ \mathcal{E}_{f,2} \ , \notag\\
        \Box_\mathcal{H} f_3 - 2f_3 + 2 f_1 + 2\partial_\rho f_2 \ &= \
        \mathcal{E}_{f,3} \ . \notag
\end{align}
where $\mathcal{E}_{g,i}$ and $\mathcal{E}_{f,i}$ satisfy the same kind of estimates
as $\mathcal{E}_{K_i}$. However since we assume that $\widehat{\beta}$ vanishes in a
neighborhood of the origin $g_0[\beta]$ and $g_1[\beta]$ can be defined as before
by $\widehat{g}_0(\xi) =\widehat{\beta}(\xi)/\xi^2$ and $\widehat{g}_1=0$.

With $f_i$ defined as above we see that the only additional error we have to
bound in $H^1$ is
\begin{equation}
\frac{1}{\rho} \|\sum_{i=0}^3 \mathcal{E}_{f,i} F_i \|_{H^1}
\lesssim \frac{1}{\rho} \sum_{i=0}^3 \|\mathcal{E}_{f,i}\|_{C^1} \|F_i\|_{H^1}
\lesssim \frac{1}{\rho} \sum_{i=0}^3 \|F_i\|_{H^1}\lesssim
\|(w,\dot{w})\|_{L^\infty}^2\|(w,\dot{w})\|_{H^1}.
\end{equation}


\section{A resonant parametrics construction}

We need to asymptotically solve the complex
equations:
\begin{align}
        (\Box_\mathcal{H} + 1) K_1 \ &\sim\  e^{i\rho}\beta(\rho y) \ ,
        &(\hbox{$0$ resonance equation})
        \ , \label{0_res_eq}\\
        (\Box_\mathcal{H} + 1) K_3 \ &\sim \  e^{3i\rho}\beta(\rho y) \
        ,
        &(\hbox{$\pm\sqrt{8}$ resonance equation})
        \ . \label{sqrt8_res_eq}
\end{align}
The reader should keep in mind that this system is a more detailed
replacement for the system above,when the derivative with respect to $\rho$ was removed; it will allow us to
precisely track the dispersive behavior of the $\pm\sqrt{8}$
resonances.\\

We will only need to solve the system \eqref{0_res_eq}--\eqref{sqrt8_res_eq}
asymptotically. To do this, we now compute an approximate fundamental solution
to the equation $\Box_{\mathcal{H}}+1=\partial_\rho^2 - \rho^{-2}\partial_y^2 +1$
which becomes more and more accurate as $\rho\to\infty$. First, we define the approximate phase:
\begin{equation}
    \psi(\rho,s;\xi) \ = \ \int_s^\rho \Big(\frac{1}{\zeta^2}\xi^2
    + 1\Big)^\frac{1}{2} d\zeta \ , \label{phase}
\end{equation}
and then use this to define the integral kernel:
\begin{equation}
    U(\rho,s;\xi) \ = \ \frac{\sin\big(\psi(\rho,s;\xi)\big)}{\big(\frac{1}{s^2}\xi^2 + 1\big)^\frac{1}{2}}
     \ . \notag
\end{equation}
Finally, for a source term $H(\rho,y)$ we form the approximate Duhamel integral:
\begin{align}
    u(\rho,y) \ &= \ \frac{1}{2\pi}\ \int_{\rho_0}^\rho  U(\rho,s;\xi) e^{iy\xi} \widehat{H}(s,\xi) \ d\xi ds
     \ , \label{approx_fund}\\
      &= \  \int_{\rho_0}^\rho  U(\rho,s;\partial_y) {H}(s,y) \  ds
     \ , \notag
\end{align}
where $\widehat{H}$ is the Fourier transform \eqref{FT} of $H$ with
respect to $y$. This kernel attempts to construct a solution to
$(\Box_\mathcal{H}+1)u=H$ with vanishing Cauchy data when
$\rho=\rho_0$. However, this solution in not exact
as can easily be seen. We will estimate the error terms shortly.\\

We now define the normal forms coefficients $K_i$ using the formula \eqref{approx_fund}:
\begin{align}
    K_1(\rho,y) \ &=  \ \int_{1}^\rho  e^{is}U(\rho,s;\partial_y)  \beta\, \chi_1\big(\frac{s}{\rho}\big) \  ds
     \ , \label{K1_def}\\
     K_3(\rho,y) \ &=  \ \int_{1}^\rho  e^{3is}U(\rho,s;\partial_y) \beta\, \chi_1\big(\frac{s}{\rho}\big) \  ds
     \ . \label{K2_def}
\end{align}
Here $\chi_1$ is a smooth bump function with $\chi_1(\zeta)\equiv 1$ when $\frac{1}{2} < |\zeta| < 2$, and vanishing close
to $\zeta=0$. With these definitions, we form the error terms:
\begin{align}
    (\Box_\mathcal{H} + 1) K_1 -  e^{i\rho}\beta(\rho y) \ &= \ \mathcal{E}_{K_1} \ ,
    &(\Box_\mathcal{H} + 1) K_3 -   e^{3i\rho}\beta(\rho y) \ &= \ \mathcal{E}_{K_3} \ . \label{med_kern_error_line}
\end{align}
The estimates we will need are the following:\\

\begin{prop}[Estimates for Medium Frequency Cubic NF Coefficients]
Consider the functions $K_1$, $K_3$, and $\mathcal{E}_{K_1}$,
$\mathcal{E}_{K_3}$ defined above, where $\beta$ is a Schwartz
class function with $\widetilde{\beta}(\xi)\equiv 0$ for $|\xi|\ll
1$ and $1\ll|\xi|$. Then the following pointwise estimates hold:
\begin{align}
    |\partial_\rho^k D_y^l\chi K_i| \ &\leqslant \ C_{k,l} \ ,
     &|\partial_\rho^k D_y^l\chi \mathcal{E}_{K_i}| \
     &\leqslant \ \rho^{-1}C_{k,l} \ , \label{med_K_ests1}\\
     \lp{(\chi K_i,D_y\chi K_i,\partial_\rho \chi K_i }{B_\rho^\infty} \ &\lesssim \ 1 \ . \label{med_K_ests2}
\end{align}
Here $\chi=\chi(y)$ is a smooth bump function in the $y$ variable.
\end{prop}\ret

\noindent These estimates will in turn easily follow from:\\

\begin{lem}[Stationary Phase Estimates]
Consider the oscillatory integrals:
\begin{equation}
    I_k^{\pm}(\rho,y) \ = \ \rho^{-1} \int_{0}^\rho\int_\mathbb{R}\
    e^{\pm i \psi(\rho,s;\xi)} e^{iy\xi}e^{iks}  \mathfrak{m}(\rho,s;\xi)
     \chi_1\big(\frac{s}{\rho}\big) \ d\xi ds \ , \label{I_form}
\end{equation}
where $\mathfrak{m}(\rho,s;\xi)$ is a symbol such that $\mathfrak{m}(\rho,s;\xi)\equiv 0$
for $s^{-1}|\xi|\ll 1$ and $1\ll s^{-1}|\xi|$, and such that:
\begin{equation}
    |\partial_s^i \partial_\rho^j \partial_\xi^l  \mathfrak{m}| \ \leqslant \
    \rho^{-i-j-l}C_{i,j,l} \ , \label{symbol_bounds}
\end{equation}
when $s\sim\rho$. The phase $\psi$ is from \eqref{phase}, and the cutoff $\chi_1$
is defined as in the previous paragraph. Then for $k=1,3$ the following uniform estimate holds:
\begin{equation}
    |\chi(y) I_k^{\pm}(\rho,y)|\ \lesssim \ 1 \ . \label{I_est}
\end{equation}
\end{lem}\ret

\begin{proof}[Proof of estimate \eqref{I_est}]
We will treat the two cases $k=1$ and $k=3$ separately. In the case where $k=1$, the total phase is
non-stationary on the $\xi$ support of $\mathfrak{m}(\rho,s;\xi)$. We have that:
\begin{equation}
    \partial_s\big[\pm \psi(\rho,s;\xi) + s\big]
    \ = \ \mp(\frac{1}{s^2}\xi^2 +1)^\frac{1}{2} + 1 \ \sim \ 1 , \ \notag
\end{equation}
as long as $\frac{1}{s^2}\xi^2\sim 1$, which holds owing to the  support properties
we are assuming of $\mathfrak{m}(\rho,s;\xi)$. Therefore, integrating by parts
one time with respect to $s$, and then using the symbol bounds \eqref{symbol_bounds} and
the $\xi$ support properties of  $\mathfrak{m}(\rho,s;\xi)$, we easily have \eqref{I_est}
in the case $k=1$.\\

We now turn to the case $k=3$, which is the main work in establishing estimate  \eqref{I_est}.
Here the total phase always has a stationary point, so we need to use
stationary phase techniques to evaluate the integral. For the most part this turns out
to be standard, although a bit of care is needed
to deal with the temporal boundary $s=\rho$. We remark here that it is likely  the entire estimate can be done in
a very general context by considering Gaussian integrals on half spaces, but we will avoid generalities
of this form, and simply work directly with the specific form of our phase when the stationary point
is sufficiently close to the  temporal boundary.
As a preliminary reduction, we only consider the case  of the integral $I^+$ restricted to the branch of
 $\mathfrak{m}(\rho,s;\xi)$ where $\xi>0$. Other combinations are similar and left to the reader.\\

Under the restriction just imposed, we have the total phase:
\begin{equation}
    \Phi(\rho,y;s,\xi) \ = \ \psi(\rho,s;\xi) + 3s + y\xi \ . \notag
\end{equation}
First note that on the range where $s\sim\xi\sim\rho$, and $|y|\lesssim 1$,
this phase obeys uniform derivative bounds of the form:
\begin{equation}
    |\partial_s^i\partial_\xi^j \Phi| \ \lesssim \ (s+|\xi|)^{1-i-j} \ . \label{phase_bounds}
\end{equation}
The gradient of this phase where $\xi>0$ has components:
\begin{align}
    \partial_s \Phi \ &= \ -(\frac{1}{s^2}\xi^2 +1)^\frac{1}{2} + 3 \ ,
    &\partial_\xi \Phi \ &= \ \int_{{s}/{\xi}}^{{\rho}/{\xi}}\ \frac{1}{\sqrt{1+\zeta^2}}
    \frac{d\zeta}{\zeta}  + y \ . \label{phase_grad}
\end{align}
Clearly, for fixed values of $\rho$ and $y$ with $|y|\leqslant 1$, there is a unique stationary point $(s_0,\xi_0)$
in the range of our integrations. At this (or any) point, the phase Hessian  has components:
\begin{align}
    \partial_s^2 \Phi \ &= \ \frac{\xi^2}{s^2}\frac{1}{\sqrt{\xi^2 + s^2}} \ , \notag\\
    \partial_s\partial_\xi \Phi \ &= \ -\frac{\xi}{s}\frac{1}{\sqrt{\xi^2 + s^2}} \ , \notag\\
    \partial_\xi^2 \Phi \ &= \ \frac{1}{\sqrt{\xi^2+s^2}}-\frac{1}{\sqrt{\xi^2 + \rho^2}} \ . \notag
\end{align}
The phase determinant and trace are thus computed to be:
\begin{align}
    D \ &= \ \partial_s^2 \Phi\cdot \partial_\xi^2 \Phi - \big(\partial_s\partial_\xi \Phi\big)^2
    \ = \ -\frac{\xi^2}{s^2}\frac{1}{\sqrt{\xi^2 + s^2}}\frac{1}{\sqrt{\xi^2 + s^2}} \ , \notag\\
    T \ &= \  \partial_s^2 \Phi + \partial_\xi^2 \Phi \ = \
    \frac{\sqrt{\xi^2 + s^2}}{s^2} -\frac{1}{\sqrt{\xi^2 + \rho^2}} \ . \notag
\end{align}
Therefore, in the region where $s\sim\rho\sim\xi$, which is where our integrand is restricted, we have
both $|D|\sim \rho^{-2}$ and $|T|\lesssim \rho^{-1}$. Therefore, we easily have
that the eigenvalues of the phase Hessian in this region are:
\begin{align}
    \lambda_1(s_0,\xi_0) \ &\sim \ \rho^{-1} \ ,
    &\lambda_2(s_0,\xi_0) \ &\sim \ -\rho^{-1} \ . \label{phase_eigen}
\end{align}
This is the correct balance of factors needed to prove \eqref{I_est}. Again, the main
non-standard issue is to deal with the integration boundary where $s=\rho$. To handle this, we first decompose the integral
\eqref{I_form} into two bulk pieces $I^+_3  =  I^{near} + I^{far}$ where:
\begin{align}
     I^{near}  \ &= \  \rho^{-1} \int_{0}^\rho\int_\mathbb{R}\
    e^{ i\Phi}  \mathfrak{m}\
     \chi\big(\rho^{-\frac{1}{2}}(\rho-s)\big)\chi_1\big(\frac{s}{\rho}\big) \ d\xi ds \ , \label{close_int}\\
     I^{far} \ &= \ \rho^{-1} \int_{0}^\rho\int_\mathbb{R}\
    e^{ i\Phi}  \mathfrak{m}\
     [1-\chi\big(\rho^{-\frac{1}{2}}(\rho-s)\big)]\chi_1\big(\frac{s}{\rho}\big) \ d\xi ds \ . \label{far_int}
\end{align}
Here $\chi$ is a $C_0^\infty$ function with $\chi\equiv 1$ in a neighborhood of the origin.\\

The estimate \eqref{I_est} for $I^{far}$ is completely standard.
From the uniform bounds \eqref{phase_bounds}, we may assume without
loss of generality\footnote{This may be done by considering a
sufficiently small  $O(\rho)$ ball about the stationary point
$(s_0\xi_0)$. In the compliment of this region the phase $\Phi$ is
non-stationary, so a simple integration by parts argument suffices
to produce \eqref{I_est}.} that $\mathfrak{m}(\rho,s;\xi)$ is
supported in a region where there is a global change of variables
$(s,\xi)=F(\tau,\eta)$ with uniform derivative bounds:
\begin{equation}
        |\partial_\tau^i\partial_\eta^j F| \ \leqslant \ C_{i,j}(1+ |\tau| + |\eta|)^{1-i-j} \ , \notag
\end{equation}
and such that (notice that from \eqref{phase_eigen} the critical point $(s_0,\xi_0)$ is hyperbolic):
\begin{equation}
    \Phi\circ F (\tau,\eta)\ = \ \rho^{-1} (\tau^2 - \eta^2)+\Phi(s_0,\xi_0) \ . \notag
\end{equation}
In fact, we have by Taylors' formula
with integral remainder
\begin{equation}
\Phi(s,\xi)=\Phi(s_0,\xi_0)+(s-s_0)^2 \Phi_{11}(s,\xi)+(s-s_0)(\xi-\xi_0)\Phi_{12}(s,\xi)+(\xi-\xi_0)^2\Phi_{22}(s,\xi),
\notag
\end{equation}
where
\begin{equation}
\Phi_{ij}(s,\xi)=\int_0^1 (1-t) )(\partial_i\partial_j\Phi)
\big( (s_0,\xi_0)+t(s-s_0,\xi-\xi_0)\big)\, dt.\notag
\end{equation}
 Since we have uniform bounds (independent of $\rho$ and $y$ with the above restrictions)
\begin{equation}
|\partial_s^k \partial_\xi^l \Phi_{ij}(s,\xi)|\lesssim \rho^{-1} (1+s+|\xi|)^{-k+l},\qquad \Phi_{11}(s,\xi)\sim \rho^{-1} ,\quad D(s,\xi)\sim \rho^{-2}\notag
\end{equation}
the uniform bound for the change of variables obtained by completing the square
as in the usual proof of Morse Lemma follows.
Therefore, after the further change of variables $\underline{\tau}=\rho^{-\frac{1}{2}}\tau$ and
$\underline{\eta}=\rho^{-\frac{1}{2}}\eta$, we may write the integral \eqref{far_int}
as follows:
\begin{equation}
    I^{far} \ = \ e^{i\Phi(s_0,\xi_0)}\int\!\!\int \
    e^{ i(\underline{\tau}^2 - \underline{\eta}^2)}  \mathfrak{n}(\underline{\tau},\underline{\eta}) \ d\underline{\tau}
    d\underline{\eta} \ , \notag
\end{equation}
where $\mathfrak{n}$ is some new symbol obeying the bounds $|\partial_{\underline{\tau}}^i
\partial_{\underline{\eta}}^j\mathfrak{n}|  \leqslant  C_{i,j}$, and compactly
supported in some large ball (of radius $\sim\rho^{\frac{1}{2}}$, although the size does not matter).
The bound \eqref{I_est} for integrals of this form is a simple matter of integration by parts
away from where $|\underline{\tau}|\leqslant 1$ and $|\underline{\eta}|\leqslant 1$. \\

It remains to deal with the integral \eqref{close_int}. There are two cases here depending on
the size of the spatial variable $y$. In the easy case, where $\rho^{-\frac{1}{2}}\ll|y|$,
the phase $\Phi$ is non-stationary on the support of the integrand. This is easily confirmed
from the second of the gradient calculations \eqref{phase_grad}, which in the region where
$\rho-s\lesssim \rho^{\frac{1}{2}}$ may be written as:
\begin{equation}
    \partial_\xi \Phi \ = \ \frac{\rho-s}{\xi} h(\rho,s;\xi) + y \ \sim \ y \ , \notag
\end{equation}
as long as $\rho^{-\frac{1}{2}}\leqslant C^{-1}|y|$ for a sufficiently large constant $C$ which
only depends on how we cut out the integration region of $I^{near}$ to begin with. Here
$h$ is a function obeying the uniform derivative bounds:
\begin{equation}
    |\partial_s^i\partial^j_\xi h| \ \leqslant \ C_{i,j}\rho^{-i-j} \ . \notag
\end{equation}
Thus, under the assumption that $\rho^{-\frac{1}{2}}\ll|y|$ we have in the region where
$\rho-s\lesssim \rho^{\frac{1}{2}}$ and $\rho\sim\xi$ the symbol  bound:
\begin{equation}
    \Big|\big(\partial_\xi \frac{1}{\partial_\xi \Phi}\big)
    \mathfrak{m}\Big| \ \leqslant \ C\rho^{-\frac{1}{2}} \ . \notag
\end{equation}
The bound \eqref{I_est} easily follows from this and one integration by parts with respect to $\xi$.\\

Our final task here is to establish \eqref{I_est} for the integral
\eqref{close_int} under the additional assumption that $|y|\lesssim
\rho^{-\frac{1}{2}}$. In this case, the $O(\rho^{\frac{1}{2}})$
stationary region around the point $(s_0,\xi_0)$ contains the
boundary $s=\rho$, and a bit of care is needed to achieve the
desired result. The main difficulty is the following: for very small
values of $|y|$ the boundary phase $\Phi|_{s=\rho}$ can be quite
degenerate due to the hyperbolic nature of the critical point
$(s_0,\xi_0)$. In fact, if $y=0$, then $\partial_\xi
\Phi|_{s=\rho}=0$. This means that simply integrating  by parts with
respect to the characteristic directions of the eigenvalues from
\eqref{phase_eigen} can leave one with a singular boundary term
that does not oscillate enough to recover uniform boundedness. The
way to get around this is to carefully delineate a new stationary
region where one cannot control the phase, and then integrate by
parts on the compliment. Our first decomposition is to cut
$I^{near}=I^{near}_1 + I^{near}_2$ where:
\begin{align}
     I^{near}_1  \ &= \  \rho^{-1} \int_{0}^\rho\int_\mathbb{R}\
    \mathcal{I}^{near}\ \chi\big(C\rho^{-1}(\rho^{-1} + |y|)^{-1}(\rho-s)\big)
      \ d\xi ds \ , \notag\\
     I^{near}_2  \ &= \  \rho^{-1} \int_{0}^\rho\int_\mathbb{R}\  \mathcal{I}^{near}
    \ [1-\chi\big(C\rho^{-1}(\rho^{-1} + |y|)^{-1}(\rho-s)\big)]
      \ d\xi ds \ . \notag
\end{align}
Here $\mathcal{I}^{near}$ is the integrand of the original $I^{near}$, and $C$ is a flexible large  constant.
For $C$ sufficiently large, the second integral
$I^{near}_2$ above is estimated directly via integration by parts with respect to $\xi$. On the support
of the corresponding cutoff, a short calculation shows that one has access to the symbol bounds:
\begin{equation}
    \Big|\big(\partial_\xi \frac{1}{\partial_\xi \Phi}\big)^i
    \mathfrak{m}\Big| \ \leqslant \ C_{i}(1 + \rho-s)^{-i} \ . \notag
\end{equation}
The bound \eqref{I_est} for $I^{near}_2$ is therefore a result of two integrations by parts
with respect to the $\xi$ variable and then directly estimating the absolute value
of the resulting integral.\\

We now move on to estimating the integral $ I^{near}_1$ defined in the last paragraph. This
may be further decomposed as $I^{near}_1=I^{near}_{1,1} + I^{near}_{1,2}$ where:
\begin{align}
     I^{near}_{1,1}  \ &= \  \rho^{-1} \int_{0}^\rho\int_\mathbb{R}\
    \mathcal{I}^{near}_1\ \chi\big(C^{-2} (\rho^{-1} + |y|) (\xi-\sqrt{8}\rho)
    \big)
      \ d\xi ds \ , \notag\\
     I^{near}_{1,2}  \ &= \  \rho^{-1} \int_{0}^\rho\int_\mathbb{R}\  \mathcal{I}^{near}_1
    \ [1-\chi\big(C^{-2} (\rho^{-1} + |y|) (\xi-\sqrt{8}\rho)
    \big)]
      \ d\xi ds \ . \notag
\end{align}
Here $C$ is the same large constant used above, which again is assumed to
be sufficiently large.
This time $\mathcal{I}^{near}_1$ is the integrand of the original $I^{near}_1$.
Notice that the support of the integrand in $I^{near}_{1,1}$ is a rectangle
of dimensions $\sim \rho^{-1}(\rho^{-1} + |y|)^{-1}\times(\rho^{-1} + |y|)$.
Therefore, the bound \eqref{I_est} for the integral $I^{near}_{1,1}$ follows
from direct absolute integration.\\

Our final task is to estimate the integral $I^{near}_{1,2}$ via
another integration by parts argument. Notice that it suffices to
only consider the case where $\rho^{-1}\leqslant |y|$, as in the
complimentary case the integrand of $I^{near}_{1,2}$ is supported on
a slab where $\rho-s\lesssim 1$, and the bound \eqref{I_est} in this
case again follows from direct absolute integration. This time, we
first integrate by parts once with respect to the $s$ variable which
gives us:
\begin{multline}
    I^{near}_{1,2} \ = \ -\rho^{-1} \int_{0}^\rho\int_\mathbb{R}\  e^{i\Phi}
    \partial_s\Big[
    \mathfrak{n}(\rho,s;\xi)\chi\big(C(\rho|y|)^{-1}(\rho-s)\big)\Big] \ d\xi ds\\
    + \rho^{-1} e^{3i\rho}\ \int_\mathbb{R}\  e^{iy\xi}
    \mathfrak{n}(\rho,\rho;\xi)
    \ d\xi \ , \label{last_I_12}
\end{multline}
where $\mathfrak{n}$ is the truncated symbol:
\begin{equation}
    \mathfrak{n}(\rho,s;\xi) \ = \ \frac{1}{i\partial_s \Phi(\rho,s;\xi)}\mathfrak{m}(\rho,s;\xi)
     [1-\chi\big( C^{-2} |y| (\xi-\sqrt{8}\rho)
    \big)] \ . \notag
\end{equation}
A few quick calculations show that for this symbol, we have the following derivative bounds
as long as we choose $C$ sufficiently large:
\begin{align}
    \Big|\partial_s
    \mathfrak{n}(\rho,s;\xi)\Big| \
    &\lesssim \ \frac{\rho}{(\xi - \sqrt{8}\rho)^2}
    \Big[1-\chi\big( C^{-2} |y|(\xi-\sqrt{8}\rho)\big)\Big] \ , \label{N_bound1}\\
    \Big|\partial_\xi \mathfrak{n}(\rho,\rho;\xi)\Big| \ &\lesssim \
    \frac{\rho}{(\xi - \sqrt{8}\rho)^2}\Big[1-\chi\big(C^{\prime-2}|y|(\xi-\sqrt{8}\rho)\big)\Big]
    \ , \label{N_bound2}
\end{align}
when $\rho-s\lesssim \rho |y|$ and $|y|\lesssim \rho^{-1/2}$.
Moreover, since $\rho-s\lesssim \rho |y|$:
\begin{equation}
    \Big|\big(\partial_\xi \frac{1}{\partial_\xi \Phi}\big)^i\
    \big(\mathfrak{n}\big)
    \cdot(\rho|y|)^{-1}
    \chi'\big(C(\rho|y|)^{-1}(\rho-s)\big)\Big| \
    \leqslant \ C_{i}(1 + \rho-s)^{-i} \ . \label{last_N_bound}
\end{equation}
Therefore, either by directly integrating, or if necessary
integrating by parts with respect to $\xi$ one or two times, one can easily see that
we have the bound \eqref{I_est} for all terms on the RHS of \eqref{last_I_12}. For the
first term on the RHS of \eqref{last_I_12} we use both \eqref{N_bound1} and \eqref{last_N_bound}
which allows us to estimate:
\begin{align}
    &\Big| \rho^{-1} \int_{0}^\rho\int_\mathbb{R}\  e^{i\Phi}
    \partial_s\Big[
    \mathfrak{n}(\rho,s;\xi)\chi\big(C(\rho|y|)^{-1}(\rho-s)\big)\Big] \ d\xi ds\Big|\ , \notag\\
     \lesssim \ & \int_{0}^\rho\int_{\mathbb{R}}\
     \frac{1}{(\xi - \sqrt{8}\rho)^2}\Big[1-\chi\big( C^{-2}|y| (\xi-\sqrt{8}\rho)\big)\Big] \
     \chi\big(C(\rho|y|)^{-1}(\rho-s)\big)\
     \  d\xi ds\notag\\
     & \ \ \ \ \ \ \ \ \ \ \ \ \ \ \ \
     + \rho^{-1} \int_{0}^\rho\int_{|\xi|\lesssim \rho}\ (1 + \rho-s)^{-2}\ d\xi ds  \ , \notag\\
     \lesssim \ &\rho|y|^2 + 1
     \ . \notag
\end{align}
Using the condition that $|y|\lesssim\rho^{-\frac{1}{2}}$, this last formula yields the
desired bound. Finally, the last term on the RHS of \eqref{last_I_12} is estimated
in a similar fashion via \eqref{N_bound2} and one integration by parts with respect to $\xi$:
\begin{align}
    \Big| \rho^{-1}\ \int_\mathbb{R}\  e^{iy\xi}
    \mathfrak{n}(\rho,\rho;\xi)
    \ d\xi \Big| \ &\lesssim \ \int_{\mathbb{R}}\
    \frac{|y|^{-1}}{(\xi - \sqrt{8}\rho)^2}\Big[1-\chi\big(C^{-2}|y|(\xi-\sqrt{8}\rho)\big)\Big]\ d\xi\ , \notag\\
    &\lesssim \ 1 \ . \notag
\end{align}
This completes our demonstration of estimate \eqref{I_est}.
\end{proof}

\begin{proof}[Proof of the estimates \eqref{med_K_ests1}--\eqref{med_K_ests2}]
We begin with a preliminary reduction, which is that it suffices to
consider the first set of estimates \eqref{med_K_ests1}. More
specifically, we claim that the first estimate in
\eqref{med_K_ests1} implies the estimate in
\eqref{med_K_ests2}. To see this, simply note that the frequency of
$\beta(s\rho)$ in both of the integrals
\eqref{K1_def}--\eqref{K2_def} is restricted to the region where
$|\xi|\sim \rho$ owing to the presence of the cutoff
$\chi_1\big(\frac{s}{\rho}\big)$ and the fact that $\beta(s\rho)$
is assumed to have frequency support where $|\xi|\sim s$. Therefore,
the bulk of the frequency of $\chi K_i$ is contained where
$|\xi|\sim\rho$ modulo a piece with very fast decay in $|\xi|$ as
$\rho\ll |\xi|\to\infty$ due to the Schwartz tails of
$\widehat{\chi}(\xi)$.\\

 We now prove   the first estimate in
\eqref{med_K_ests1}. Because we have $|D_y^k \chi|\leqslant C_k$ and
$\partial_\rho \chi =0$, it suffices to prove that $|\chi
\partial_\rho^k D_y^l K_i|\leqslant C_{k,l}$. This follows
immediately from the estimate \eqref{I_est} if we can show that
$\partial_\rho^k D_y^l K_i$ has a symbol $\mathfrak{m}_{k,l}$
obeying \eqref{symbol_bounds}. By decomposing
\eqref{K1_def}--\eqref{K2_def} into exponentials, it suffices to
consider the integrals:
\begin{align}
    I_1^\pm \ &= \ \rho^{-1}\int_{\rho_0}^\rho\int_\mathbb{R}
    e^{\pm i\psi(\rho,s;\xi)} e^{iy\xi}e^{is} \frac{\rho s^{-1}}{\sqrt{s^{-2}\xi^2 + 1}}
    \widetilde{\beta}(s^{-1}\xi)\chi_1\big(\frac{s}{\rho}\big) \ d\xi ds \ , \notag\\
    I^\pm_{3} \ &= \ \rho^{-1}\int_{\rho_0}^\rho\int_\mathbb{R}
    e^{\pm i\psi(\rho,s;\xi)} e^{iy\xi}e^{3is} \frac{\rho s^{-1} }{\sqrt{s^{-2}\xi^2 + 1}}
    \widetilde{\beta}(s^{-1}\xi)\chi_1\big(\frac{s}{\rho}\big) \ d\xi ds \ . \notag
\end{align}
Here $\widetilde{\beta}$ is the Fourier transform of $\beta(z)$.
Notice that $\frac{\rho s^{-1}}{\sqrt{s^{-2}\xi^2 + 1}}
\widetilde{\beta}(s^{-1}\xi)\chi_1\big(\frac{s}{\rho}\big)$ obeys the symbol bound \eqref{symbol_bounds},
and the space of all functions of course forms an algebra. Thus, we only need to show that the $D_y$ and $\partial_\rho$
derivatives of the phase $\pm \psi + y\xi$ obeys \eqref{symbol_bounds} on the region where $s\sim\rho\sim\xi$
as well. This easily follows from the definition of \eqref{phase}.\\

The proof of the second bound on  \eqref{med_K_ests1} is very
similar. The only difference here is that one needs to first compute
the expression $(\Box_\mathcal{H} +1)K_i$ and compare it to the RHS
of \eqref{0_res_eq}--\eqref{sqrt8_res_eq}, and then show that the
resulting difference is an integral with a symbol $\mathfrak{n}$
such that the quantity $\mathfrak{m}=\rho^2\mathfrak{n}$ obeys the
bounds \eqref{symbol_bounds} (further derivatives of the error are
handled similarly). We only deal with the expression for $K_1$,
because the corresponding calculations for $K_3$ are completely
analogous. A short calculation reveals that:
\begin{multline}
    (\Box_\mathcal{H}+1)K_1 - e^{i\rho}\beta(\rho y) \ ,
    = \
    \int_{1}^\rho (\Box_\mathcal{H}+1)\Big[U(\rho,s;\partial_y)
    \beta\, \chi_1\big(\frac{s}{\rho}\big) \Big]\ ds \  \\
    = \ -\frac{1}{\rho}\int_{1}^\rho \int_\mathbb{R}\cos\big(\psi(\rho,s;\xi)\big)e^{i y\xi}e^{is}\
      \frac{\xi^2}{\rho^2}\frac{1}{\sqrt{\xi^2 + \rho^2}}
      \frac{\rho s^{-1}}{\sqrt{s^{-2}\xi^2 + 1}}
    \widetilde{\beta}(s^{-1}\xi)\chi_1\big(\frac{s}{\rho}\big) \ d\xi ds \\
    -\frac{2}{\rho}\int_{1}^\rho \int_\mathbb{R}\cos\big(\psi(\rho,s;\xi)\big)e^{i y\xi}e^{is}\
      \frac{1}{\sqrt{\rho^{-2}\xi^2 + 1}}
      \frac{\rho s^{-1}}{\sqrt{s^{-2}\xi^2 + 1}}
    \widetilde{\beta}(s^{-1}\xi)\chi_1'\big(\frac{s}{\rho}\big)
    \left(\frac{s}{\rho^2}\right) \ d\xi ds\notag\\
    + \frac{1}{\rho}\int_{1}^\rho\int_\mathbb{R} \sin\big(\psi(\rho,s;\xi)\big)e^{i y\xi}e^{is}\
      \frac{\rho s^{-1}}{\sqrt{s^{-2}\xi^2 + 1}}
    \widetilde{\beta}(s^{-1}\xi)\partial_\rho^2\Big[\chi_1\big(\frac{s}{\rho}\big)\Big] \ d\xi ds \ . \notag
\end{multline}
The reader may easily verify that the symbol in each integral on the right hand side
of this last equation obeys the bounds \eqref{symbol_bounds} with an extra factor
of $\rho^{-1}$ to spare. Therefore, after an application of \eqref{I_est} we have the
desired bounds.
\end{proof}



\section*{Acknowledgments}
This work began as a joint project with I. Rodnianski and J. Sterbenz. We thank both for their major contributions.
H.L \ is partially supported by NSF grant DMS--1237212.
A.S.\ is partially supported by NSF grant DMS-- 1201394.


\end{document}